\documentclass{article}
\usepackage{amsfonts,amssymb,amsmath, mathtools,amsthm}
\usepackage{xcolor, color}
\usepackage{graphicx}
\usepackage{multicol}
\usepackage{authblk}
\usepackage{url}
\usepackage{faktor}
\usepackage{ulem}
\usepackage{tikz}
\usetikzlibrary{calc,shapes.geometric,positioning}
\usepackage{pdfpages}
\usepackage{soul}
\usepackage{epsfig}
\usepackage{epic}
\usepackage{eepic}
\usepackage{url}
\usepackage{longtable}
\usepackage{mathrsfs}
\usepackage{multirow}
\usepackage{bigstrut}
\usepackage{setspace}
\usepackage{array}
\usepackage{xspace}
\usepackage{siunitx}
\usepackage{booktabs} 
\usepackage{hyperref} 
\usepackage[noabbrev]{cleveref}

\tikzset{
blacknode/.style={circle,fill=black,inner sep=2.5pt},
whitenode/.style={circle,draw=black,fill=white,inner sep=2.5pt},
edge/.style={line width=0.9pt},
strand/.style={line width=1.1pt}
dimedge/.style={line width=1.7pt}
}

\def\multiset#1#2{\ensuremath{\left(\kern-.3em\left(\genfrac{}{}{0pt}{}{#1}{#2}\right)\kern-.3em\right)}}

\newtheorem{theorem}{Theorem}[section]
     \newtheorem{lemma}[theorem]{Lemma}
     \newtheorem{proposition}[theorem]{Proposition}

     \newtheorem{definition}[theorem]{Definition}
     \newtheorem{example}[theorem]{Example}

     \newtheorem{remark}[theorem]{Remark}

\newcommand{\Sumblank}[2]{\displaystyle{\sum_{#1}^{#2}}}
\newcommand{\SumBlank}[1]{\displaystyle{\sum_{#1}}}

\newcommand{\ProdBlank}[1]{\displaystyle{\prod_{#1}}}

\newcommand{\bi}{\begin{itemize}}
\newcommand{\ei}{\end{itemize}}
\newcommand{\bc}{\begin{center}}
\newcommand{\ec}{\end{center}}
\newcommand{\be}{\begin{enumerate}}
\newcommand{\ee}{\end{enumerate}}

\newcommand{\C}{\mathbb{C}}

\newcommand{\ZZ}{\mathbb{Z}}

\newcommand{\RR}{\mathbb{R}}
\newcommand{\CC}{\mathbb{C}}

\newcommand{\grass}{\mathrm{Gr}(k,n)}
\newcommand{\tnngrass}{\mathrm{Gr}(k,n)_{\ge 0}}
\newcommand{\affgrass}{\widehat{\mathrm{Gr}}(k,n)}

\newcommand{\coring}{\C[\widehat{\mathrm{Gr}}(k,n)]}

\newcommand{\poring}{\CC[\widehat{\Pi_{\mathcal{P}}^\circ}]}

\newcommand{\tmax}[1]{\mathrm{tmax}\left\{ #1 \right\}}
\newcommand{\TMAX}[1]{\mathrm{Tmax}\left\{ #1 \right\}}
\newcommand{\Tmax}{\mathrm{tmax}}
\newcommand{\TTMAX}{\mathrm{Tmax}}
\newcommand{\tmin}[1]{\mathrm{tmin}\left\{ #1 \right\}}
\newcommand{\TMIN}[1]{\mathrm{Tmin}\left\{ #1 \right\}}
\newcommand{\Tmin}{\mathrm{tmin}}
\newcommand{\TTMIN}{\mathrm{Tmin}}

\newcommand{\cev}[1]{\reflectbox{\ensuremath{\vec{\reflectbox{\ensuremath{#1}}}}}}

\newcommand{\gneck}{\vec{\mathcal{I}}}
\newcommand{\rgneck}{\cev{\mathcal{I}}}
\newcommand{\gneckset}[1]{\vec{I}_{#1}}
\newcommand{\rgneckset}[1]{\cev{I}_{#1}}

\DeclareMathOperator{\wt}{\mathrm{wt}}

\title{Extremal Matchings and Height Functions}

\author{Nickolas Anderson\thanks{School of Mathematics, University of Minnesota~ \texttt{and04339@umn.edu}} ~and Gregg Musiker\thanks{School of Mathematics, University of Minnesota~ \texttt{musiker@umn.edu}}}
\date{June 9, 2026}

\begin{document}

\maketitle

\abstract{
This paper studies a lattice structure for almost perfect matchings on certain planar, bipartite (plabic) graphs embedded in a disk. Postnikov's boundary measurement map, and subsequent related work, yielded that plabic graphs parameterize \textit{positroid cells} within the totally nonnegative Grassmannian with the map itself given in terms of almost perfect matchings with fixed boundary condition. For finite planar bipartite graphs, Propp introduced a distributive lattice structure on their set of perfect matchings. Subsequently Muller--Speyer, provided this distributive lattice structure on the aforementioned almost perfect matchings with fixed boundary condition.  Their work also identified the \textit{extremal matchings} of this lattice for boundary conditions that coincide with face labels of the plabic graph given by the positroid structure. We extend this by giving an explicit construction of extremal matchings in terms of \textit{height functions} and show that all possible boundary conditions of an almost perfect matching can be obtained within this construction.}

\tableofcontents

\section{Introduction}

As constructed by Postnikov \cite{Post}, the \textit{totally nonnegative Grassmannian} admits a stratification by \textit{positroid cells} $\{\Pi_\mathcal{P}\}$. In the same paper, Postnikov introduced certain classes of \textit{(reduced) planar bicolored graphs} (\textit{plabic graphs}) $G$ which parameterize these positroid cells via the \textit{boundary measurement map}. Informally, the boundary measurement map takes the input of a positive edge weighting on $G$ and returns a collection of homogeneous \textit{Pl\"ucker coordinates} through considering the sets of \textit{(almost perfect) matchings} on $G$. 

Postnikov \cite{Post} defined \textit{positroids} $\mathcal{P}$ as matroids with totally nonnegative maximal representations. One observation of Postnikov was that the set of \textit{loop erased walks} on a planar network $N(G)$, a prototype for almost perfect matchings on $G$, formed a positroid $\mathcal{P}$ which coincides with the positroid indexing the image of $N(G)$ under the boundary measurement map. Follow up work of Talaska \cite{Talaska} extended this to allow loops in such planar network utilizing \textit{conservative flows} and \textit{perfect orientations} on $N(G)$. Postnikov, Speyer, and Williams \cite{PSW} noted that perfect orientations on $G$ (with source set $I$) were in bijection with matchings on $G$ with \textit{boundary condition} $I$. 
An important consequence of characterizing positroids in this way is that $\mathcal{P}$ is the collection of all achievable boundary conditions for matchings on the associated plabic graph $G$.

Due to Thurston \cite{Thurston}, one means of studying matchings on a finite, bipartite planar graph $H$ is through the lens of \textit{height functions}. Although the context of Thurston's height functions arose in the context of lozenge tilings, Propp \cite{Propp}, whose pre-print first appeared in 1993, linked height functions to \textit{(perfect)} matchings on $H$. In particular, Propp's formulation of height functions assigned real values to the faces of $H$ such that differences in heights between adjacent faces represented relative probabilistic measures, and under a certain cyclic orientation about vertices of $H$, a positive difference in heights corresponded to a matched edge. For further details, see \Cref{sec:heights}. 

In the same work of Propp, he used height functions to establish a finite distributive lattice structure on the set of perfect matchings on finte, planar bipartite graphs $H$ with the partial given in terms of \textit{swivels} \cite{Propp}. More recent work of Muller and Speyer showed that the sets of almost perfect matchings on plabic graphs $G$ (in the sense of Postinikov) with fixed boundary conditions also form finite distributive lattices under the same partial order by swivels \cite[Appendix B]{MSTwist}. In Section 5 of their paper, Muller and Speyer also gave a construction for certain \textit{extremal matchings} in these lattices, however, their construction fails to yield extremal matchings for those with boundary conditions not appearing as face labels of $G$. In particular, a result of Oh, Postnikov, and Speyer \cite{OPS} yields that the number of faces of $G$ is at most $k(n-k)+1$ but the number of boundary conditions for matchings on $G$ is larger than this in general, approaching $\binom{n}{k}$ in the case of the top-dimensional positroid cell.

The current work addresses this gap by providing extremal matchings for all boundary conditions, regardless of whether or not it appears as a face label of $G$. We provide this by using a novel construction for extremal matchings on $G$ in terms of height functions. In particular, we define a special family of height functions called \textit{fundamental height functions} which result from summing over the face label entries of $G$ under a shifted cyclic order on $[n]$. Using the fundamental height functions as building blocks, we obtain new height functions called \textit{translated maximums}, $\TTMAX$, and \textit{translated minimums}, $\TTMIN$ (see \Cref{def:Tmax_Tmin}). As main results, we show that $\TTMAX$'s and $\TTMIN$'s produce the \textit{minimal matchings} and \textit{maximal matchings}, respectively (see \Cref{thm:tmax_tmin_are_extremal}), with their boundary conditions read in terms of \textit{Grassmann}- and \textit{reverse Grassmann necklaces} (see Theorems \ref{thm:main} and \ref{thm:main_restate}). As a consequence, we obtain that all boundary conditions of extremal matchings on $G$ can be obtained in this way (see \Cref{lem:all_matchings}). 

This paper is organized as follows. In \Cref{sec:positroid_combo}, we give background on the positroid stratification of the totally nonnegative Grassmannian and the combinatorial objects indexing these positroids, namely, Grassmann necklaces, reverse Grassmann necklaces, and \textit{decorated permutations}. In \Cref{sec:plabic_graphs}, we discuss the anatomy of (reduced) plabic graphs, face labeling conventions for $G$, and how the \textit{almost monotonically decreasing property} arises in face labels of $G$ (see \Cref{lem:common_face_indexing}). In \Cref{sec:matchings}, we discuss Postnikov's boundary measurement map and the lattice structure of matchings of Propp and Muller--Speyer. Additionally, we give the Muller--Speyer construction for extremal matchings on $G$. In \Cref{ssec:heights}, we define our family of fundamental height functions and develop our theory of height functions on $G$ by deducing many results about $\TTMAX$'s and $\TTMIN$'s. Of particular note, we show that $\TTMAX$ constructs minimal matchings and $\TTMAX$ constructs maximal matchings in \Cref{thm:tmax_tmin_are_extremal}. In \Cref{ssec:extreme_match}, we state and prove our main result, namely, that the extremal matchings constructed from $\TTMAX$ and $\TTMIN$ have boundary conditions read from the data of Grassmann- and reverse Grassmann necklaces (see Theorems \ref{thm:main} and \ref{thm:main_restate}). We close this paper with \Cref{appsec:Applications} which details the motivation from cluster algebras to study almost perfect matchings.

\textbf{Acknowledgements:} A portion of this work appeared in the first author's M.S. Thesis \cite{NAThesis}. The authors would like to thank B. Brubaker, P. Pylyavskyy, and D. Speyer for their meaningful discussions and questions around this work.

\section{The Combinatorics of Positroid Cells}
\label{sec:positroid_combo}

We study the positroid stratification of the totally nonnegative Grassmannian as in \cite{Post,SOh,KLS}. We first define the \textbf{Grassmannian}, $\grass$, to be the space of $k$-planes in $n$-dimensional space. Without loss of generality, we will let our ambient $n$-dimensional space be $\RR^n$. 

We encode a $k$-subspace of $\RR^n$ as a rank $k$ matrix $A$ where $\mathrm{rowspan}(A)$ properly defines this subspace. Via the \textbf{Pl\"ucker embedding} of $\grass$ into $\mathbb{P}(\wedge^k\RR^n)$, the projectivization of the $k$th exterior power of $\RR^n$, a point $A \in \grass$ is a full-rank $k \times n$ matrix with entries in $\RR$. Thus, $\grass$ is a projective variety and taking the affine cone over $\grass$, denoted $\widehat{\mathrm{Gr}}(k,n)$, gives a resulting affine variety. The homogeneous equations defining the projective variety structure of $\grass$ are in terms of $k \times k$ minors $\Delta_I = \begin{vmatrix}
    v_{i_1} & v_{i_2} & \cdots & v_{i_k}
\end{vmatrix}$, referred to as \textbf{Pl\"ucker coordinates}, where $I = \{i_1,i_2,\ldots, i_k\}$ indexes column vectors $v_{i_j}$ of a point $A \in \grass$. Thus Pl\"ucker coordinates generate the homogeneous coordinate ring of the affine cone over the Grassmannian, denoted by $\coring$. Of particular note, the Pl\"ucker embedding gives the following quadratic \textit{Pl\"ucker relation} \[\Delta_{J,a,c}\Delta_{J,b,d} = \Delta_{J,a,b}\Delta_{J,c,d} + \Delta_{J,a,d}\Delta_{J,b,c}\] where $J \in \dbinom{[n]}{k-2}$ and $a,b,c,d \in [n]$ are distinct and not contained in $J$. 

Scott showed that $\coring$ has a cluster algebra structure in \cite{JS} where the cluster variables are Pl\"ucker coordinates and the exchange relations induced by mutation are subject to Pl\"ucker relations. In particular, the quadratic Pl\"ucker relations are subtraction-free expressions. 

The \textbf{totally nonnegative Grassmannian}, $\tnngrass$, is the subspace of $\grass$ characterized as $\tnngrass = \left\{A \in \grass \bigg| \Delta_I(A) \ge 0 \text{ for all } I \in \dbinom{[n]}{k}\right\}$. For any $A \in \grass$, we can define its \textit{column matroid} $\mathcal{C}(A):= \left\{ J \in \dbinom{[n]}{k} \, \bigg| \, \Delta_J(A) \ne 0\right\}$ where $J$ indexes column vectors of $A$ assembling a basis for $\mathbb{R}^k$. When $X \in \tnngrass$, the corresponding column matroid $\mathcal{C}(X)$ belongs to the following special class of matroids.

\begin{definition}[\cite{Post}]
\label{def:positroid}
A \textbf{positroid}, $\mathcal{P}$, is a representable matroid whose (representation's) maximal minors are totally nonnegative.
\end{definition}

One may quickly observe that many distinct $X,Y \in \tnngrass$ give rise to the same positroids, i.e. $\mathcal{P}(X) = \mathcal{P}(Y)$, meaning that $X$ and $Y$ share the same set of vanishing Pl\"ucker coordinates. This gives rise to our main object of study. 

\begin{definition}[{\cite[Definition 3.2]{Post}}]
Fixing $X \in \tnngrass$, the \textbf{positroid cell}, $\Pi(X)$, is the collection \[\Pi(X) = \{ Y \in \tnngrass \, | \, \mathcal{P}(Y) = \mathcal{P}(X)\}.\]
\end{definition}

In particular, any positroid cell $\Pi(X)$ is determined by the collection nonvanishing minors of $X$. For this reason, we often omit $X$ in this notation, writing $\Pi_\mathcal{P}$ to represent a generic positroid cell which is indexed by its column positroid $\mathcal{P}$ of nonvanishing Pl\"ucker coordinates. Postnikov showed that $\tnngrass$ is stratified by positroid cells $\Pi_\mathcal{P}$ \cite{Post}. Similar to the case of $\coring$, one can define the \textit{open positroid variety}, $\Pi_{\mathcal{P}}^\circ$, and the homogeneous coordinate ring (of its affine cone), $\poring$, which is generated by nonvanishing Pl\"ucker coordinates $\Delta_I$. Then also, $\poring$ has a cluster algebra structure due to \cite{GL} and the cluster structure of $\coring$ arises as a special case when $\mathcal{P} = \dbinom{[n]}{k}$. We begin our study of $\Pi_\mathcal{P}$ via two combinatorial objects, namely, Grassmann necklaces and decorated permutations of $[n]$.

\begin{definition}[{\cite[Definition 16.1]{Post}}]
\label{def:grassman_necklace}
A \textbf{Grassmann Necklace} is a sequence $\gneck = (\gneckset{1}, \ldots,\gneckset{n})$ of sets $\gneckset{i} \in \dbinom{[n]}{k}$ satisfying the recurrence \[\gneckset{i+1} = \begin{cases}
 (\gneckset{i}\setminus \{i\}) \cup \{i_j\} & \text{if  }i \in \gneckset{i} \\
  \gneckset{i} & \text{if  } i \not\in \gneckset{i}
\end{cases}\] for all $1 \le i \le n$ with some $i_j \in [n]$.

Similarly, a \textbf{reverse Grassmann Necklace} is a sequence $\rgneck = (\rgneckset{1}, \ldots, \rgneckset{n})$ of sets $\rgneckset{i} \in \dbinom{[n]}{k}$ satisfying the recurrence
\[\rgneckset{i-1} =  \begin{cases}
(\rgneckset{i}\setminus\{i\}) \cup \{i_j\} & \text{if  } i \in \rgneckset{i} \\
\rgneckset{i} & \text{if  } i \not\in \rgneckset{i}
\end{cases}\] for all $1 \le i \le n$ with some $i_j \in [n]$.
\end{definition}

In addition, we define a partial ordering on $k$-subsets of $[n]$ as follows. 

\begin{definition}[c.f. {\cite[Section 16]{Post}}, {\cite[Section 2]{OPS}}]
\label{def:a_order}
For $a \in [n]$, we define the \textbf{$\mathbf{a}$-ordering}, $\le_a$, modulo $n$, as the total ordering on $[n]$ by \[a <_a a+1 <_a \cdots <_a a-1.\] For $I, J \in \dbinom{[n]}{k}$, we extend $a$-ordering to a partial ordering where $I \le_a J$ means $i_\ell \le_a j_\ell$ for all $1 \le \ell \le k$. 
\end{definition}

Let $\mathcal{P}$ be any column positroid. Then $\mathcal{P}$ has a unique $a$-minimal element $I_a$ for any $a \in [n]$. Then we have the correspondence $I_a = \gneckset{a}$ for all $a \in [n]$ {\cite[Lemma 16.3]{Post}}. Moreover, for any $X \in \Pi_\mathcal{P}$, the columns of $X$ indexed by $I_a$ form a basis for $\RR^k$, and so we can equivalently refer to $\gneck$ as the sequence of $a$-minimal bases. Symmetrically, $\mathcal{P}$ has a unique $(a+1)$-maximal element $J_a$ for any $a \in [n]$. Then $J_a = \rgneckset{a}$ for all $a \in [n]$ and we refer to $\rgneck$ as the sequence of $(a+1)$-maximal bases. The following theorem gives parameterizations of $\mathcal{P}$ by $\gneck$ and $\rgneck$.

\begin{theorem}[{\cite[Theorem 6]{SOh}}]
\label{thm:gneck_positroid_biject}
For a Grassmann necklace $\gneck$, define \[\mathcal{P} = \left\{J \in \dbinom{[n]}{k} \bigg| \gneckset{a} \le_a J \text{ for all } a \in [n]\right \}.\] Then $\mathcal{P}$ is a positroid, and any positroid is obtained from a Grassmann necklace in this way. 
\end{theorem}

Analgously, we obtain a bijection of positroids with reverse Grassmann necklaces when taking $J \le_{a+1} \rgneckset{a}$ for all $a \in [n]$ in \Cref{thm:gneck_positroid_biject}. Consequently, any $\Pi_\mathcal{P}$ is indexed by $\gneck$ and $\rgneck$ in this way.

\begin{remark} \label{rem:GNwrittenorder} When writing the $k$-subsets making up a Grassmann or reverse Grassmann necklace, we either write these in the usual integer order or utilize the appropriate $a$-ordering and $(a+1)$-ordering, respectively. If we follow the latter convention, $\gneckset{i}$ will begin with an $i$ (if present) and then increase circularly. On the other hand, $\rgneckset{i}$ will begin with an $i$ (again if present) and then decrease circularly.
\end{remark}

\begin{example}
\label{ex:positroid_example_matrix}   
Consider the following matrix \[B = \begin{pmatrix}
    1 & 3 & 2 & 1 & 0 & -1 & 0 \\
    0 & 1 & 1 & 1 & 1 & 0 & 0  \\
    0 & -4 & -3 & -2 & 0 & 3 & 1
\end{pmatrix} \in \mathrm{Gr}(3,7)_{\ge 0}.\]
One can verify $\mathcal{P}(B) = \dbinom{[7]}{3} \setminus \{234,167\}$ as $\Delta_{234}(B) = \Delta_{167}(B) = 0$. Alternatively, one can obtain the Grassmann and reverse Grassmann necklaces of $\mathcal{P}(B)$ as 
$\gneck = (123,235,345,456,567,267,127)$ and $\rgneck = (157,127,123,134,345,456,567)$, where in particular $234 <_2 235$ (or $234 >_5 134$) and $167<_6 267$ (or $167 >_2 157$), which agrees with 
\Cref{thm:gneck_positroid_biject}.
\end{example}

We next define decorated permutations, which are equivalent to \textit{bounded affine permutations} \cite{MSTwist} or \textit{juggling patterns} \cite{KLS}.  In particular, one can extend a decorated permutation, essentially a map from $[n]$ to $[n]$, by $n$ periodically to yield a bijection $\ZZ \to \ZZ$. 

\begin{definition}[{\cite[Definition 13.3]{Post}}]
\label{def:dec_perm}
Let $\pi \in S_n$ and denote $\mathrm{Fix}(\pi) = \left\{i \in [n] \, | \, \pi(i) = i\right\}$. We say $\tilde{\pi} = (\pi,\mathrm{col})$ is a \textbf{decorated permutation} when pairing $\pi$ with the coloring function $\mathrm{col}: \mathrm{Fix}(\pi) \to \{-1,1\}$.
\end{definition}

For $i \in \mathrm{Fix}(\pi)$, we write $\tilde{\pi}(i) = \begin{cases}
    \underline{i} & \text{ if } \mathrm{col}(i) = -1 \\
    \overline{i} & \text{ if } \mathrm{col}(i) = 1
\end{cases}.$ This yields $\tilde{\pi}$ can be written in one-line notation. We also let $\tilde{\pi}^{-1} = (\pi^{-1},\mathrm{col})$ be the formal inverse of $\tilde{\pi}$ where we take the $S_n$-inverse of $\pi$ and the coloring function on $\mathrm{Fix}(\pi)$ is unchanged. Then each Grassmann necklace $\gneck$, resp. reverse Grassmann necklace $\rgneck$, gives rise to a decorated permutation $\tilde{\pi}\left[\gneck\right]$, resp. $\tilde{\pi}\left[\rgneck\right]$ by the following:

For $\gneck$,
\bi
\item if $\gneckset{i+1} = (\gneckset{i} \setminus \{i\}) \cup \{j\}$ and $j \ne i$, then $\tilde{\pi}\left[\gneck\right](i) = j$;
\item if $\gneckset{i+1} = \gneckset{i}$ and $i \not \in \gneckset{i}$, then $\tilde{\pi}\left[\gneck\right](i) = \underline{i}$;
\item and if $\gneckset{i+1} = \gneckset{i}$ and $i \in \gneckset{i}$, then $\tilde{\pi}\left[\gneck\right](i) = \overline{i}$.
\ei

For $\rgneck$,
\bi
\item if $\rgneckset{i-1} = (\rgneckset{i} \setminus \{i\}) \cup \{j\}$ and $j \ne i$, then $\tilde{\pi}\left[\rgneck\right](i) = j$;
\item if $\rgneckset{i-1} = \rgneckset{i}$ and $i \not \in \rgneckset{i}$, then $\tilde{\pi}\left[\rgneck\right](i) = \underline{i}$;
\item and if $\rgneckset{i-1} = \rgneckset{i}$ and $i \in \rgneckset{i}$, then $\tilde{\pi}\left[\rgneck\right](i) = \overline{i}$.
\ei

When $\gneck$ and $\rgneck$ parameterize the same positroid cell $\Pi_\mathcal{P}$, then $\tilde{\pi}\left[\gneck\right]^{-1} = \tilde{\pi}\left[\rgneck\right]$ and so when context of $\Pi_\mathcal{P}$ is clear, we write $\tilde{\pi}\left[\gneck\right]:= \tilde{\pi}$ and $\tilde{\pi}\left[\rgneck\right]:= \tilde{\pi}^{-1}$.

We say $i \in [n]$ is an \textit{anti-exceedance} of $\tilde{\pi}$ if either $i<\tilde{\pi}^{-1}(i)$ or $\tilde{\pi}(i) = \overline{i}$. Similarly, we say $i \in [n]$ is an \textit{exceedance} of $\tilde{\pi}$ if either $i > \tilde{\pi}^{-1}$ or $\tilde{\pi}(i) = \overline{i}$. We consider the anti-exceedance sets of $\tilde{\pi}$ with respect to $a$-ordering to obtain $\gneck$ from $\tilde{\pi}$ and the exceedence sets of $\tilde{\pi}^{-1}$ with respect to $(a+1)$-ordering to obtain $\rgneck$ from $\tilde{\pi}^{-1}$. Explicitly, \[\gneckset{a} = \left\{ i \in [n] \, |\, i <_a \tilde{\pi}^{-1}(i) \text{ or } \tilde{\pi}(i) = \overline{i} \right\}\] and \[\rgneckset{a} = \left\{ i \in [n] \, |\, i >_{a+1} \tilde{\pi}(i) \text{ or } \tilde{\pi}(i) = \overline{i} \right\}.\]

\begin{lemma}[{\cite[Lemma 16.2]{Post}}]
\label{lem:necklace_dec_perm_bijection}
Let $\tilde{\pi}\left[\gneck\right] = \tilde{\pi}$ and $\tilde{\pi}\left[\rgneck\right]= \tilde{\pi}^{-1}$. The maps $\gneck \to \tilde{\pi}$ and $\tilde{\pi} \to \gneck$, resp. the maps $\rgneck \to \tilde{\pi}^{-1}$ and $\tilde{\pi}^{-1} \to \rgneck$, as constructed above, are pairs of inverse maps and provide a bijection of decorated permutations on $n$ letters with Grassmann necklaces and reverse Grassmann necklaces, respectively. 
\end{lemma}

We remark that \cite{Post} only gave this construction and bijection for $\gneck$, but one can easily verify that the dual construction and bijection for $\rgneck$ follows similarly. Thus, we have that decorated permutations are in bijection with positroids via \Cref{thm:gneck_positroid_biject} and \Cref{lem:necklace_dec_perm_bijection}, and so any decorated permutation $\tilde{\pi}$ uniquely indexes positroid cell $\Pi_\mathcal{P}$.

We conclude this section in defining a statistic on $\tilde{\pi}$ called the \textbf{alignment number} of $\tilde{\pi}$ and denoted $\ell(\tilde{\pi})$ {\cite[Section 17]{Post}}. The choice of denoting the alignment number of $\tilde{\pi}$ by $\ell(\tilde{\pi})$ which is reminiscent of the \textit{length} function on $S_n$ is due to the role the alignment number plays as the length function in the \textit{circular Bruhat order} (see {\cite[Definition 17.5]{Post}}). Moreover, for $\tilde{\pi}$ corresponding to $\mathcal{P}$, then $\ell(\tilde{\pi})$ is the codimension of $\Pi_\mathcal{P}$ inside the top-dimensional cell of $\tnngrass$ {\cite[Proposition 17.10]{Post}}.

\begin{definition}[{\cite[Section 17]{Post}},{\cite[Definition 4.7]{OPS}}]
    \label{def:alignments}
    For $i,j \in [n]$, the pair $(i,j)$ forms an \textbf{alignment} in $\tilde{\pi}$ if $i <_i j <_i \tilde{\pi}(j) <_i \tilde{\pi}(i)$. Then the \textbf{alignment number} of $\tilde{\pi}$, denoted $\ell(\tilde{\pi})$, is the number of alignments in $\tilde{\pi}$. 
\end{definition}

\begin{example}
\label{ex:positroid_example_matrix_a}
The decorated permutation $\tilde{\pi}$ and its inverse $\tilde{\pi}^{-1}$ corresponding to the positroid $\mathcal{P}(B)$ in Example \ref{ex:positroid_example_matrix} are $\tilde{\pi} = 5467213$ and $\tilde{\pi}^{-1} = 6572134$. Then $\ell(\tilde{\pi}) = 2$ as $(1,2)$ and $(5,6)$ are the only alignments in $\tilde{\pi}$ since $1 <_1 2 <_1 \tilde{\pi}(2) = 4 <_1 \tilde{\pi}(1) = 5$ and $5<_5 6 <_5 \tilde{\pi}(6) = 1 <_5 \tilde{\pi}(5) = 2$. Similarly, $\ell(\tilde{\pi}^{-1}) = 2$ with $(1,2)$ and $(4,5)$ being the only alignments in $\tilde{\pi}^{-1}$.
\end{example}

\section{Plabic Graphs}
\label{sec:plabic_graphs}

In this section, we give another object which parametrizes the positroid cell $\Pi_\mathcal{P}$ (here $\mathcal{P}$ is an arbitrary positroid), namely, planar bicolored (plabic) graphs $G$ introduced in \cite{Post}. This parameterization is given explicitly through the famous \textit{boundary measurement map} (see \Cref{ssec:boundary_meas_map}). The combinatorial data presented in \Cref{sec:positroid_combo} which index $\Pi_{\mathcal{P}}$, namely Grassmann and reverse Grassmann necklaces and decorated permutations, also surface as combinatorial data of plabic graphs.

\subsection{Basic definitions and well-definedness}
\label{ssec:plabic_basic}

We begin by defining plabic graphs in the sense of \cite{Post}. 

\begin{definition}
\label{def:plabic_graph}
A \textbf{plabic graph} $G$ is an undirected graph embedded in a disk satisfying that each boundary vertex (on the disk) is incident to exactly one internal edge together with a proper 2-coloring on vertices.
Let $F(G)$ denote faces of $G$, $E(G)$ denote the edges of $G$, and $V(G) = B \sqcup W$ the vertices of $G$. We label vertices on the boundary (disk) of $G$ clockwise cyclically with $[n]$ and all boundary vertices are black. We say $G$ is of type $(k,n)$ if $|B| - |W| = n-k$.
\end{definition}

\Cref{fig:strands} gives an example plabic graph $G_0$. We proceed in illustrating the so-called ``reduction moves" and ``local moves" on $G$ in Figures \ref{fig:reduction} and \ref{fig:local_moves}. The reduction moves are (R1) Parallel edge reduction, (R2) (internal) Leaf reduction, and (R3) Dipole contraction. The local moves are (M1) Square move and (M2) Bivalent vertex contraction or expansion. For further details on these moves, see {\cite[Section 12]{Post}}. We define the \textit{move equivalence class} of $G$ to be the set of all other plabic graphs $G'$ which can be reached by applying some sequence of (M1) and (M2) to $G$. We say $G$ is \textbf{reduced} if no other $G'$ in the move equivalence class of $G$ is such that we can apply a reduction move to $G'$ {\cite[Definition 12.5]{Post}}. We assume $G$ is always reduced. 

\begin{figure}[ht]
\begin{center}
\includegraphics[width=0.9\linewidth]{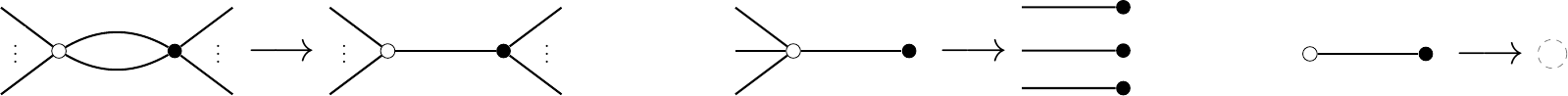}
\caption[Reduction Moves]{(Left): Parallel edge reduction (R1); (Middle): Leaf reduction (R2); (Right): Dipole reduction (R3).}
\label{fig:reduction}
\end{center}
\end{figure}

\begin{figure}[ht]
\begin{center}
\includegraphics[width=.9\linewidth]{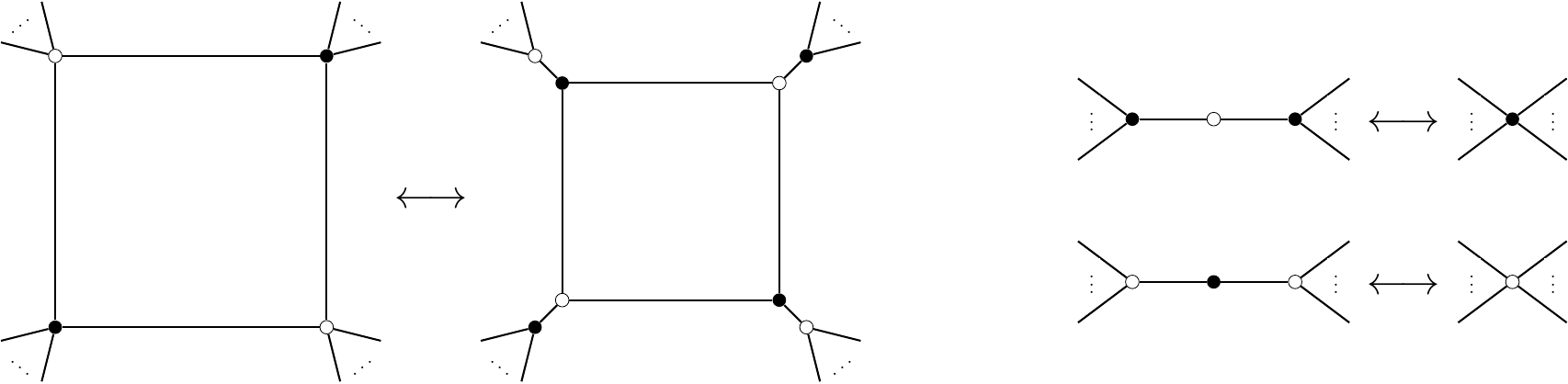}
\caption[Local Moves]{(Left): Square move (M1); (Right): Bivalent vertex contraction and expansion (M2).}
\label{fig:local_moves}
\end{center}
\end{figure}

While (R2) allows for removing internal leaves, leaves appearing on the boundary of $G$ are allowed and we call such leaves \textbf{lollipops} (see \Cref{fig:lollipops}). Suppose that such a lollipop is incident to the boundary vertex $i$ and the internal vertex of the lollipop is $v \in V(G)$. We say that $v$ is a \textit{white} lollipop at $i$ if $v \in W$ and otherwise that $v$ is a \textit{black} lollipop at $i$ if $v \in B$. Since we require a proper 2-coloring, if $v$ is a black lollipop at $i$, then we insert a white vertex $w$ between $i$ and $v$ by (M2). We now introduce an invariant of $G$ under equivalence moves and an important object of study for developing our theory of height functions. 

\begin{figure}[ht]
\begin{center}
\includegraphics[width = .3\linewidth]{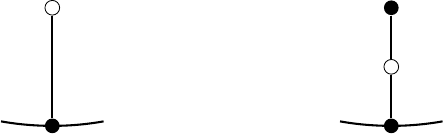} 
\caption[Lollipops]{(Left): A white lollipop; (Right): A black lollipop.}
\label{fig:lollipops}
\end{center}
\end{figure}

\begin{definition}[\cite{Post}]
\label{def:strands}
A \textbf{Postnikov} (also \textbf{zigzag}) \textbf{strand} $\tau$ is a directed path traversing edges of $G$ which originates at a boundary vertex $b_i$ and terminates at some boundary vertex $b_j$ such that at any $w \in W$, $\tau$ turns maximally left, and at each $b\in B$, $\tau$ turns maximally right.
\end{definition}

\Cref{fig:strands} also gives an example of strands overlaid on $G_0$. We consider each strand up to homotopy and obtain the following consequences from the reduced property of $G$. Each strand $\tau$ has a unique origin and unique terminal boundary vertex with no self-intersections {\cite[Theorem 13.2, Corollary 14.2]{Post}}. For any strand $\tau$ we write $\vec{\tau}_i$ (or simply $\vec{i}$) to emphasize that $\tau$ terminates at the $i$th boundary vertex and $\cev{\tau}_i$ (or simply $\cev{i}$) to emphasize that $\tau$ originates at the $i$th boundary vertex. Further, for any pair of strands $\tau$ and $\delta$, these strands intersect transversely with a finite number of intersections, and the orientations of $\tau$ and $\delta$ at their intersections are opposite of each other {\cite[Theorem 13.2, Corollary 14.2]{Post}}.

\begin{figure}[ht]
\begin{center}
\includegraphics[width = 6.5in]{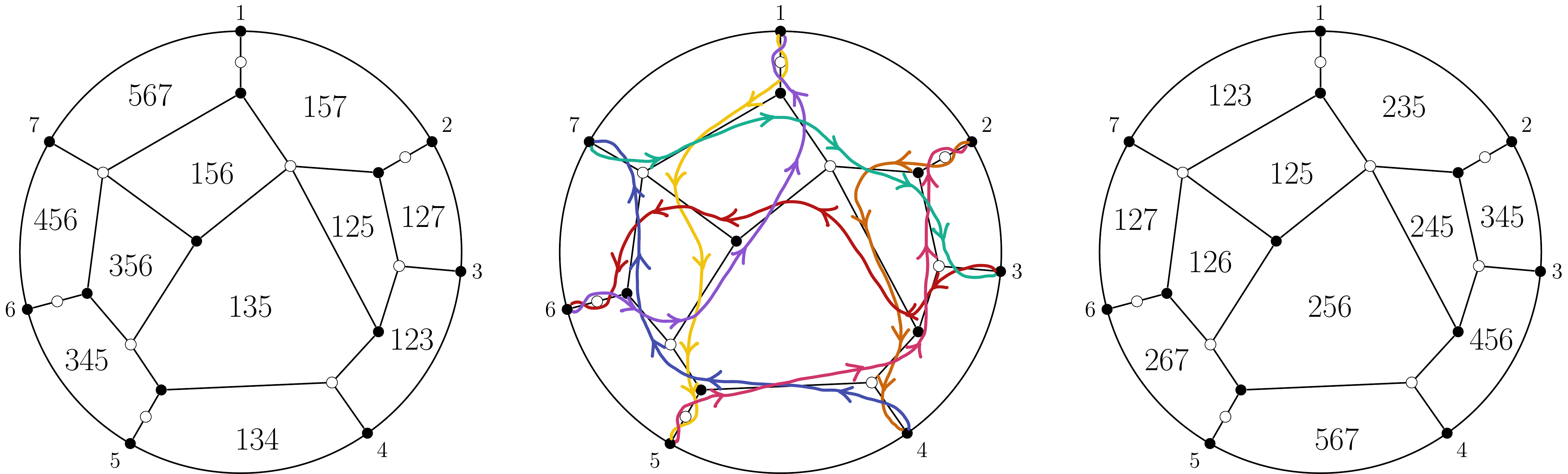}
\caption{(Left): The source-labeling of $G$; (Middle): The strands of $G$ with orientations indicated by arrowheads; (Right): The target-labeling of $G$.}
\label{fig:strands}
\end{center}
\end{figure}

We introduce a particularly useful edge-labeling of $G$ by strands as in \cite{KW}. For each $e \in E(G)$ traversed by strands $\cev{\tau}_i$ and $\cev{\tau}_j$, we label $e$ with $\{i,j\}$. If $i < j$, we label $e$ with $[i,j]$ to encode this ordering. The following definition gives an equivalent characterization of reduced $G$. 

\begin{definition}[{\cite[Definition 10.3]{KW}}]
    \label{def:res_prop}
Let $G$ have the edge-labeling described above and for any internal $v \in V(G)$, let $E_v \subseteq E(G)$ be the collection of edges incident to $v$. Then $G$ has the \textbf{resonance property} if and only if for each $v \in V(G)$, edges of $E_v$ have labels $[i_1,i_2],[i_2,i_3], \ldots, [i_{m-1},i_m], [i_1,i_m]$ appearing in clockwise order about $v$ such that $i_1 < i_2 < \cdots < i_m$.
\end{definition}

A depiction of the resonance property is given locally at both white and black vertices in \Cref{fig:resprop}. Then the result that the resonance property coincides with reduced $G$ is as follows. 

\begin{theorem}[{\cite[Theorem 10.5]{KW}}]
    \label{thm:resonance_is_reduced}
A plabic graph $G$ is reduced if and only if $G$ has the resonance property. 
\end{theorem}

Note that the choice of labeling edges of $G$ by indices of strands $\cev{\tau}_i$ can be interchanged with labeling edges of $G$ by indices of strands $\vec{\tau}_i$ to equivalently define the resonance property. This will allow for a more robust use of the resonance property with respect to source- and target-labeling of $G$.

One can verify that the collection of all $n$ strands of $G$ give rise to a permutation $\pi_G$ defined by $\pi_G(i) = j$ if $\cev{\tau_i} = \vec{\tau_j}$. Further, we upgrade $\pi_G$ to a decorated permutation $\tilde{\pi}_G$ where for any $\cev{\tau_i} = \vec{\tau_i}$, we write $\tilde{\pi}_G(i) = \underline{i}$ if there is a black lollipop at boundary vertex $i$ and $\tilde{\pi}_G(i) = \overline{i}$ if there is a white lollipop at boundary vertex $i$. We refer to $\tilde{\pi}_G$ as the \textbf{trip permutation} of $G$. We often will omit the subscript $G$ when referencing the trip permutation $\tilde{\pi}_G$ and simply write $\tilde{\pi}$ due to the following theorem.

\begin{theorem}[Fundamental theorem of reduced plabic graphs {\cite[Theorem 13.4]{Post}},{\cite[Theorem 7.1.23]{FWZ7}}]
\label{thm:fundamental_thm_reduced_plabic_graphs}
For $G$ and $G'$ reduced plabic graphs, then $G$ and $G'$ are move equivalent if and only if $G$ and $G'$ have the same trip permutation.     
\end{theorem}

Now since decorated permutations index $\Pi_\mathcal{P}$ and the trip permutation $\tilde{\pi}$ of any reduced plabic graph $G$ is a decorated permutation, then also move equivalence classes of reduced $G$ give a parameterization of $\Pi_\mathcal{P}$. For this reason, we say $G$, up to move equivalence, is the plabic graph of $\Pi_\mathcal{P}$. Using {\cite[Proposition 16.4]{Post}}, for $\gneck$ (or $\rgneck$) and $G$ corresponding to the same positroid cell $\Pi_\mathcal{P}$, then $\tilde{\pi}\left[\gneck\right] = \tilde{\pi}_G$, that is, the decorated permutations coming from Grassmann and reverse Grassmann necklaces are exactly the trip permuations for $G$. 

\begin{example}
\label{ex:square_move_w_strands} 
Let $G_0$ be the plabic graph of \Cref{fig:strands}. Then $G_0$ has the corresponding trip permutation $\tilde{\pi} = 5467213$, indeed the same decorated permutation of $\mathcal{P}(B)$ in \Cref{ex:positroid_example_matrix_a}, so $G_0$ is the plabic graph of $\mathcal{P}(B)$ and we will write $G_B$ in place of $G_0$. 
\end{example}

\subsection{The face labeling of \textit{G}}

We associate face labels to $G$ via the data of Postnikov strands under the following conventions as introduced in \cite{QCScott},\cite{JS} under the guise of \textit{double-wiring diagrams} and \textit{alternating strand diagrams}, respectively. 

\bi
\item  \textbf{Source-labeling.} For each strand $\cev{i}$, we label each face to the left of $\cev{i}$ with index $i$. We denote the set of all source-labeled faces of $G$ by $\cev{F}(G)$.
\item \textbf{Target-labeling.} For each strand $\vec{i}$, we label each face to the left of $\vec{i}$ with index $i$. We denote the set of all target-labeled faces of $G$ by $\vec{F}(G)$.
\ei

\Cref{fig:strands} gives source- and target-labeling of $G_B$ in our running example. After fixing a labeling on $G$, we refer to the strand which contributes the index $i$ to face labels as the $i$th-strand. For example, in target-labeling, we would call the strand $\vec{i}$ the $i$th-strand. 

Let $\overline{a+b} \equiv a+b \mod{n}$ denote modular arithmetic taken modulo $n$. We then say that a trip permutation $\tilde{\pi}$ is of \textbf{type} $\mathbf{(k,n)}$ if we observe $k = \frac{1}{n}\SumBlank{1 \le i \le n}(\overline{\tilde{\pi}(i) - i)} + \SumBlank{\substack{i \\ \tilde{\pi}(i) = \overline{i}}}1$ {\cite[Section 3]{KLS}}{\cite[Section 2]{MSTwist}}. Unsurprisingly, for any $G$ of type $(k,n)$, then $\tilde{\pi}_G$ is of type $(k,n)$. 

\begin{proposition}[{\cite[Proposition 8.3]{OPS}}{\cite[Proposition 4.3]{MSTwist}}]
\label{prop:plabic_type_kn}
For a reduced plabic graph $G$ with trip permutation of type $(k,n)$, then each face of $G$ lies to the left of exactly $k$ strands.
\end{proposition}

From \Cref{prop:plabic_type_kn}, we have that any $f \in F(G)$, independent of labeling, is labeled by $k$ elements of $[n]$. In particular, we will let $f$ refer to both the region of $G$ and the face label in $\dbinom{[n]}{k}$. From the perspective of cluster algebras, $G$ encodes a seed for $\poring$ where each $f \in F(G)$ corresponds to the Pl\"ucker coordinate $\Delta_f$. Generally, choice of source- or target-labeling will result in $f\in F(G)$ having differing face labels, i.e. differing (isomorphic) cluster algebras. However, for faces $f \in F(G)$ which lie on the boundary disk of $G$, these faces are labeled by $\rgneck$, resp. $\gneck$, in source-labeling, resp. target-labeling. 

\begin{proposition}[{\cite[Proposition 8.3 (1)]{OPS}}{\cite[Proposition 4.3]{MSTwist}}]
\label{prop:gneck_rgneck_as_faces}
    For reduced, source-labeled $G$ parameterizing the positroid cell $\Pi_\mathcal{P}$, the boundary face $f$ contained between boundary vertices $i$ and $i+1$ is $\rgneckset{i}$ where $\rgneck$ is the reverse Grassman necklace of $\Pi_\mathcal{P}$. For reduced, target-labeled $G$ parameterizing the positroid cell $\Pi_\mathcal{P}$, the boundary face $f$ contained between boundary vertices $i$ and $i+1$ is $\gneckset{i+1}$ where $\gneck$ is the Grassman necklace of $\Pi_\mathcal{P}$.
\end{proposition}

\begin{remark}
\label{rk:boundary_faces}
One can observe that the sequences $\rgneck$ and $\gneck$ of \Cref{ex:positroid_example_matrix} appear as the boundary face labels in source- and target-labeling of $G_B$ in \Cref{fig:strands} as is consistent with \Cref{prop:gneck_rgneck_as_faces}.
\end{remark}

\subsection{Results of almost monotonicity on plabic graphs}

We study local properties of face labels of $G$ which we obtain from strands. As was in the previous section, we let any face $f \in F(G)$ refer to both the face label and region of $G$. The properties of face labels obtained in this subsection will be crucial in defining and developing the theory of height functions on $G$ in \Cref{ssec:heights}.

\begin{figure}[ht]
    \centering
    \includegraphics[width=2in]{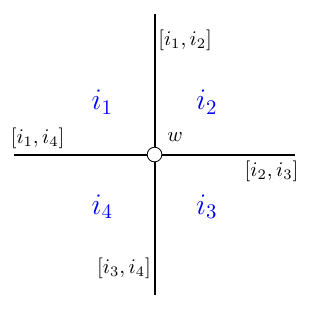} \hspace{7em} \includegraphics[width=2in]{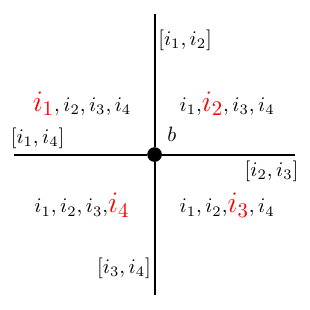}
    \caption{(Left): The resonance property at a white vertex $w$ and its implications for \textit{included} face labels indices (in blue) as in \Cref{lem:common_face_indexing}; (Right): The resonance property at a black vertex $b$ and its implications for \textit{excluded} face labels indices (in red) as in \Cref{lem:common_face_indexing}.}
    \label{fig:resprop}
\end{figure}

\begin{lemma}
\label{lem:common_face_indexing}
For fixed $w \in W$, let $\mathbf{f}_w = \{f_1,f_2,\ldots,f_p\}$ be the collection of faces incident to $w$ such that $f_i$ is adjacent to $f_{i-1}$ and $f_{i+1}$ with indices taken $\mod p$. Then for each $i \in [p]$, $f_i = S_w \cup \{j_i\}$ where $S_w$ is the unique $(k-1)$-subset of $[n]$ such that $S_w = \bigcap_{i=1}^p f_i$. Moreover, there exists a choice of indexing such that $j_1 > j_2 > \cdots > j_p$ when traversing faces of $\mathbf{f}_w$ counterclockwise about $w$. 

For fixed $b \in B$, let $\mathbf{g}_b = \{g_1,g_2,\ldots,g_p\}$ be the collection of faces incident to $b$ such that $g_i$ is adjacent to $g_{i-1}$ and $g_{i+1}$ with indices taken $\mod p$. Then for each $i \in [p]$, $g_i = T_b \setminus \{j_i\}$, where $T_b$ is the unique $(k+1)$-subset of $[n]$ such that $T_b = \bigcup_{i=1}^p g_i$. Moreover, there exists a choice of indexing such that $j_1 > j_2 > \cdots > j_p$ when traversing faces of $\mathbf{g}_b$ counterclockwise about $b$.
\end{lemma}

\begin{proof}
Fix any $w \in W$. Any strand $\tau$ which approaches $w$ turns maximally left at $w$ and thus immediately leaves $w$. Thus, $\tau$ contributes a face label entry to only one face incident to $w$, namely, the face lying to the left of $\tau$ with respect to the strand's orientation. As this is true for all $f_\ell \in \mathbf{f}_w$, we have that each $f_\ell$ has a unique face label index, say $j_\ell$. By the resonance property, we have edge-labels $[i_1,i_2],[i_2,i_3],\ldots, [i_1,i_p]$ appearing in a clockwise order about $w$ such that $i_1 < i_2 < \cdots < i_p$. In particular, one can observe that the face between edges $[i_{\ell-1},i_\ell]$ and $[i_\ell, i_{\ell +1}]$ has the unique face label index $i_\ell$ for $\ell \in [p-1]$ and the face between edges $[i_{p-1},i_p]$ and $[i_1,i_p]$ contains the unique face label index $i_p$. This yields that traversing faces of $\mathbf{f}_w$ clockwise about $w$ yields a sequence $i_1 < i_2 < \cdots < i_p$ where each $i_\ell$ is a face label index unique to a single face incident to $w$. Setting $j_\ell = i_{p-\ell +1}$, we have the desired decreasing sequence $j_1 > j_2 > \cdots > j_p$ such that $f_\ell = S_w \cup \{j_\ell\}$ for all $f_\ell \in \mathbf{f}_w$ and $f_{\ell+1}$ is adjacent and counterclockwise from $f_\ell$ about $w$ for all $\ell \in [p]$. 

Now fix any $b \in B$. Any strand $\tau$ which approaches $b$ turns maximally right at $b$ and thus immediately leaves $b$. Thus, only one face $g_\ell \in \mathbf{g}_b$ lies to the right of $\tau$ and does not contain the face label index, say $i_\ell$, contributed by $\tau$. As this is true for all $g_\ell \in \mathbf{g}_b$, we have that each $g_\ell$ has a uniquely excluded face label index $i_\ell$ and contains the remaining $p-1$ face label indices contributed by the remaining strands approaching $b$. As each face of $G$ is labeled by some $k$-subset of $[n]$, then the face label of any $g_\ell \in \mathbf{g}_b$ is given as $g_\ell = (T'_b \cup \bigcup_{j=1}^p i_j) \setminus \{i_\ell\}$ where $T'_b \in \dbinom{[n]}{k-p+1}$ is common to each $g_s$. Letting $T_b = T'_b \cup \bigcup_{j=1}^p i_j$, we reach the desired conclusion about excluded face label indices. Again applying the resonance property, we have edge-labels $[i_1,i_2],[i_2,i_3],\ldots, [i_1,i_p]$ appearing in a clockwise order about $b$ such that $i_1 < i_2 < \cdots < i_p$. One can observe that the face between edges $[i_{\ell-1},i_\ell]$ and $[i_\ell,i_{\ell+1}]$ is the face with excluded face label index $i_\ell$ for all $\ell \in [p-1]$ and the face between edges $[i_{p-1},i_p]$ and $[i_1,i_p]$ is the face with excluded face label index $i_p$. Traversing faces of $\mathbf{g}_b$ clockwise about $b$ yields a sequence $i_1 < i_2 < \cdots < i_p$ where each $i_\ell$ is a excluded face label index unique to a single face incident to $b$. Setting $j_\ell = i_{p-\ell +1}$, we have the desired decreasing sequence $j_1 > j_2 > \cdots > j_r$ such that $g_\ell = T_b \setminus \{j_\ell\}$ for all $g_\ell \in \mathbf{g}_b$ and $g_{\ell+1}$ is adjacent and counterclockwise from $g_\ell$ about $b$ for all $\ell \in [p]$.
\end{proof}

By \Cref{lem:common_face_indexing}, for any fixed internal vertex $v \in V(G)$ we can uniquely identify a face $f$ incident to $v$ by either an included face label index if $v \in W$ or an excluded face label index if $v \in B$. Further, these included of excluded face label indices come with an ordering under a cyclic orientation about $v$, namely, the face label indices are decreasing when traversing faces counterclockwise about $v$. We codify this property in the following definition. 

\begin{definition}[Almost monotonicity]
\label{def:almost_monotonicity}
For white vertices, let $w \in W$ and $\mathbf{f}_w = (f_{1'},f_{2'},\ldots, f_{p'})$ be the collection of faces incident to $w$ such that $f_{i'}$ is adjacent to $f_{(i-1)'}$ and $f_{(i+1)'}$ for all $i'$ and where increasing in $i'$ corresponds to traversing faces of $\mathbf{f}_w$ in a clockwise or counterclockwise orientation about $w$. We say a function $\varphi_w: \mathbf{f}_w \to \RR$ is \textbf{almost monotonically decreasing at $w$} if we can choose $f_1 \in \mathbf{f}_w$ in such a way that \[\varphi_w(f_1) > \varphi_w(f_2) > \cdots >\varphi_w(f_p), \qquad \text{and} \qquad \varphi_w(f_p) < \varphi_w(f_1).\]

Analogously for internal black vertices, let $b \in B$ and $\mathbf{g}_b = (g_{1'},g_{2'},\ldots, g_{p'})$ be the collection of faces incident to $b$ such that $g_{i'}$ is adjacent to $g_{(i-1)'}$ and $g_{(i+1)'}$ for all $i'$ and where increasing in $i'$ corresponds to traversing faces of $\mathbf{g}_b$ in a clockwise or counterclockwise orientation about $b$. We say a function $\varphi_b: \mathbf{g}_b \to \RR$ is \textbf{almost monotonically decreasing at $b$} if we can choose $g_1 \in \mathbf{g}_b$ in such a way that \[\varphi_b(g_1) > \varphi_b(g_2) > \cdots >\varphi_b(g_p), \qquad \text{and} \qquad \varphi_b(g_p) < \varphi_b(g_1).\]
\end{definition}

In the special case when we have some globally defined function $\varphi:F(G) \to \RR$ for which the restrictions $\varphi|_{\mathbf{f}_w}:\mathbf{f}_w \to \RR$ are almost monotonically decreasing for all $w \in W$, or equivalently $\varphi|_{\mathbf{g}_b}:\mathbf{g}_b \to \RR$ are almost monotonically decreasing for all internal $b \in B$, then $\varphi$ will be a \textit{height function} (see \Cref{def:gen_heights}). Height functions are our main object of study in \Cref{sec:heights}. For simplicity, we typically choose to define almost monotonically decreasing functions with respect to each $w \in W$ rather than internal $b \in B$, however, the following example illustrates how \Cref{lem:common_face_indexing} yields almost monotonically decreasing functions which will serve as prototypes for the height functions introduced in \Cref{prop:fund_heights}.

\begin{example}
    \label{ex:almost_monotone_functs}
Assume the notation of \Cref{lem:common_face_indexing} and consider the counterclockwise cyclic orientation of faces about any $w\in W$. For each $w\in W$ we define the local function $\varphi_w:\mathbf{f}_w \to [n]$ given by $\varphi_w(f_i) = j_i$. As a consequence of \Cref{lem:common_face_indexing}, then $\varphi_w$ is almost monotonically decreasing at all $w \in W$. 

Alternatively, we consider the counterclockwise cyclic orientation of faces about any internal $b \in B$. For each internal $b \in B$ we define the local function $\psi_b:\mathbf{g}_b \to [n]$ given by $\psi_b(g_i) = \ell_i$. As a consequence of \Cref{lem:common_face_indexing}, then $\psi_b$ is almost monotonically decreasing at all internal $b \in B$.
\end{example}

\section{Matchings and Poset of Matchings}
\label{sec:matchings}

In this section, we introduce the combinatorial tool of \textit{almost perfect matchings} on plabic graphs $G$. We can organize almost perfect matchings on $G$ by their \textit{boundary conditions}, that is, the matched edges of $G$ which include some $k$-subset of boundary vertices of $G$. In our study of matchings on $G$, we explore two different applications. In \Cref{ssec:boundary_meas_map}, we explicitly parameterizate any positroid cell $\Pi_\mathcal{P}$ by (move equivalence classes of) $G$ via Postnikov's boundary measurement map {\cite[Theorem 12.7]{Post}} (see \Cref{thm:boundary_meas_map}). From this parameterization of $\Pi_\mathcal{P}$, we can equivalently characterize $\mathcal{P}$ as the collection of all $k$-subset boundary conditions of matchings possible on $G$. Alternatively, in \Cref{ssec:lattice_dimer} we study the lattice structures of \textit{perfect} matchings on arbitrary planar bipartite graphs $H$ of Propp {\cite[Theorem 2]{Propp}} and of almost perfect matchings on $G$ with boundary $I \in \mathcal{P}$ of Muller--Speyer {\cite[Theorem B.1]{MSTwist}}. We also discuss Muller--Speyer's construction of extremal matchings on $G$ with certain boundaries $I \in \mathcal{P}$, of which the present paper completes.

\subsection{The boundary measurement map}
\label{ssec:boundary_meas_map}

We begin in defining 1-dimers for arbitrary planar bipartite graphs, although this subsection will focus on the case of plabic $G$. 

\begin{definition}[1-Dimer]
\label{def:dimer_model}
Given a planar bipartite graph $H$ (not necessarily embedded in a disk), a \textbf{1-dimer}, or single dimer, $M$, is a collection of edges $M = \{e_i\} \subseteq E(G)$ such that each vertex $v \in V(G)$ is incident to exactly one edge $e_i \in M$.
\end{definition}

For clarity, when referring to an arbitrary planar bipartite graph $H$, we let $B(H)$ and $W(H)$ denote the black and white vertex sets of $H$, respectively. In the case when $|B(H)| = |W(H)|$, then a 1-dimer is also called a \textit{perfect matching}. We let $\mathcal{M}$ denote the set  of all perfect matchings on $H$. In our setting of plabic graphs $G$ of type $(k,n)$, we have an exceedance of black vertices to white vertices of $G$, namely, $|B| - |W| = n-k$. Therefore, a 1-dimer on $G$ is called an \textit{almost perfect matching}, as not all $b \in B$ will be incident to an edge $e_i \in M$. In particular, we specifically choose the $(n-k)$ unmatched black vertices of $G$ to be chosen from the $n$ black boundary vertices of $G$. Equivalently, this means any almost perfect matching $M$ on $G$ comes with a list, $I$, of $k$ boundary vertices of $G$ which are matched and we refer to $I$ as the \textbf{boundary condition} of $M$ and write $\partial M = I$. 

For $I \in \dbinom{[n]}{k}$, let $\mathcal{M}_I$ be the set of all almost perfect matchings on $G$ with boundary condition $I$. Note that we can also characterize any almost perfect matching $M$ on $G$ by the amendment that for each \textit{internal} $v \in V(G)$, $v$ is incident to exactly one edge $e_i \in M$ in \Cref{def:dimer_model}. We proceed in calling any 1-dimer on arbitrary $H$ or plabic $G$ as a matching, with adjectives perfect and almost perfect implied by context. 

We now give the explicit parameterization of $\Pi_\mathcal{P}$ by its corresponding plabic graph $G$ via the \textit{boundary measurement map}. To each $e \in E(G)$, we assign a real, positive edge weighting $\wt: E(G) \to \RR_{>0}$. Then for a matching $M$, we give the edge weight of $M$, $\wt_e(M)$, to be the product of edge weights $\wt(e_i)$ over all $e_i \in M$, that is, $\wt_e(M) = \ProdBlank{e_i\in M}\wt(e_i)$. We fix a boundary condition $I$ and obtain the partition function, $\Delta_I(G)$, as \[ \Delta_I(G) = \SumBlank{M \in \mathcal{M}_I}\wt_e(M) = \SumBlank{M \in \mathcal{M}_I}\ProdBlank{e_i \in M}\wt(e_i).\] Then the parameterization is as follows, where the formulation is that found in $\cite{Lam,GrassmannLam}$.

\begin{theorem}[{\cite[Theorem 12.7]{Post}}{\cite[Proposition 2.8]{PSW}}{\cite[Theorem 2.1]{Lam}}{\cite[Theorem 4.1]{GrassmannLam}}]
\label{thm:boundary_meas_map}
Let $\wt: E(G) \to \RR_{>0}$ be an edge weighting on $G$. Then the collection of coordinates $(\Delta_I(G))_{I \in \binom{[n]}{k}}$, where $\Delta_I(G) = \SumBlank{M \in \mathcal{M}_I}\wt_e(M)$, defines a unique point $\tilde{X} \in \affgrass_{\ge 0}$ such that $\Delta_I(G) = \Delta_I(\tilde{X})$ for all $I$ with respect to $\wt$. Alternatively, each $\tilde{X} \in \affgrass_{\ge 0}$ is realizable as some plabic graph $G$.
\end{theorem}

\begin{figure}[h]
    \centering
    \includegraphics[width=1.5in]{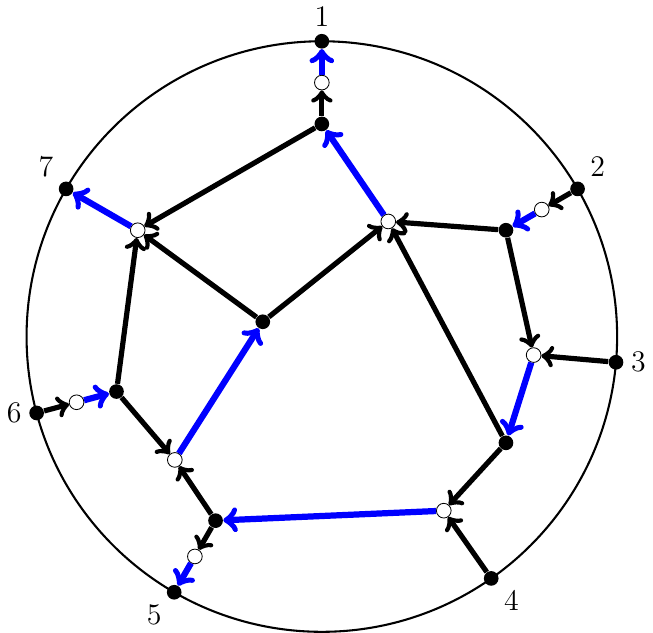} \hspace{.5em} \includegraphics[width=1.5in]{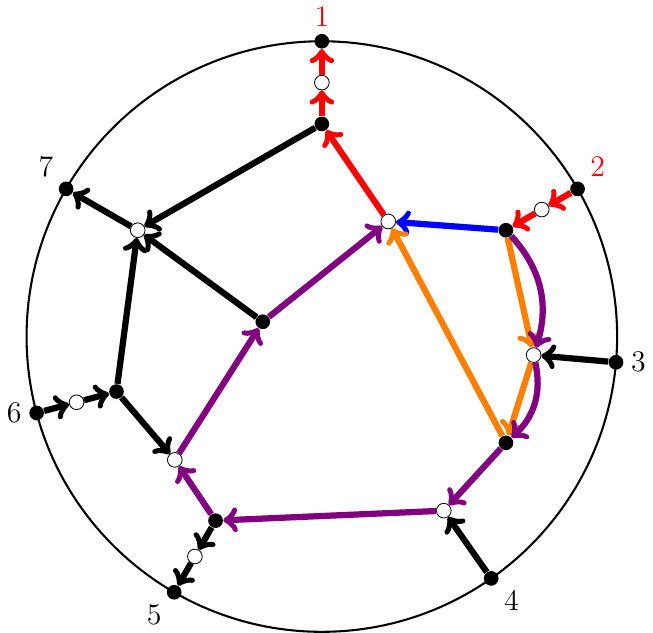} \hspace{.5em} \includegraphics[width=1.5in]{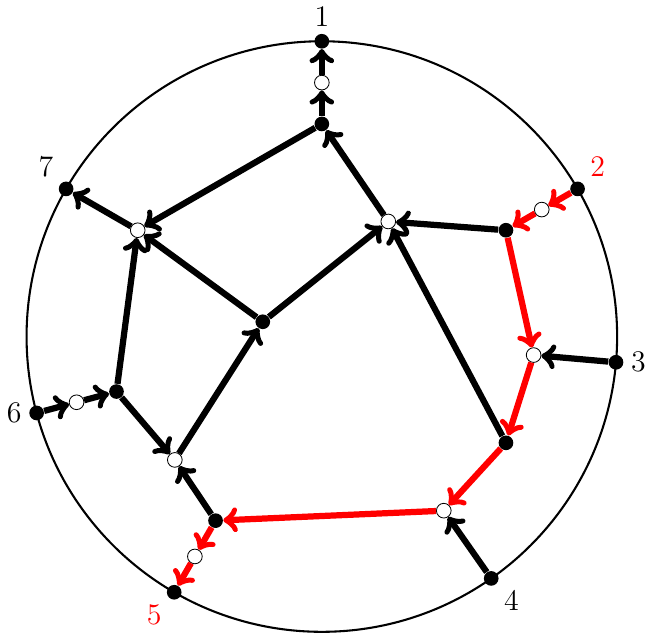} \hspace{.5em} \includegraphics[width=1.5in]{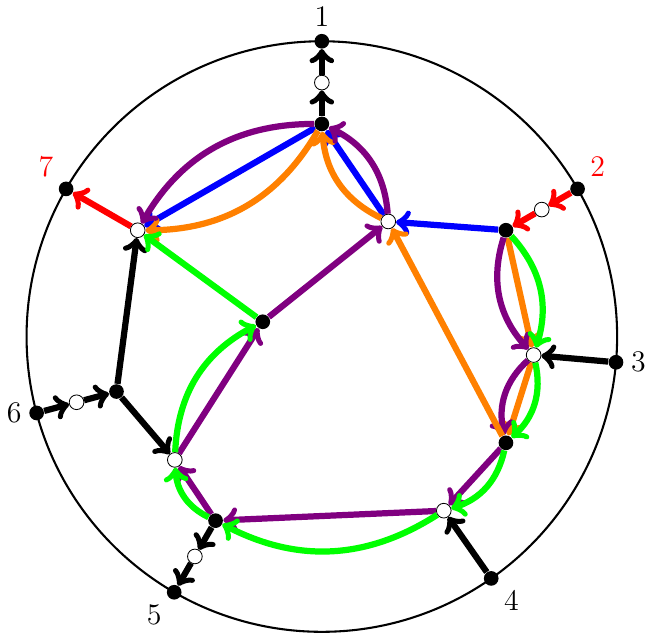}
    \caption{(Left): The perfect orientation with sink set $I_\mathcal{O} = \{157\}$ and corresponding matched edges in {\color{blue}}; (Remaining left-to-right): The paths from source vertex $2$ to sinks $1$, $5$, and $7$, respectively. Red arrows are those used in every path from $2$ to a sink vertex.}
    \label{fig:perforientations}
\end{figure}

We can explicitly define the image $\tilde{X} \in \affgrass_{\ge 0}$ under the boundary measurement map using the language of \textit{perfect orientations} of $G$ due to {\cite[Definition 2.6]{Talaska}} and appearing also in \cite{PSW,KO}. On $G$, a \textbf{perfect orientation} of $G$, $\mathcal{O}_G$ is an assignment of orientations to each $e \in E(G)$ satisfying that every internal $b \in B$ is incident to exactly one edge oriented towards $b$ and every $w \in W$ is incident to exactly one edge oriented away from $w$. Given the condition on white vertices, then there is some $I \in \dbinom{[n]}{k}$ such that the boundary vertex $b_i$ for each $i \in I$ is a \textit{sink} of $\mathcal{O}_G$ and we denote $I_\mathcal{O}$ as the set of sinks of $\mathcal{O}_G$. In particular, perfect orientations of $G$ are in bijection with matchings on $G$ where given a matching $M$, each $e \in M$ is the edge incident to some internal $b \in B$ and $w \in W$ such that $e$ is oriented from $w$ to $b$ and for any $e' \not \in M$, then $e' = (b',w')$ is oriented from $b'$ to $w'$. We fix a perfect orientation $\mathcal{O}_G = \mathcal{O}$ with set of sinks $I_\mathcal{O} = I$. For any directed path $P:j \to i$ from boundary vertex $j$ to boundary sink $i$ on $G$, we define the weight of $P$, $\wt_e(P)$, to be the product of edge weights along $P$ under the edge weighting $\wt$. If $i \in I$, then $\wt_e(P) = 1$ if and only if $P: i \to i$ and $\wt_e(P) = 0$ otherwise. We let $w(P)$ denote the winding number of $P$ (allowing that $P$ may contain a directed cycle) and $N_{i,j}$ the number of sinks $\ell \in I$ on the boundary of $G$ such that $i <_i \ell <_i j$. Then we define $\tilde{X} = (x_{i,j})$ as follows
\[
x_{i,j} = (-1)^{N_{i,j}} \SumBlank{P:j \to i} (-1)^{w(P)}\wt_e(P),
\]
where it is sufficient to index rows of $\tilde{X}$ by $i \in I$ as we cannot find a directed path which terminates at a source boundary vertex. This matrix $\tilde{X}$ we call the \textit{boundary measurement matrix} of $G$.

Then following from {\cite[Theorem 12.7, Corollary 16.5]{Post}}, {\cite[Proposition 2.8]{PSW}}, and {\cite[Theorem 3.3]{MSTwist}}, for $G$ corresponding to some $\Pi_\mathcal{P}$, then for the corresponding boundary measurement matrix of $G$, say $\tilde{X}$, we have $\tilde{X} \in \Pi_\mathcal{P}$. Moreover, \Cref{thm:boundary_meas_map} gives that the non-vanishing Pl\"ucker coordinates of $\tilde{X}$ are indexed by exactly the set of possible boundary conditions for matchings on $G$ and those Pl\"ucker coordinates which do vanish on $\tilde{X}$ are indexed by $J \in \dbinom{[n]}{k}$ such that $J$ is not a possible boundary condition for a matching on $G$. In particular, this yields that we can characterize $\mathcal{P}$ by matchings on $G$ as \[\mathcal{P} = \left\{ I \in \dbinom{[n]}{k} \, \bigg| \, \mathcal{M}_I \ne \emptyset\right\}.
\]
We conclude this subsection with example of relating the matrix $B$ in \Cref{ex:positroid_example_matrix} with the plabic graph $G_B$ via this boundary measurement map.

\begin{figure}[!b]
\begin{center}
\includegraphics[width =1.8in]{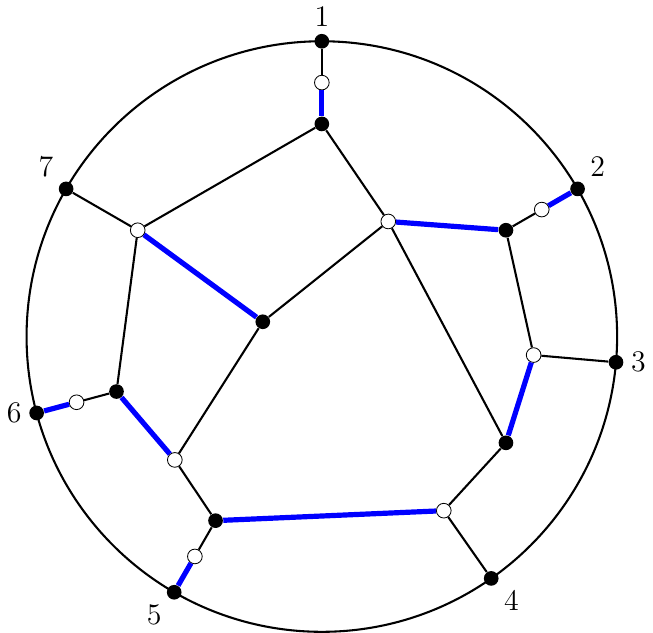} \hspace{2em} \includegraphics[width =1.8in]{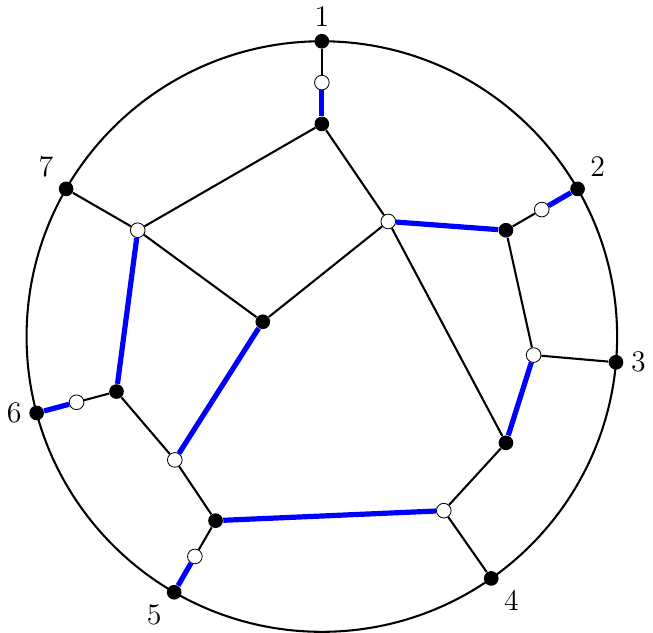} \hspace{2em} \includegraphics[width =1.8in]{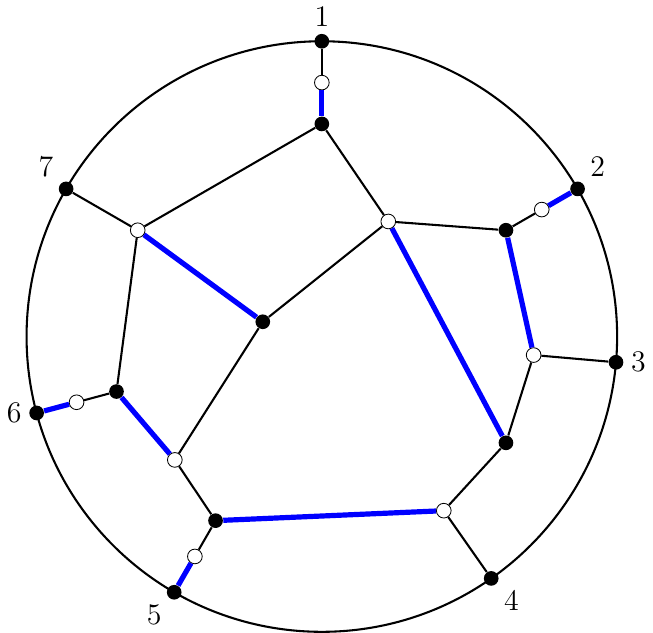}  
\includegraphics[width =1.8in]{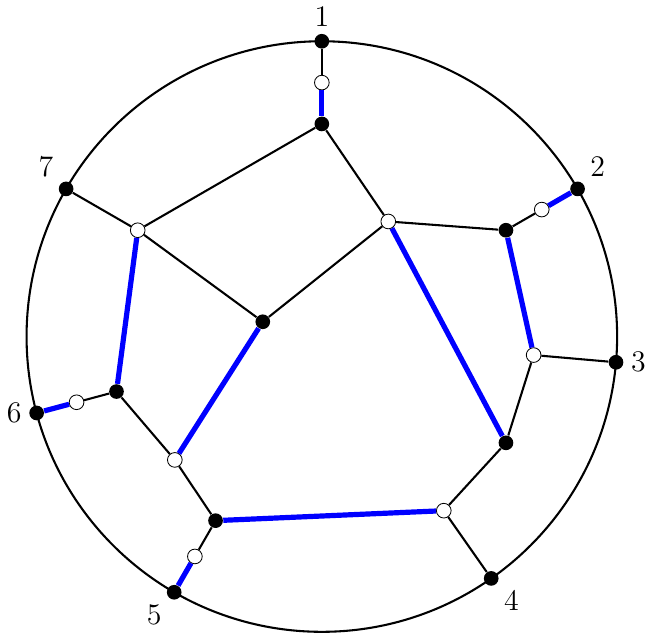} \hspace{2em} \includegraphics[width =1.8in]{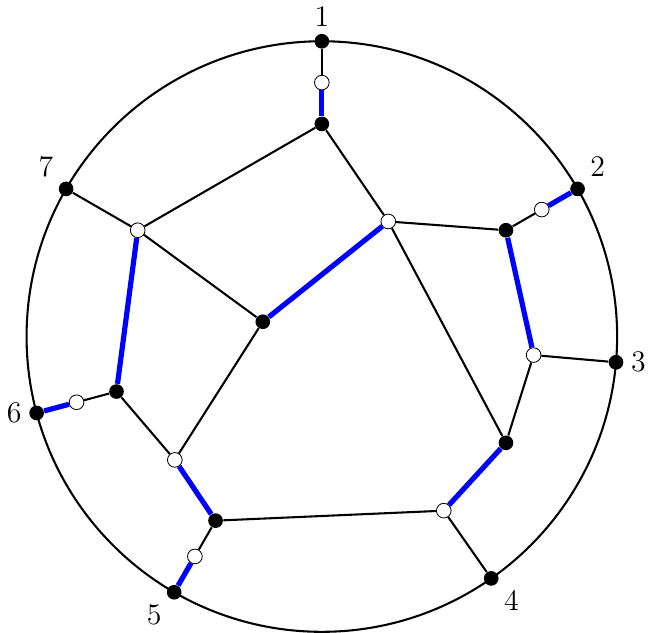}
\caption{All matchings with boundary condition $\{2,5,6\}$ on $G_B$. From top left to bottom right, across rows, label these matchings as $M_1$, $M_2$, $M_3$, $M_4$, and $M_5$.}
\label{fig:256_matchings}
\end{center}
\end{figure}

\begin{example}
    \label{ex:boundary_meas_map} 
One can note that there is a unique matching on $G_B$ with boundary $I = \{157\}$, i.e. there is a unique (acyclic) perfect orientation $\mathcal{O}$ on $G_B$ with sink set $I_\mathcal{O} = \{157\}$ (see \Cref{fig:perforientations}). If one chooses $\wt(e) = 1$ for all $e \in E(G_B)$, then $B$ is exactly the boundary measurement matrix of $G_B$. In particular, \Cref{fig:perforientations} shows all paths from $2$ to each sink vertex in this perfect orientation. From this we construct the second column of $B$, noting that there are three paths from $2$ to $1$ (no sinks between 1 and 2), one path from $2$ to $5$ (two sinks between 5 and 2), and four paths from $2$ to $7$ (one sink between 7 and 2), yielding $B_{1,2} = 3$, $B_{5,2} = 1$, and $B_{7,2}=-4$. 

Applying \Cref{thm:boundary_meas_map}, we consider the set of matchings $\mathcal{M}_{256}$ on $G_B$ (see \Cref{fig:256_matchings}), retaining
the edge weighting $\wt(e) = 1$ for all $e \in E(G_B)$. Then $\SumBlank{M \in \mathcal{M}_{256}}\wt_e(M) = \Delta_{256}(G_B) = \left|\mathcal{M}_{256}\right|$, which equals 
$\Delta_{256}(B) = 
\det \begin{pmatrix}
    3 & 0 & -1 \\
    1 & 1 & 0  \\
    -4 & 0 & 3
\end{pmatrix} =
5.$ More generally, $\Delta_I(B)$ for any $I \in \mathcal{P}(B)$ can be interpreted as this.
\end{example}

\subsection{The lattice structure for dimers}
\label{ssec:lattice_dimer}

We now briefly return to consideration of perfect matchings on planar bipartite graph $H$ to present the lattice structure on matchings due to \cite{Propp} in fullest generality. In particular, we are able to consider $G$ as a special case of $H$. One may observe the following procedure of obtaining new matchings from old matchings on $H$. Given a matching $M$ on $H$ such that $M$ contains exactly half of the edges of some $f \in F(H)$, then we can obtain a new matching $M'$ on $H$ by taking $M'$ to be the edges of $f$ complementary to those contained in $M$ and exactly the remaining edges not of $f$ contained in $M$. In such an instance, we say that $M$ contains all alternating edges of $f$. Further, for $M$ containing all alternating edges of $f$, if we orient these alternating edges clockwise about $f$, then either all alternating edges are from black vertices to white vertices or from white vertices to black vertices in this orientation and we refer to these edges as \textit{black-to-white edges} and \textit{white-to-black edges}, respectively. We encode this in the following definition.

\begin{definition}
\label{def:swivel_matchings}
For a matching $M$ on $H$ such that $M$ contains all alternating edges of some $f \in F(H)$, we say we can \textbf{swivel} $M$ at $f$ to obtain the matching $M'$. Orienting the alternating edges of $f$ in $M$ clockwise about $f$, we say that a swivel is an \textbf{upswivel} if $M$ contains all white-to-black edges of $f$ and $M'$ contains all black-to-white edges of $f$. Similarly, we say that a swivel is a \textbf{downswivel} if $M$ contains all black-to-white edges of $f$ and $M'$ contains all white-to-black edges of $f$. 
\end{definition}

The nomenclature of up- and downswivels in \Cref{def:swivel_matchings} corresponds to the partial ordering introduced by \cite{Propp}. We give this partial ordering of matchings on $H$ in the following.

\begin{definition}[\cite{Propp}]
\label{def:partial_order_on_matchings}
Let $M, M' \in \mathcal{M}$. Then $M \lessdot M'$ if $M'$ is obtained by upswiveling $M$ at some $f \in F(H)$, or equivalently if $M$ is obtained by downswiveling $M'$ at $f \in F(H)$. Generally, we write $M \le M'$ if $M'$ is obtained from $M$ by a sequence of upswivels at faces $f_1, \ldots, f_p$, or equivalently if $M$ is obtained from $M'$ by a sequence of downswivels at faces $f_p, \ldots, f_1$. 
\end{definition}

\begin{figure}[h]
\begin{center}
\includegraphics[height=1.1in]{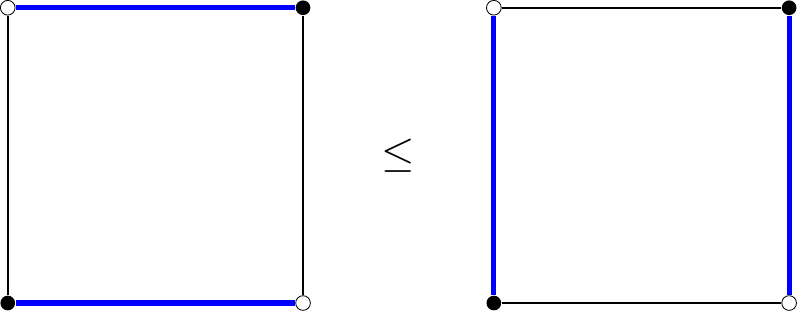} \hspace{2em} \includegraphics[height =1.1in]{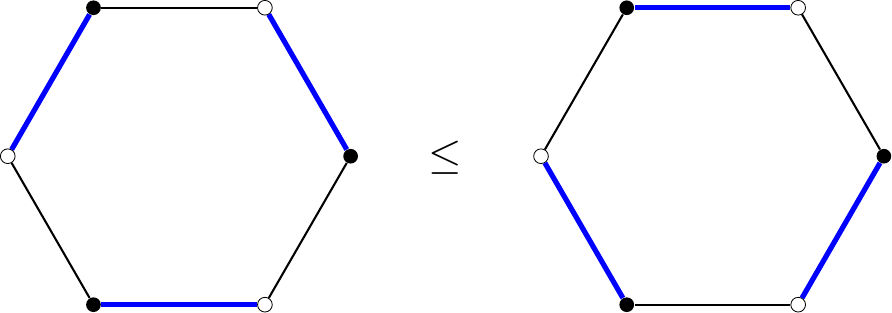}
\caption[Partial Order on Matchings]{(Left): Swiveling (with direction) at a square face; (Right): Swiveling (with direction) at a hexagonal face.}
\label{fig:square_and_hex_swivels}
\end{center}
\end{figure}

\Cref{fig:square_and_hex_swivels} demonstrates this partial ordering on isolated square and hexagonal faces. We now build towards a main result of Propp, giving a lattice structure to the set of matchings on $H$. Suppose $\vec{d} = (d_{v_1}, \ldots, d_{v_r}) \in \ZZ_{\ge 0}^{V(H)}$. A \textit{mixed dimer of type} $\vec{d}$ (a $\vec{d}$-factor in \cite{Propp}), $D_{\vec{d}}$, is a collection of edges (with possibly multiplicity\footnote{In fact, \cite{Propp} considers only a sublattice where no multiple edges occur.  We thank Yucong Lei for this observation.}) $D_{\vec{d}} = \{e\} \subseteq E(H)$ such that each $v_i \in V(H)$ is incident to exactly $d_{v_i}$ edges of $D_{\vec{d}}$, where $d_{v_i}$ is also known as the dimer-degree of $v_i$. Then a 1-dimer is the special case in which $d_{v_i} = 1$ for all $i$. We let $\mathcal{D}_{\vec{d}}$ denote the set of all mixed dimers of type $\vec{d}$. One can observe that given a mixed dimer $D \in \mathcal{D}_{\vec{d}}$ such that we may swivel $D$ at a face $f \in F(H)$ in the sense of \Cref{def:swivel_matchings}, we obtain a (possibly) new mixed dimer $D' \in \mathcal{D}_{\vec{d}}$. We can then assign the same partial ordering on mixed dimers and obtain the following lattice structure.

\begin{theorem}[{\cite[Theorem 2]{Propp}}]
\label{thm:matching_poset_is_fin_dist_lattice}
Under the partial ordering of \Cref{def:partial_order_on_matchings}, $(\mathcal{D}_{\vec{d}}, \le)$ is a finite distributive lattice for nonempty $\mathcal{D}_{\vec{d}}$.
\end{theorem}

Considering a plabic graph $G$ parameterizing $\Pi_\mathcal{P}$, the analogous result for $\mathcal{M}_I$ on $G$ is not a direct implication of \Cref{thm:matching_poset_is_fin_dist_lattice}, but rather is due to Muller--Speyer (see Appendix B of \cite{MSTwist}). We nonetheless have the following. 

\begin{theorem}[{\cite[Theorem B.1]{MSTwist}}]
\label{thm:matching_poset_bdy_cond}
Under the partial ordering of \Cref{def:partial_order_on_matchings}, $(\mathcal{M}_{I}, \le)$ is a finite distributive lattice for $I \in \mathcal{P}$.
\end{theorem}

An example of this lattice structure is shown for $(\mathcal{M}_{256},\le)$ on $G_B$ in \Cref{fig:256_poset}. We obtain from \Cref{thm:matching_poset_bdy_cond} that for nonempty $\mathcal{M}_I$, there is a unique minimal and unique maximal matching we denote by $M_-$ and $M_+$, respectively. We characterize these extremal matchings in the following. 

\begin{figure}[ht]
\begin{center}
\includegraphics[width=3.6in]{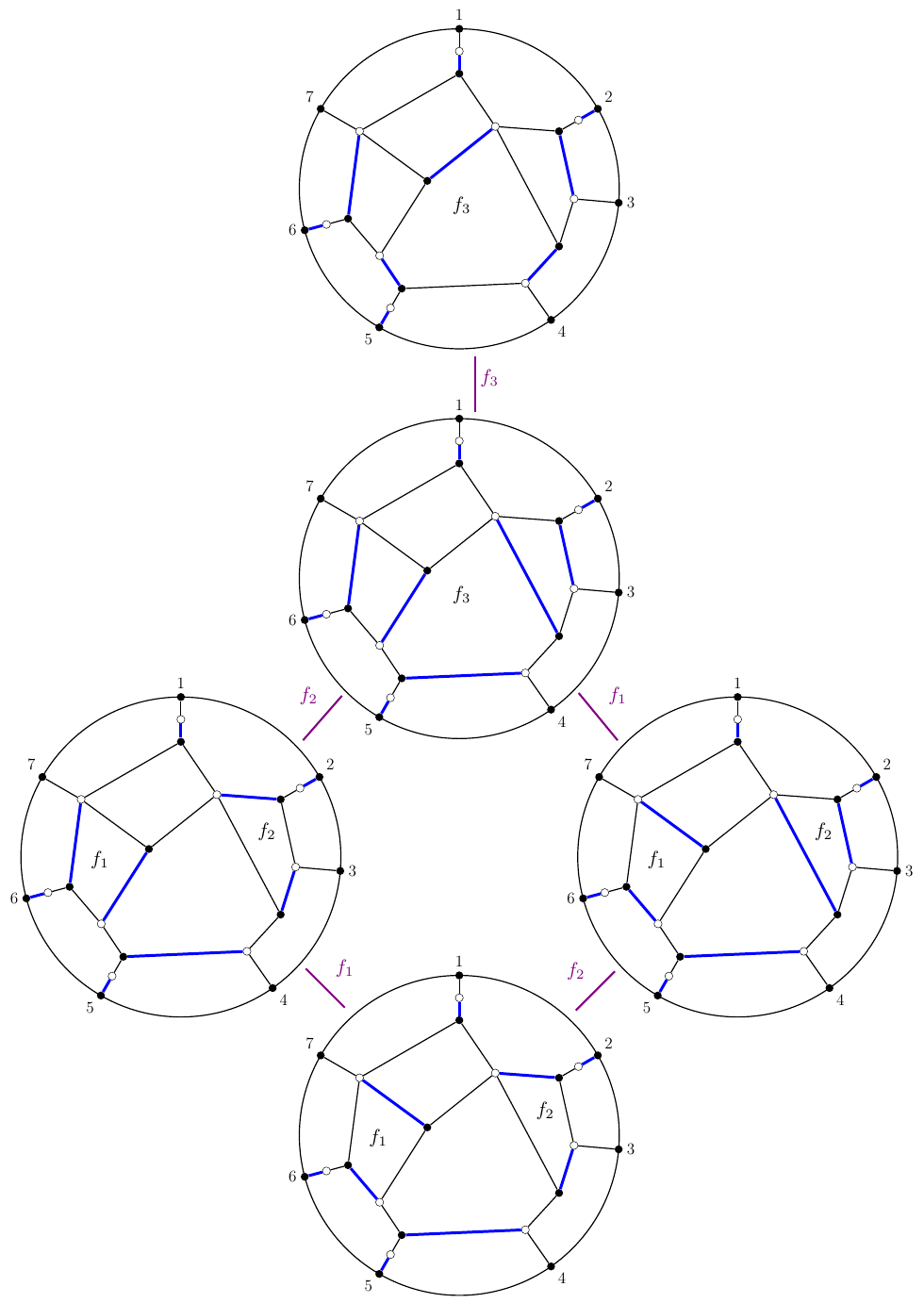}
\caption[Poset of Matchings]{The poset $(\mathcal{M}_{256},\le)$ on the matchings given in \Cref{fig:256_matchings}. The covering relations indicate the swiveled face.}
\label{fig:256_poset}
\end{center}
\end{figure}

\begin{definition}
\label{def:min_match}
Given $I \in \mathcal{P}$, $M_- \in \mathcal{M}_I$ is the \textbf{minimal matching} with boundary condition $I$ if for any $f \in F(G)$ for which $M_-$ can be swiveled, $M_-$ may only be upswiveled. Analogously, $M_+ \in \mathcal{M}_I$ is the \textbf{maximal matching} with boundary condition $I$ if for any $f\in F(G)$ for which $M_+$ can be swiveled, $M_+$ may only be downswiveled.
\end{definition}

We now give the constructions of $M_-$ and $M_+$ due to Muller--Speyer in terms of \textit{downstream wedges} and \textit{upstream wedges}, respectively, using the strands of $G$. Let $\tau$ and $\delta$ be strands of $G$ which cross at $e \in E(G)$. We say that $f \in F(G)$ is \textbf{downstream} from $e$ if $f$ lies in the region of $G$ between $\tau$ and $\delta$ \textit{after} their crossing at $e$. Similarly, we say $f \in F(G)$ is \textbf{upstream} from $e$ if $f$ lies in the region of $G$ between $\tau$ and $\delta$ \textit{before} their crossing at $e$. Then we define the \textbf{downstream wedge} of $e$ as the collection $\{f \in F(G)\, | \, f \text{ is downstream from } e\}$ and the \textbf{upstream wedge} of $e$ as the collection $\{f \in F(G)\, | \, f \text{ is downstream from } e\}$. Then for any $f \in F(G)$, Muller--Speyer define the following matchings {\cite[Theorem 5.3]{MSTwist}}: 
\[\overset{\rightarrow}{M}(f) = \{e \in E(G)\,|\, f \text{ is in the downstream wedge of }e\}\] 
and 
\[\overset{\leftarrow}{M}(f) = \{e \in E(G)\,|\, f \text{ is in the upstream wedge of }e\}.\]  
In particular, when $f \in \cev{F}(G)$, i.e. $G$ is source-labeled, then $\overset{\rightarrow}{M}(f)$ is \textit{minimal} with boundary condition read from the face label of $f$ {\cite[Proposition B.6]{MSTwist}}. Analogously, when $f \in \vec{F}(G)$, i.e. $G$ is target-labeled, then $\overset{\leftarrow}{M}(f)$ is \textit{maximal} with boundary condition read from the face label of $f$. 

\begin{remark}
\label{rem:extension}
These constructions in Muller-Speyer \cite{MSTwist} of extremal matchings $M_-$ and $M_+$ are limited to those whose boundary conditions must appear as a face label in either the source- or target-labeling of $G$, despite there typically existing several more boundary conditions of matchings possible on $G$. In particular, for a reduced plabic graph $G$ with decorated permutation $\tilde{\pi}$, there are only $k(n-k) - \ell(\tilde{\pi})+1$ such faces {\cite[Theorem 6.8]{OPS}}, where $\ell(\tilde{\pi})$ is the alignment number as in \Cref{def:alignments}. In the current work, we extend this result for all $I \in \mathcal{P}$, that is, for all possible boundary conditions of matchings on $G$ in \Cref{sec:heights}, of which there are possibly $\binom{n}{k}$.
\end{remark}

\section{Height Functions}
\label{sec:heights}

We build a combinatorial framework for matchings on plabic graphs $G$ via \textit{height functions}. Originally introduced in \cite{Thurston}, height functions provide a height assignment to $F(G)$ which leads to a simple bijection with matchings (up to small perturbations) as exploited in \cite{Propp}. In \Cref{ssec:heights}, we define a novel family of height functions called \textit{fundamental height functions} and methods for combining these fundamental heights to obtain new height functions, namely \textit{translated maximums} and \textit{translated minimums}. We conclude this first subsection in showing that translated maximums and translated minimums of fundamental heights gives minimal and maximal matchings, respectively. In \Cref{ssec:extreme_match}, we choose face-labeling conventions for $G$ so that we can obtain boundary conditions for our constructed extremal matchings. 
Moreover, we use the combinatorics of positroids $\mathcal{P}$ to obtain all possible boundary conditions for extremal matchings on $G$ parameterizing the associated cell $\Pi_\mathcal{P}$ from the construction of \Cref{ssec:heights}.

\subsection{Extremal matchings from height functions}
\label{ssec:heights}

We begin with a definition analogous to that given in {\cite[Definition 5]{Propp}}. In particular, Propp defined height functions to record probabalistic data of so-called $c$-orientations (circulations) on graph structures which are dual to our current setting.

\begin{definition}
\label{def:gen_heights}
    Given a planar bipartite graph $H$, let $\varphi:F(H) \to \RR$ be a globally defined function such that for each $v \in V(H)$, the restrictions $\varphi|_{\mathbf{f}_v}$ are almost monotonically decreasing (under a set cyclic orientation about all $v \in V(H)$). Then $\varphi$ is \textbf{height function}. 
\end{definition}

As in the preceding section, we restrict from planar bipartite graphs $H$ to plabic graphs $G$ with only requiring $\varphi$ restricts to almost monotonically decreasing functions at all \textit{internal} $v \in V$ under the set cyclic orientation. In \cite{Propp}, it was shown that every height function $\varphi$ is associated to a matching $M$ on $H$ by including $e \in M$ if $e$ separates faces $f_1, f_p \in F(H)$ such that $\varphi(f_p) > \varphi(f_1)$, where notation here is borrowed from \Cref{def:almost_monotonicity}. In the present work, we take the cyclic orientation of faces \textit{counterclockwise} about any $w \in W(H)$. This is equivalent to taking the cyclic orientation of faces \textit{clockwise} about any $b \in B(H)$. In the case of a plabic graph $G$, we only consider the cyclic orientation of faces about \textit{internal} $b \in B$. Aside from the following definition, we choose to examine the almost monotonicity of a height function at only white vertices. Thus, we define our novel family of height functions.

\begin{proposition}[Fundamental heights]
\label{prop:fund_heights}
Let $f \in F(G)$ have face label $I = \{i_1,\ldots,i_k\}$. Let $\varphi_j:F(G) \to \ZZ_{>0}$ for $0\le j < n$ be the mapping given as \[\varphi_j(f) = \Sumblank{\ell = 1}{k}i_\ell' \text{ where } i_\ell' = \begin{cases}
    i_\ell & \text{if } i_\ell > j \\
    i_\ell + n & \text{if } i_\ell \le j
\end{cases}.\] Then $\varphi_j$ is a height function and the collection $\displaystyle{\{\varphi_j\}}_{j=0}^{n-1}$ we call \textbf{fundamental height functions}. 
\end{proposition}

\begin{proof}
For any $w \in W$, applying \Cref{lem:common_face_indexing} as done in \Cref{ex:almost_monotone_functs} yields that $\varphi_0$ is a height function, i.e $\varphi_0$ restricts to an almost monotonically decreasing function at each $w \in W$. For $j > 0$, note that $\varphi_j$ is inducing the $(j+1)$-ordering on face label indices in $[n]$, so in particular \Cref{lem:common_face_indexing} also yields that $\varphi_j$ is a height function. 

Similarly, for any internal $b \in B$, \Cref{lem:common_face_indexing} yields that $\varphi_0$ is a height function, where now the cyclic orientation on faces is clockwise about $b$ since \textit{excluded} face label indices are almost monotonically decreasing counterclockwise about $b$ (see \Cref{ex:almost_monotone_functs}). Again recognizing that for $j > 0$, then $\varphi_j$ induces the $(j+1)$-ordering on face label indices in $[n]$, so $\varphi_j$ is also a height function as desired. 
\end{proof}

\begin{example}
   \label{ex:fund_heights}
On source-labeled $G_B$, we have the set of all fundamental height functions and their corresponding matchings given in Figure \ref{fig:fund_heights_source}. Similarly, on target-labeled $G_B$, we have the set of all fundamental height functions and their corresponding matchings given in Figure \ref{fig:fund_heights_targ}.

\begin{figure}[ht]
    \centering
    \includegraphics[width=1.5in]{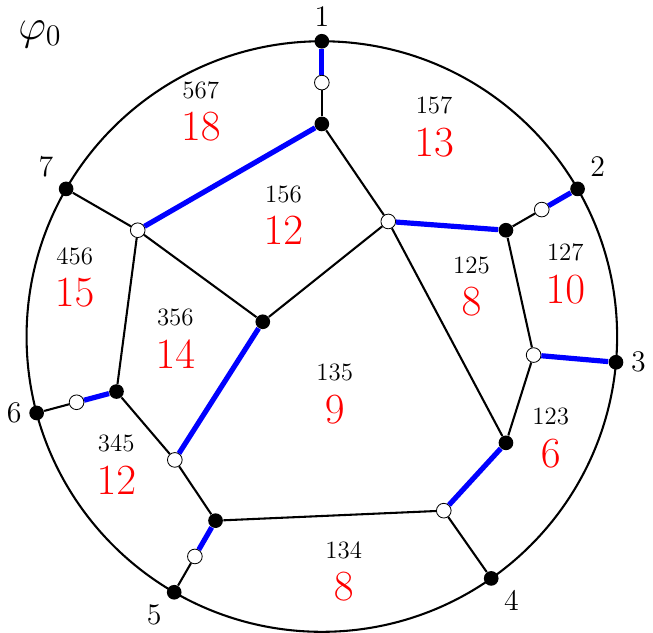} \hspace{.5em} \includegraphics[width=1.5in]{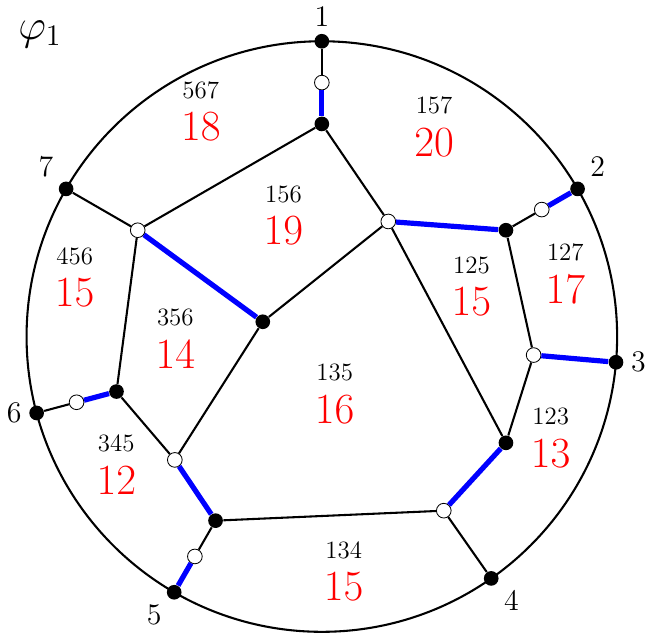} \hspace{.5em} \includegraphics[width=1.5in]{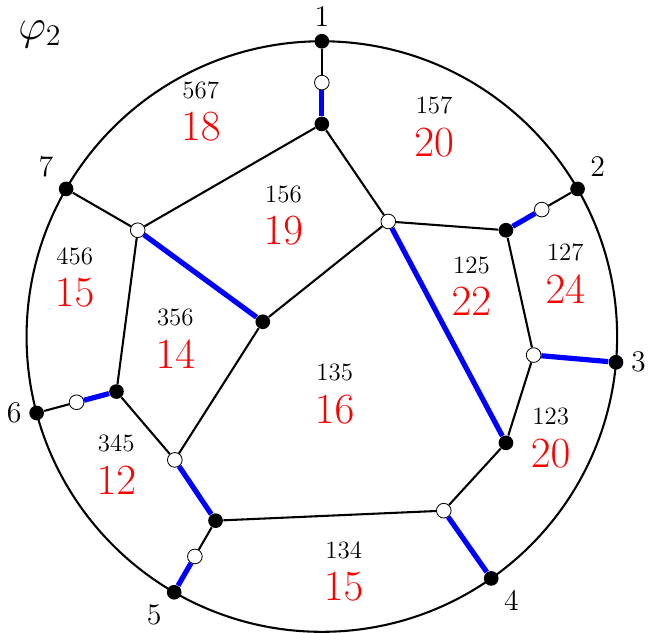} \hspace{.5em} \includegraphics[width=1.5in]{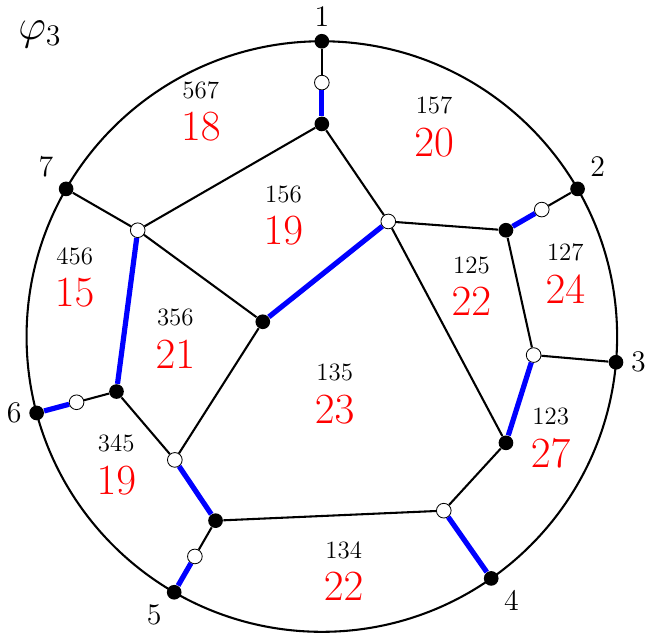} 
    \includegraphics[width=1.5in]{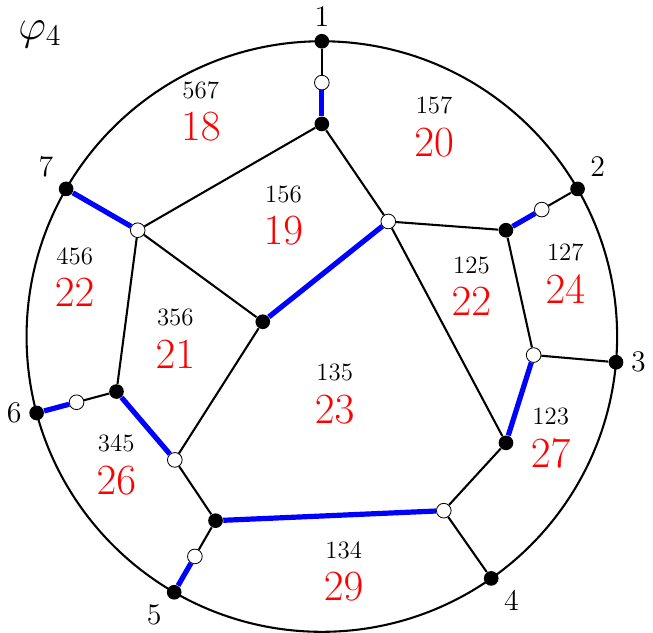} \hspace{1em} \includegraphics[width=1.5in]{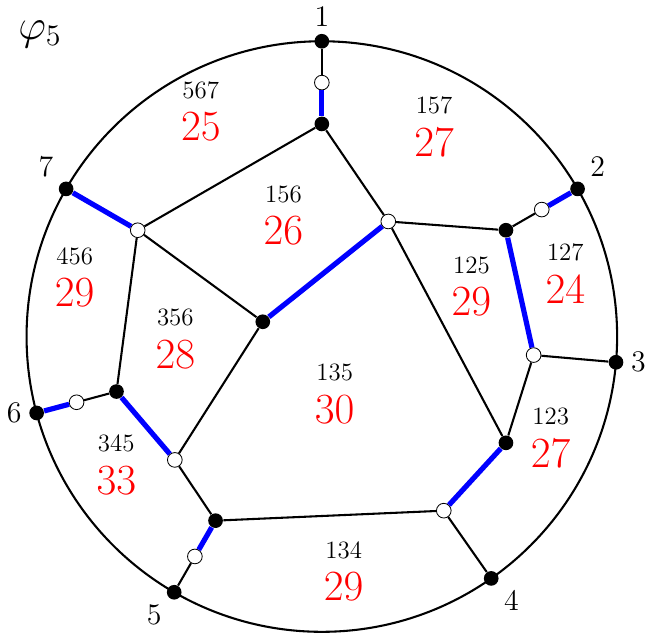} \hspace{1em} \includegraphics[width=1.5in]{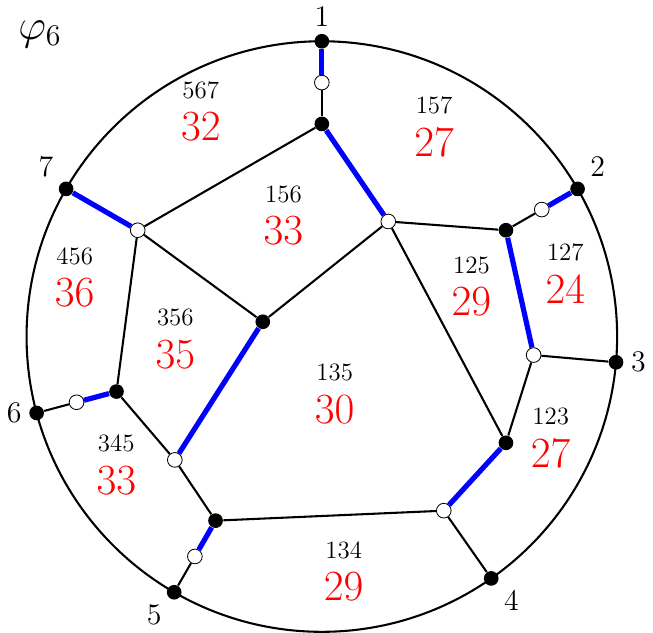}
    \caption{The set of all fundamental height functions on source-labeled $G_B$ with corresponding matching. Note that heights are given in red and face labels in black. Each $\varphi_j$ is listed above and left of its corresponding $G_B$.}
    \label{fig:fund_heights_source}
\end{figure}

\begin{figure}[ht]
    \centering
    \includegraphics[width=1.5in]{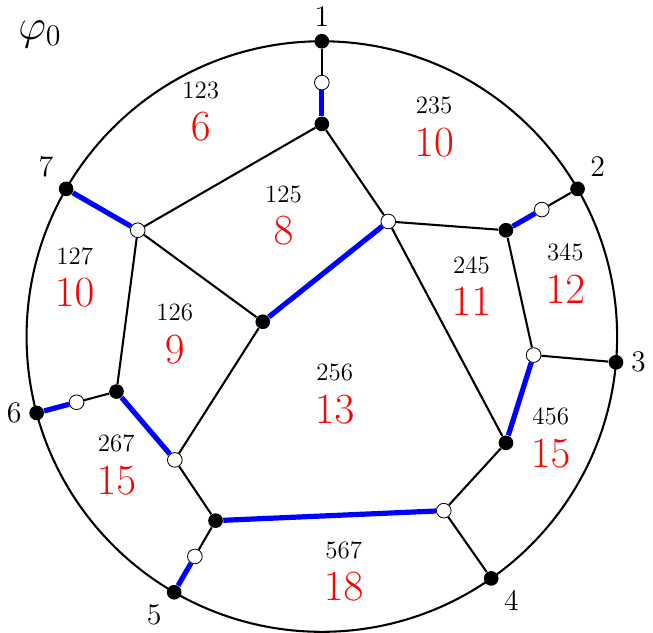} \hspace{.5em} \includegraphics[width=1.5in]{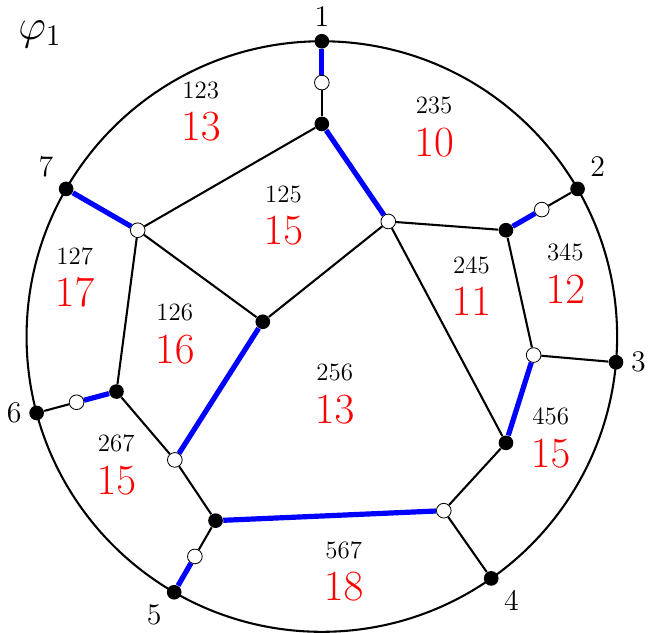} \hspace{.5em} \includegraphics[width=1.5in]{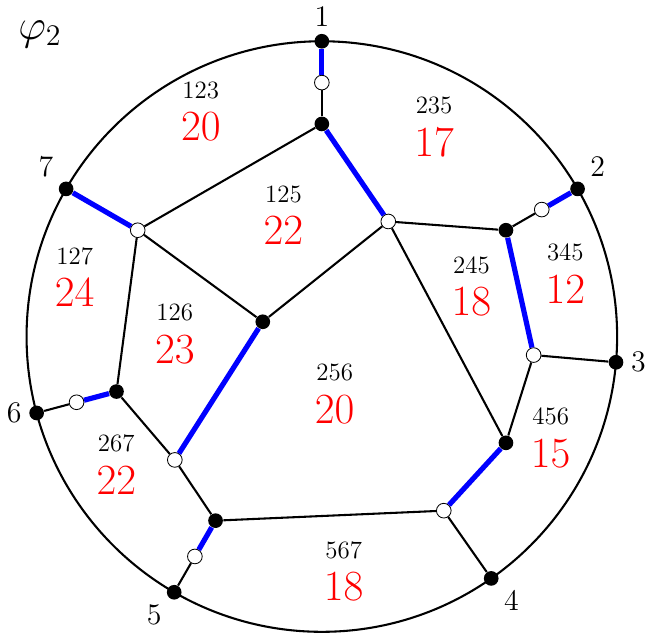} \hspace{.5em} \includegraphics[width=1.5in]{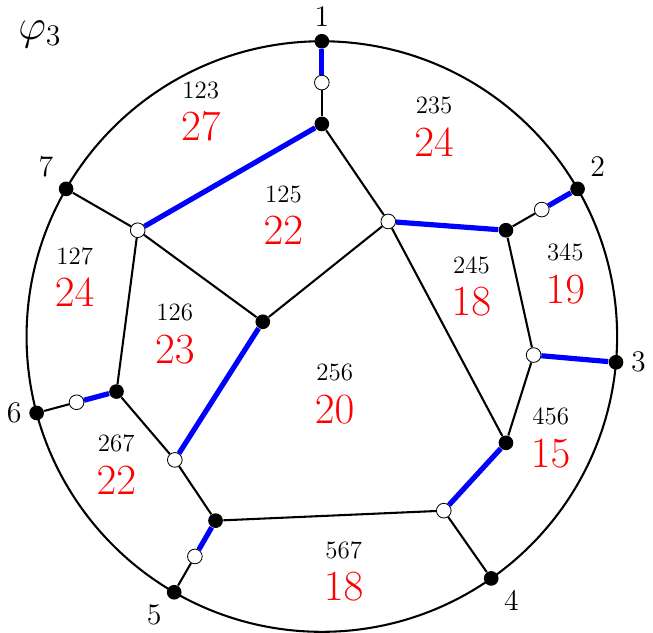} 
    \includegraphics[width=1.5in]{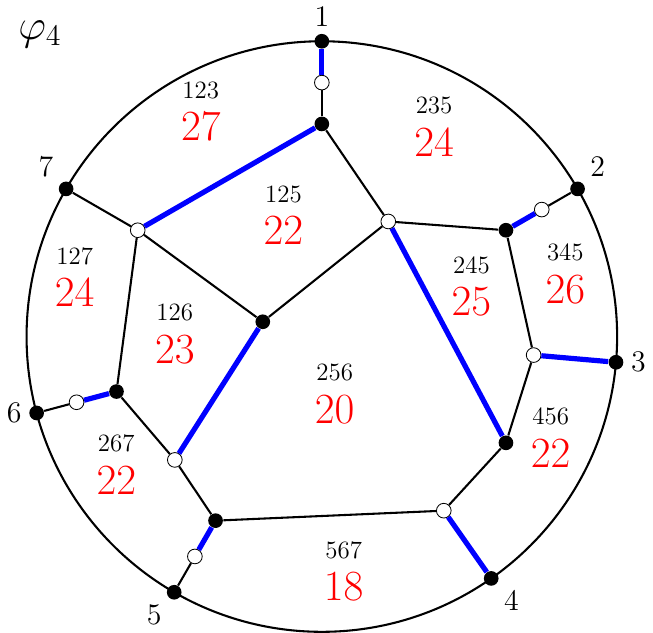} \hspace{1em} \includegraphics[width=1.5in]{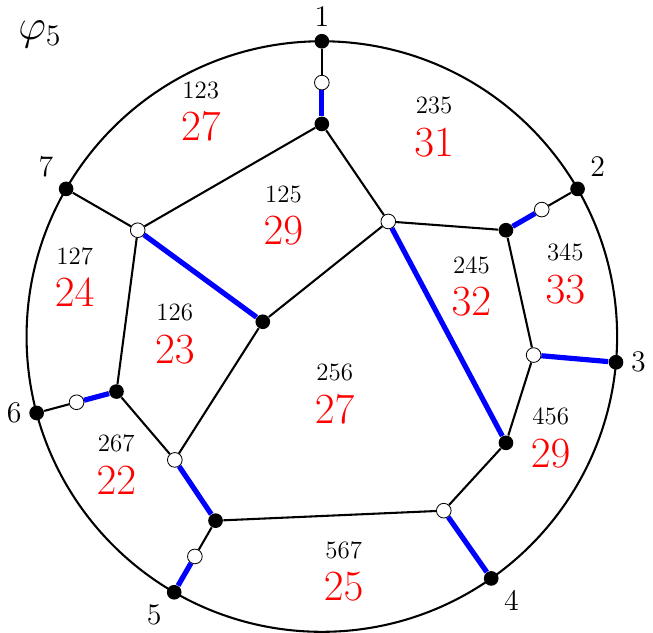} \hspace{1em} \includegraphics[width=1.5in]{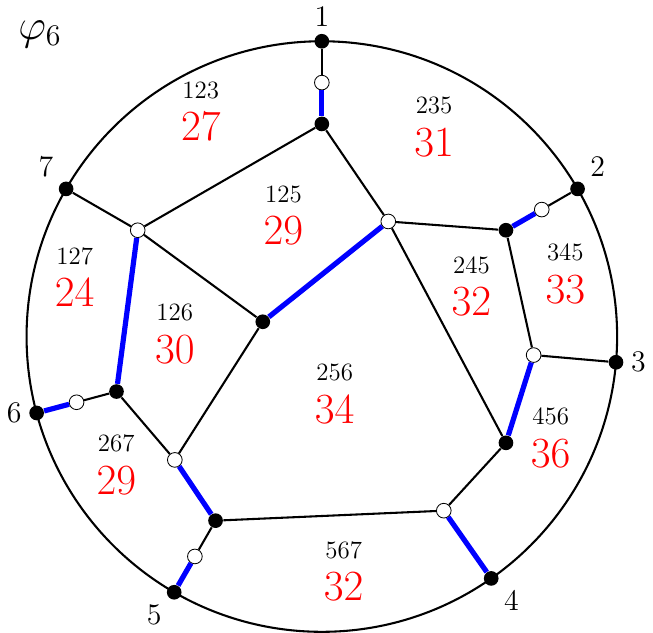}
    \caption{The set of all fundamental height functions on target-labeled $G_B$ with corresponding matching. Note that heights are given in red and face labels in black. Each $\varphi_j$ is listed above and left of its corresponding $G_B$.}
    \label{fig:fund_heights_targ}
\end{figure}
\end{example}

The correspondence of height functions with matchings is in fact a bijection up to the equivalence of height functions under monotonicity preserving transformations on faces about each internal $v \in V$ \cite{Propp}. Note that these monotonicity preserving transformations need not be uniform transformations over all faces of $G$, but rather could be small perturbations of values on some subset of $F(G)$. With the following definition, we reinforce this bijection with equivalence of height functions up to uniform, affine translations over all faces of $G$.

We revisit faces which are upstream and downstream from an edge $e \in E(G)$. Suppose $e = (b,w)$, for $w \in W$ and $b \in B$, is traversed by strands $i$ and $j$ such that $i$ enters $e$ at $w$ and $j$ enters $e$ at $b$. Suppose $f_i, f_j \in F(G)$ are the faces separated by $e$ such that $f_i$ contains the face label index $i$ and $f_j$ contains the face label index $j$. Given the proximity of $f_i$ and $f_j$ to $e$, we say that $f_i$ is \textit{directly} upstream from $e$ and $f_j$ is \textit{directly} downstream from $e$ {\cite[Section 5.1]{MSTwist}}. We give the following adjective for height functions.

\begin{definition}
\label{def:tight_2}
Let $e$ be the edge separating faces $f_i$ and $f_j$, with $f_i$ directly upstream and $f_j$ directly downstream from $e$, and $M$ the matching corresponding to  a height function $\varphi$. We say $\varphi$ is \textbf{tight} if $\varphi$ satisfies
\[ \varphi(f_i) - \varphi(f_j) = \begin{cases} i-j-n & \text{ if } e \not \in M \\ 
i-j & \text{ if } e \in M 
\end{cases} \qquad \text{for  $i > j$  and}
\]
\[ \varphi(f_i) - \varphi(f_j) = \begin{cases} i-j & \text{ if } e \not \in M \\ 
n + (i-j) & \text{ if } e \in M 
\end{cases} \qquad \text{for  $i < j$.}
\]
\end{definition}

In particular, we can reformulate the statement of the bijection between height functions and matchings as follows. 
\textit{Tight} height functions up to uniform, affine translations are in bijection with matchings on $G$. We can characterize the action of swiveling matchings in terms of tight height functions. 

\begin{proposition}
    \label{prop:swivel_action_2}
 Let $\varphi$ be a tight height function whose corresponding matching $M$ can be swiveled up at some face $f$ to obtain the matching $M'$. Then the tight height function corresponding to $M'$ is $\varphi'$, given by \[\varphi'(g) = \begin{cases}
   \varphi(g) & \text{ if } g \ne f \\
    \varphi(f) + n & \text{ if } g = f   
 \end{cases}. \]
\end{proposition}

\begin{proof}
For any $g \in F(G)$, fix orientation of edges of $g$ the clockwise about $g$. Observe that $g$ is directly downstream from all of its alternating edges and directly upstream from all of its complementary alternating edges. We now consider the $f \in F(G)$ such that $M$ can be upswiveled at $f$. Then $M$ contains all white-to-black edges of $f$ which coincides exactly with edges from which $f$ is upstream. Equivalently, $M$ does not contain any black-to-white edges of $f$ which are exactly the edges from which $f$ is downstream. Suppose $e,e' \in E(G)$ are a pair consecutive edges of $f$ in the clockwise orientation such that $e \in M$ and $e' \not \in M$, i.e. $f$ is directly upstream from $e$ and directly downstream from $e'$. Let $j$ be the unique strand traversing both $e$ and $e'$ so that $f$ lies to the right of $j$, let $g_j$ and $h_j \in F(G)$ be the faces separated from $f$ be $e$ and $e'$, respectively, and let $i_j$ and $\ell_j$ be the strands traversing $e$ and $e'$, respectively. Then since $\varphi$ is a tight height function, we have 
\begin{align*}
    \varphi(g_j) - \varphi(f) &= \begin{cases}
    j-i_j & \text{ if } j > i_j \\
    n + (j - i_j) & \text{ if } j < i_j
\end{cases}  && & \varphi(f) - \varphi(h_j) = \begin{cases}
\ell_j - j - n & \text{ if } \ell_j > j \\
\ell_j - j & \text{ if } \ell_j < j
\end{cases}.&
\end{align*}

Define $\varphi'$ as in the statement of \Cref{prop:swivel_action_2}. Then we observe
\begin{align*}
    \varphi'(g_j) - \varphi'(f) &= \varphi(g_j) - \varphi(f) - n && & \varphi'(f) - \varphi'(h_j) = \varphi(f) +n - \varphi(h_j)& \\
    &= \begin{cases}
        j - i_j - n & \text{ if } j > i_j \\
        j - i_j & \text{ if } j < i_j
    \end{cases}  && & =\begin{cases}
        \ell_j - j & \text{ if } \ell_j> j \\
        n + (\ell_j - j) & \text{ if } \ell_j < j
    \end{cases}&
\end{align*}
which implies the matching corresponding to $\varphi'$ contains $e'$ but does not contain $e$. Further, since choosing any such pair of consecutive edges $e$ and $e'$ about $f$ yields the same result, then $\varphi'$ is a tight height function and the corresponding matching to $\varphi'$, say $M'$, contains all black-to-white edges of $f$ and all edges of $M$ not incident to $f$. In particular, $\varphi'$ gives the matching $M'$ which is the result of an upswivel of $M$ at $f$ as desired. 
\end{proof}

We can analogously characterize the action of swiveling a matching down at a face $f$ on a tight height function $\varphi$ to result in a new tight height function $\varphi'$ such that $\varphi'(f) = \varphi(f) - n$ and is unchanged at all other faces. In characterizing the action of swiveling matchings up and down in terms of the values achieved by tight height functions, we demonstrate in the next proposition the uniqueness of each of the fundamental heights $\varphi_j$. Moreover, this yields that the matching corresponding to each of the $\varphi_j$ is $M_-$ and $M_+$ in its respective lattice of matchings.

\begin{proposition}
\label{prop:fund_is_tight_unique} 
The fundamental height function $\varphi_j$ is tight and its corresponding matching is unique in its lattice of matchings. 
\end{proposition}

\begin{proof}
Note that $\varphi_j$ induces a $(j+1)$-ordering on $[n]$. One can verify that $\varphi_j$ is a tight height function by exhaustion of cases at each $e \in E(G)$ traversed by strands $i$ and $\ell$. Specifically, one can verify that the edges included in and excluded from the matching corresponding to $\varphi_j$ satisfy those conditions set in \Cref{def:tight_2} under study of the cases $j < i < \ell$ (equivalently $i < \ell \le j$), $i \le j < \ell$, $j < \ell < i$ (equivalently $\ell < i \le j$), and $\ell \le j < i$. 

Now suppose by way of contradiction, we can swivel the matching corresponding to $\varphi_j$ at some face $f$. Suppose this swivel at $f$ is an upswivel, meaning for all $g \in F(G)$ neighboring $f$, we have $\varphi(f) < \varphi(g)$. Invoking \Cref{lem:common_face_indexing}, about each $w \in W$ incident to $f$ there is exactly one face label entry which almost monotonically decreases under the $(j+1)$-ordering when traversing faces counterclockwise about $w$. Suppose $f$ is a square with $f = Sac$ for $|S| = k-2$ such that the four neighboring faces of $f$ have face labels $Sab, Sad, Scd, Sbc$ in a clockwise order about $f$. Since we may swivel this matching up at $f$, we obtain the pair of inequalities on the indices $a,b,c,d$, namely, $c <_{j+1} b <_{j+1} d$ and $a <_{j+1} d <_{j+1} b$ which yields a contradiction. Inducting on the number of white vertices bordering $f$, say $m$, we obtain a system of $m$ 3-term inequalities which yield analogous contradictions. Now if the swivel at $f$ is a downswivel, we reverse the direction of inequalities on face label entries and obtain the same contradiction. Thus, the matching corresponding to $\varphi_j$ is unique.
\end{proof}

The previous claim allows us to characterize tight height functions as those which are built from piecing together fundamental height functions. We now study some behavior of matchings when varying by subsequent fundamental height functions. We obtain the following. 

\begin{lemma}
    \label{lem:lollis}
Suppose $G$ has a lollipop at boundary vertex $j$. Then $\varphi_{j-1}$ and $\varphi_j$ induce the same matching on $G$. 
\end{lemma}

\begin{proof}
    Suppose we have a white lollipop at $j$. Then every $f \in F(G)$ is labeled with the index $j$, and in particular, $\varphi_{j}(f) = \varphi_{j-1}(f) + n$ for all $f \in F(G)$. As this is a uniform affine transformation, the matchings are the same. Now suppose we have a black lollipop at $j$. Then none of $f \in F(G)$ contains the face label index $j$, and in particular, $\varphi_{j}(f) = \varphi_{j-1}(f)$ for all $f \in F(G)$ which induce the same matching on $G$. 
\end{proof}

Essentially, \Cref{lem:lollis} yields that we may develop our theory of height functions and matchings while ignoring any lollipops of $G$. Moreover, one may recognize from the characterization of any positroid $\mathcal{P}$ as possible boundary conditions of matchings in \Cref{ssec:boundary_meas_map} and that a white lollipop at $j$ implies $j \in I$ for all $I \in \mathcal{P}$, and similarly, that a black lollipop at $j$ implies $j \not\in I$ for all $I \in \mathcal{P}$, then any lollipop either \textit{always} contributes to the boundary condition of a matching or \textit{never} contributes to the boundary condition of a matching. Thus, we suppose without any loss of generality that $G$ does not have any lollipops.

\begin{lemma}
\label{lem:passing_lemma}
    The matching corresponding to $\varphi_{j-1}$ includes all black-to-white edges on the $j$-strand in the orientation of the $j$-strand. In passing to the matching corresponding to $\varphi_j$, this matching includes all white-to-black edges on the $j$-strand in the orientation of the $j$-strand, with all other edges not contained on the $j$-strand from the matching corresponding to $\varphi_{j-1}$ unchanged.
\end{lemma}

\begin{proof}
Recall that for any $j \in [n]$, $\varphi_j$ sets the $(j+1)$-ordering on $[n]$. Now for any consecutive pair $j-1, j \in [n]$, the $j$- and $(j+1)$-orderings of $[n]$ on face label entries is equivalent to cycling the boundary vertex enumeration clockwise and fixing the internal graph structure of $G$ (up to relabeling faces), i.e. $i \mapsto i+1 \mod n$ for all $i \in [n]$. Thus, it  is sufficient to prove this claim for the case $j =1$.

We first consider the matching corresponding to $\varphi_0$ along the $1$-strand. In particular, at any $w \in W$ along path traversed by the $1$-strand, there is exactly one face incident to $w$ which contains the face label entry 1 by \Cref{lem:common_face_indexing}. Since $\varphi_0$ induces the $1$-ordering on $[n]$, then the black-to-white edge in the orientation of the $1$-strand is included in the matching corresponding to $\varphi_0$. Now consider the matching corresponding to $\varphi_1$ along the $1$-strand. Since $\varphi_1$ induces the $2$-ordering on $[n]$, i.e. $1$ is $2$-maximal, then the white-to-black edge in the orientation of the $1$-strand is included in the matching corresponding to $\varphi_1$. Lastly consider the edges of the matching corresponding to $\varphi_0$ and $\varphi_1$ which do not lie on the $1$-strand. Notice that all faces incident to any fixed $w' \in W$ which do not lie on the path traversed by the $1$-strand either contain the face label index $1$ if $w'$ lies to the left of this path or do not contain the face label index $1$ if $w'$ lies to the right of this path. In particular, changing from the $1$-ordering on $[n]$ to the $2$-ordering on $[n]$ does not impact the almost monotonicity of faces about any such $w'$, so also the edges of the matching corresponding to $\varphi_0$ are the same as the edges of the matching corresponding to $\varphi_1$ when not along the $1$-strand.    
\end{proof}

We will forgo placing boundary conditions on the matching corresponding to each $\varphi_j$ until we have fixed a labeling on $G$. Instead, we define some basic operations we can perform with height functions. We use $\varphi_j$ to denote our fundamental height functions and $h_i$ to denote any arbitrary height function in the following. 

\begin{definition}
\label{def:translation}
    Let $\varphi_j$ be a fundamental height function. Then the \textbf{translation} of $\varphi_j$ by $a$ is the function $\varphi_{j}-a$ given by $(\varphi_{j}-a)(f) = \varphi_j(f) - a$ for all $f \in F(G)$. 
\end{definition}

\begin{definition}
\label{def:maximum}
    Given an ordered collection $S = \{h_1,h_2,\ldots, h_p\}$ of height functions, the \textbf{maximum} of $S$ is denoted $\max\left\{ h_1,h_2, \ldots, h_p\right\}$ which we evaluate on faces $f \in F(G)$ as \[\max\left\{ h_1,h_2, \ldots ,h_p\right\}(f) = \max\left\{ h_1(f),h_2(f), \ldots, h_p(f)\right\}.\]
\end{definition}

\begin{definition}
\label{def:minimum}
    Given an ordered collection $S = \{h_1,h_2,\ldots, h_p\}$ of height functions, the \textbf{minimum} of $S$ is denoted $\min\left\{ h_1,h_2, \ldots, h_p\right\}$ which we evaluate on faces $f \in F(G)$ as \[\min\left\{ h_1,h_2, \ldots, h_p\right\}(f) = \min\left\{ h_1(f),h_2(f), \ldots, h_p(f)\right\}.\]
\end{definition}

We will use the translation and maximum operations to meaningfully combine fundamental height functions into new height functions we call translated maximums. Similarly, we will use the translation and minimum operations to combine fundamental height functions into new height functions we call translated minimums. More precisely, given the input of an ordered collection of fundamental height functions $\varphi_{j_1}, \varphi_{j_2}, \ldots, \varphi_{j_p}$, we translate the $i$th fundamental height by $(i-1)n$ and then consider the corresponding maximum or minimum. One can quickly verify that if our list of fundamental height functions has length greater than $k$, then there are fundamental heights which will not contribute to the translated maximum or minimum, and so we choose that our list of fundamental heights is exactly $k$. We codify this construction in the following definition.

\begin{definition}
    \label{def:tmax_tmin}
Given the ordered collection $S = \{\varphi_{j_1}, \varphi_{j_2}, \ldots \varphi_{j_k}\}$ of fundamental height functions, the \textbf{translated maximum} of $S$ is 
\[\tmax{\varphi_{j_1},\varphi_{j_2},\ldots,\varphi_{j_k}} = \max\left\{ \varphi_{j_1}, (\varphi_{j_2}-n), \ldots, (\varphi_{j_k} - (k-1)n) \right\}.\]

Analogously, the \textbf{translated minimum} of $S$ is \[\tmin{\varphi_{j_1},\varphi_{j_2},\ldots,\varphi_{j_k}} = \min\left\{ \varphi_{j_1}, (\varphi_{j_2}-n), \ldots, (\varphi_{j_k} - (k-1)n) \right\}.\]
\end{definition}

We give an equivalent characterization of translated maximums and translated minimums with more concise notation in the following.

\begin{definition}
    \label{def:Tmax_Tmin}
Given the ordered collection $S = \{\varphi_{j_1}, \varphi_{j_2}, \ldots \varphi_{j_p}\}$ of fundamental height functions and the increasing sequence $1 \le t_1 < t_2 < \cdots < t_{p-1} < k$, the \textbf{compressed translated maximum} of $S$ is \[\TMAX{\varphi_{j_1}, \varphi_{j_2}(a_1), \ldots, \varphi_{j_p}(a_{p-1})} = \max\left\{ \varphi_{j_1}, (\varphi_{j_2}-t_1n), \ldots, (\varphi_{j_p} - t_{p-1}n) \right\}\] where $a_i = t_i - t_{i-1}$ for all $i \ge 2$ and $a_1 = t_1$. 

Similarly, the \textbf{compressed translated minimum} of $S$ is
    \[\TMIN{\varphi_{j_1}, \varphi_{j_2}(a_1), \ldots, \varphi_{j_p}(a_{p-1})} = \min\left\{ \varphi_{j_1}, (\varphi_{j_2}-t_1n), \ldots, (\varphi_{j_p} - t_{p-1}n) \right\}\] where $a_i = t_i - t_{i-1}$ for all $i \ge 2$ and $a_1 = t_1$.     
\end{definition}

\begin{remark}
\label{rem:Ts_equal_ts}
    In the special case where $k=p$ and $t_i = i$ for all $i$, hence $a_i = 1$ for all $i$, the compressed translated maximum simplifies as
    \[\TMAX{\varphi_{j_1}, \varphi_{j_2}(1), \ldots, \varphi_{j_k}(1)} = \tmax{\varphi_{j_1},\varphi_{j_2},\ldots,\varphi_{j_k}}.\]

    Similarly, \[\TMAX{\varphi_{j_1}, \varphi_{j_2}(1), \ldots, \varphi_{j_k}(1)} = \tmin{\varphi_{j_1},\varphi_{j_2},\ldots,\varphi_{j_k}}.\]

    On the other hand, if some of the $a_i$'s are greater than $1$, see Lemmas \ref{lem:tmax_k_funds} and \ref{lem:tmin_k_funds} on the translation between the compressed and uncompressed forms.
\end{remark}

The two notations for translated maximums and translated minimums (in \Cref{def:tmax_tmin}) and their compressed forms (in \Cref{def:Tmax_Tmin}), repsectively,  is due to the utility of considering these constructions in two separate ways. With respect to notation in \Cref{def:Tmax_Tmin}, one can express the boundary conditions of the matchings corresponding to translated maximums and minimums (c.f. \Cref{thm:main} and \ref{thm:main_restate}). With respect to the notation of \Cref{def:tmax_tmin}, one can easily verify they constructed all possible translated maximums and minimums based on the choice of $k$ fundamental height functions (c.f. \Cref{lem:tmax_k_funds} and \ref{lem:tmin_k_funds}).

In \Cref{def:Tmax_Tmin} we refer to the $a_i$ as (consecutive) \textbf{relative translations} of consecutive fundamental heights in some translated maximum or translated minimum. This is because $a_i$ determines the translation of $\varphi_{j_{i+1}}$ by $t_i n$ relative to the translation of $\varphi_{j_i}$ by $t_{i-1}n$ in the translated maximum or minimum. For nonconsecutive fundamental heights in some translated maximum or translated minimum, say $\varphi_{j_\ell}$ and $\varphi_{j_m}$, we say their relative translation is the sum $r_{\ell,m} = \SumBlank{\ell \le i < m} a_i = t_{m-1} - t_{\ell-1}$. Note that in bounding the largest translation factor $t_{p-1} < k$, we ensure $(\varphi_{j_p} - t_{p-1}n)$ may contribute to its corresponding translated maximum and $\varphi_{j_1}$ may contribute to its corresponding translated minimum. The increasing property on $t_i$ yields that $a_i \ge 1$ for all $i$, and the upper bound on $t_{p-1} < k$ yields that $\sum_i a_i < k$. 

\begin{definition}
\label{def:rel_trans}
Given a choice of positive integers $(a_1,a_2,\dots, a_{p-1})$ as above, let $a_p$ be the unique value such that $\Sumblank{i=1}{p}a_i = k$. Then we say a translated maximum or translated minimum with consecutive relative translations given by $a_i$ has relative translations of type $\vec{a} = (a_1,\ldots, a_p)$.
\end{definition}

Let $\TTMAX:=\TMAX{\varphi_{j_1}, \varphi_{j_2}(a_1), \ldots, \varphi_{j_p}(a_{p-1})}$ and $\TTMIN:= \TMIN{\varphi_{j_1}, \varphi_{j_2}(a_1), \ldots, \varphi_{j_p}(a_{p-1})}$ where $\varphi_{j_1},\varphi_{j_2}, \ldots, \varphi_{j_p}$ are an arbitrary ordered collection of fundamental height functions and their arbitrary relative translations of type $\vec{a}$. We claim that $\TTMAX$ and $\TTMIN$ result in tight height functions and further that $\TTMAX$ will correspond to minimal matchings and $\TTMIN$ will correspond to maximal matchings. Let $\ell < m$, so also the indices of fundamental heights give $j_\ell < j_m$, where $\varphi_{j_\ell}$ and $\varphi_{j_m}$ are some pair appearing in $\TTMAX$ or $\TTMIN$. Let $H = [j_m]\setminus [j_\ell]$ and let $f \in F(G)$ have the face label $I$. From above, the relative translation of $\varphi_{j_\ell}$ and $\varphi_{j_m}$ is $r_{\ell,m} = t_{m-1} - t_{\ell - 1}$. Then the following condition on $f$ determines whether $\varphi_{j_\ell}$ or $\varphi_{j_m}$ contributes to the resulting height at $f$ under the following maximum or minimums, namely we observe: 
\[
\max\left\{(\varphi_{j_\ell}-t_{\ell-1}n), (\varphi_{j_m} - t_{m-1}n) \right\}(f) = \begin{cases}
   (\varphi_{j_\ell} - t_{\ell-1}n)(f) & \text{ if } |I \cap H| \le r_{\ell,m} \\
   (\varphi_{j_m} - t_{m-1}n)(f) & \text{ if } |I \cap H| \ge r_{\ell,m}
\end{cases}
\] and \[
\min\left\{(\varphi_{j_\ell}-t_{\ell-1}n), (\varphi_{j_m} - t_{m-1}n) \right\}(f) = \begin{cases}
   (\varphi_{j_\ell} - t_{\ell-1}n)(f) & \text{ if } |I \cap H| \ge r_{\ell,m} \\
   (\varphi_{j_m} - t_{m-1}n)(f) & \text{ if } |I \cap H| \le r_{\ell,m}
\end{cases}.
\] 
Note that these particular maximums and minimums are the resulting maximums and minimums we see if restricting $\TTMAX$ and $\TTMIN$ to being taken over only $\varphi_{j_\ell}$ and $\varphi_{j_m}$ and upholding their relative translations. We give the following definition which corresponds to the case when $|I \cap H| = r_{\ell,m}$. 

\begin{figure}
    \centering
    \includegraphics[height=1.8in]{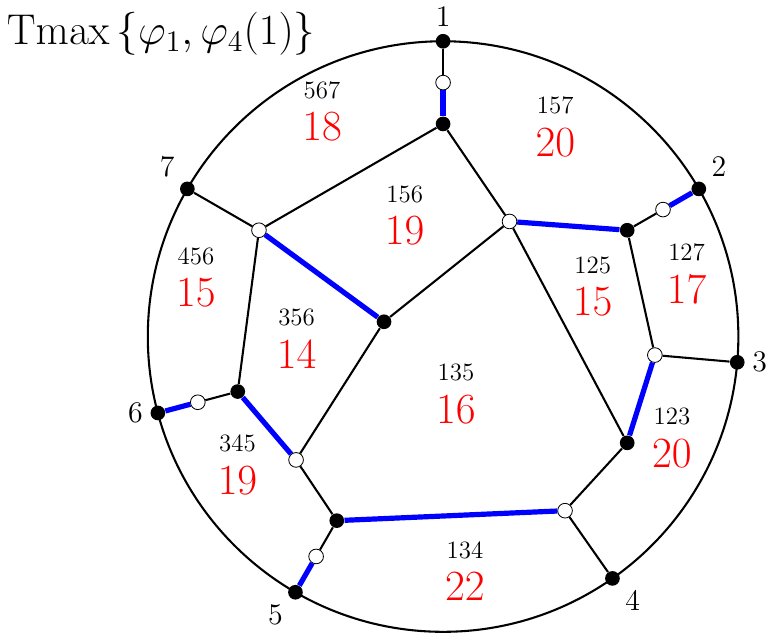} \hspace{5em} \includegraphics[height=1.8in]{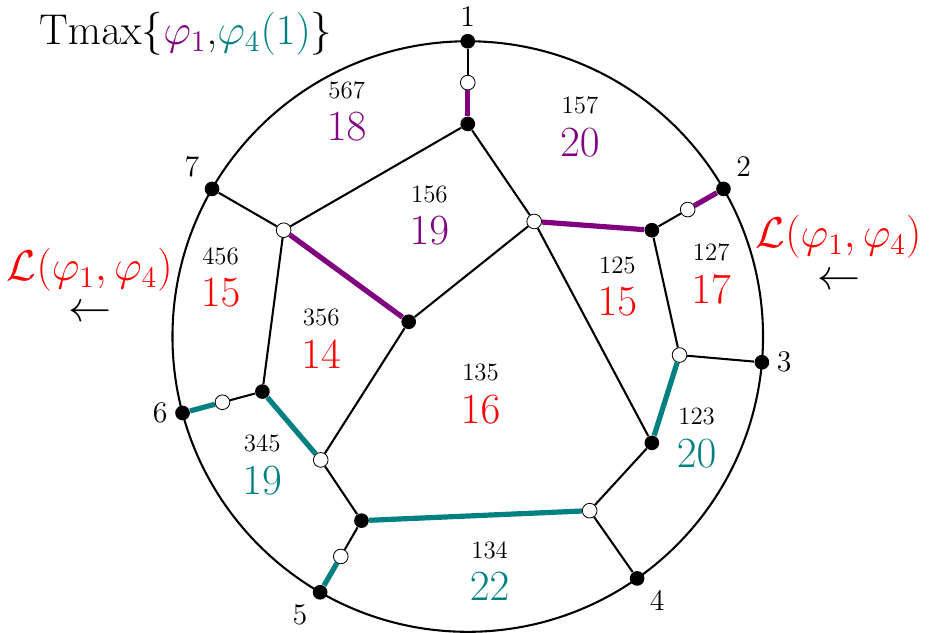}
    \caption{(Left): $\TMAX{\varphi_1,\varphi_4(1)} = \tmax{\varphi_1,\varphi_4,\varphi_5}$ on source-labeled $G_B$; (Right): The land bridge $\mathcal{L}(\varphi_1,\varphi_4)$ (in {\color{red}red}) with the heights and matched edges corresponding to $\varphi_1$ and $\varphi_4$ in {\color{violet}violet} and {\color{teal}teal}, at the top and bottom of $G_B$, respectively.
    }
    \label{fig:tmax14_example}
\end{figure}

\begin{definition}[Land Bridge]
    \label{def:land_bridge_2}
Let $\ell < m$, i.e. $j_\ell < j_m$, and let $\ell' = t_{\ell -1}n$ and $m' = t_{m-1}n$. Then the set \[\mathcal{L}(\varphi_{j_\ell}-\ell',\varphi_{j_m}-m') = \left\{ f \in F(G) \,| \, (\varphi_{j_\ell}-\ell')(f) = (\varphi_{j_m}-m')(f) \right\}
\] is called the \textbf{land bridge} with respect to $\varphi_{j_\ell} - \ell'$ and $\varphi_{j_m} - m'$.
\end{definition}

\begin{figure}[h]
    \centering
    \includegraphics[height=1.8in]{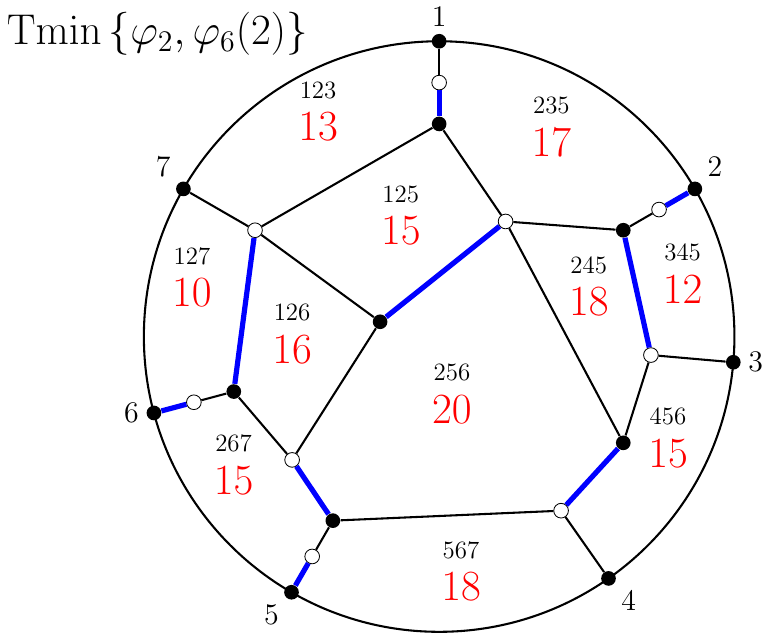} \hspace{5em} \includegraphics[width=2.25in]{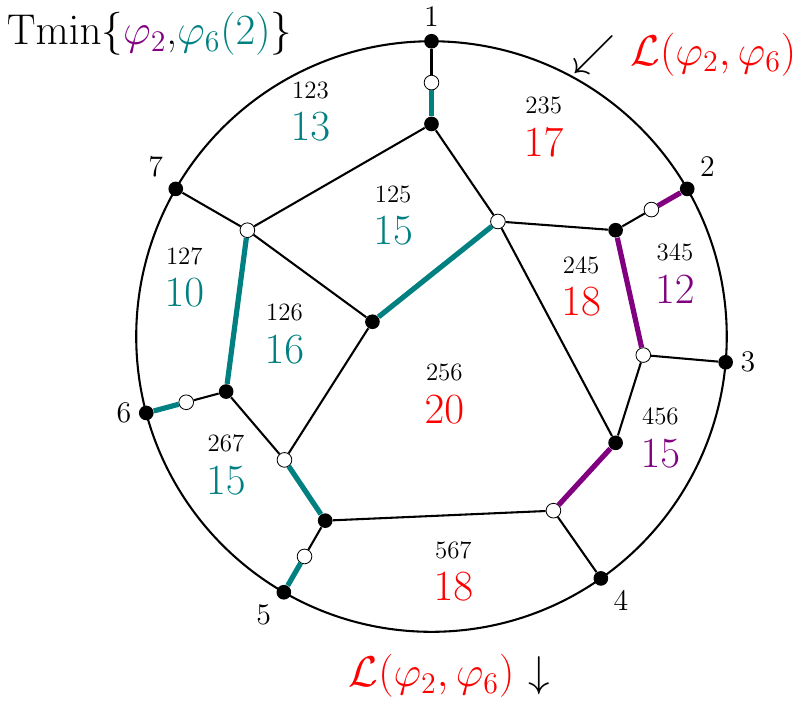}
    \caption{(Left): $\TMIN{\varphi_2,\varphi_6(2)} = \tmin{\varphi_2,\varphi_5,\varphi_6}$ on target-labeled $G_B$; (Right): The land bridge $\mathcal{L}(\varphi_2,\varphi_6)$ (in {\color{red}red}) with the heights and matched edges corresponding to $\varphi_2$ and $\varphi_6$ in {\color{violet}violet} and {\color{teal}teal}, at the bottom right and top left of $G_B$, respectively.}
    \label{fig:tmin26_example}
\end{figure}

When considering land bridges with respect to $\varphi_{j_\ell} - \ell'$ and $\varphi_{j_m} - m'$ when $\varphi_{j_\ell}$ and $\varphi_{j_m}$ appear in $\TTMAX$ or $\TTMIN$, we write this land bridge as $\mathcal{L}(\varphi_{j_\ell},\varphi_{j_m})$ where the notation omits the translations by $\ell'$ and $m'$ but the relative translation $r_{\ell,m}$ is implicit. Suppose $\varphi_{j_\ell}$ and $\varphi_{j_m}$ appear in $\TTMAX$ or in $\TTMIN$ and let their relative translation $r_{\ell,m}$ be given as in the paragraphs preceding. Let $H = [j_m]\setminus[j_\ell]$ and let $f \in F(G)$ have face label $I$. Since $I$ is determined by strands of $G$, then we can reformulate the condition $f \in \mathcal{L}(\varphi_{j_\ell},\varphi_{j_m})$ if and only if $|I \cap H| = r_{\ell,m}$ of \Cref{def:land_bridge_2} as follows. For any index $i \in H$, we say the $i$-strand is a $\mathbf{H}$\textbf{-strand}. We define the $\mathbf{a}$\textbf{th rightmost path} of $H$-strands to be the unique oriented path along all $H$-strands such taht  at any point along this path there are exactly $(a-1)$ of the $H$-strands lying to its right. Now for any $f \in \mathcal{L}(\varphi_{j_\ell},\varphi_{j_m})$, then $f$ lies to the left of exactly $r_{\ell,m}$ of the $H$-strands, that is, $f$ lies between the $r_{\ell.m}$th and $(r_{\ell,m}+1)$th rightmost paths of $H$-strands. Then in the orientation of the $r_{\ell,m}$th and $(r_{\ell,m}+1)$th rightmost paths of $H$-strands, it follows that $\mathcal{L}(\varphi_{j_\ell},\varphi_{j_m})$ is right-bordered by the $r_{\ell,m}$th rightmost path of $H$-strands and left-bordered by the $(r_{\ell,m}+1)$th rightmost path of $H$-strands. Moreover, the orientation of these left- and right-borders endow $\mathcal{L}(\varphi_{j_\ell},\varphi_{j_m})$ with an orientation.

We revisit whether $\varphi_{j_\ell}$ or $\varphi_{j_m}$ contributes to the resulting height at $f \in F(G)$ when restricting $\TTMAX$ and $\TTMIN$ to being taken over only $\varphi_{j_\ell}$ and $\varphi_{j_m}$ while upholding their relative translations. In particular, the prior condition on $f$ can be determined in terms of $f$ relative to the land bridge $\mathcal{L}(\varphi_{j_\ell},\varphi_{j_m})$. We retain the use of $\ell' = t_{\ell -1}n$ and $m' = t_{m-1}n$ in the following. Then we observe

\[
\max\left\{(\varphi_{j_\ell}-\ell'), (\varphi_{j_m} - m') \right\}(f) = \begin{cases}
   (\varphi_{j_\ell} - \ell')(f) & \text{ for $f$ to the right of $\mathcal{L}(\varphi_{j_\ell},\varphi_{j_m})$} \\
   (\varphi_{j_\ell} - \ell')(f) = (\varphi_{j_m} - m')(f) & \text{ for $f \in \mathcal{L}(\varphi_{j_\ell},\varphi_{j_m})$} \\
   (\varphi_{j_m} - m')(f) & \text{ for $f$ to the left of $\mathcal{L}(\varphi_{j_\ell},\varphi_{j_m})$}
\end{cases}
\] and \[
\min\left\{(\varphi_{j_\ell}-\ell'), (\varphi_{j_m} - m') \right\}(f) = \begin{cases}
   (\varphi_{j_\ell} - \ell')(f) & \text{ for $f$ to the left of $\mathcal{L}(\varphi_{j_\ell},\varphi_{j_m})$} \\
   (\varphi_{j_\ell} - \ell')(f) = (\varphi_{j_m} - m')(f) & \text{ for $f \in \mathcal{L}(\varphi_{j_\ell},\varphi_{j_m})$} \\
   (\varphi_{j_m} - m')(f) & \text{ for $f$ to the right of $\mathcal{L}(\varphi_{j_\ell},\varphi_{j_m})$}
\end{cases}.
\] 

This, in particular, yields for all choices $0 \le j_1 < j_2 < n$ and $\vec{a} = (a_1,a_2)$ that $\TMAX{\varphi_{j_1},\varphi_{j_2}(a_1)}$ and $\TMIN{\varphi_{j_1},\varphi_{j_2}(a_1)}$ are tight height functions. Moreover, applying \Cref{lem:passing_lemma}, the matching corresponding to $\TMAX{\varphi_{j_1},\varphi_{j_2}(a_1)}$ contains all black-to-white edges on the right-border of $\mathcal{L}(\varphi_{j_1},\varphi_{j_2})$ and all white-to-black edges on the left-border of $\mathcal{L}(\varphi_{j_1},\varphi_{j_2})$. Analogously, applying \Cref{lem:passing_lemma}, the matching corresponding to $\TMIN{\varphi_{j_1},\varphi_{j_2}(a_1)}$ contains all white-to-black edges on the right-border of $\mathcal{L}(\varphi_{j_1},\varphi_{j_2})$ and all black-to-white edges on the left-border of $\mathcal{L}(\varphi_{j_1},\varphi_{j_2})$. \Cref{fig:tmax14_example} and \Cref{fig:tmin26_example} gives examples of a $\TTMAX$ and $\TTMIN$ of fundamental heights as well as distinguishing their land bridges. Thus, we have sufficient evidence to prove the specific cases of $\TMAX{\varphi_{j_1},\varphi_{j_2}(a_1)}$ and $\TMAX{\varphi_{j_1},\varphi_{j_2}(a_1)}$ correspond to minimal and maximal matchings on $G$, respectively, although we will prove this in greater generality. We instead give the following lemmas.

\begin{lemma}
    \label{lem:tmax_k_funds}
For $G$ of type $(k,n)$, any $\TTMAX$ can be expressed with exactly $k$ fundamental height functions with consecutive relative translations $a_i = 1$ for all $i$. 
\end{lemma}

\begin{proof}
Recall $\TTMAX = \TMAX{\varphi_{j_1}, \varphi_{j_2}(a_1), \ldots, \varphi_{j_p}(a_{p-1})}$ for any arbitrary ordered collection of $p$ fundamental height functions with consecutive relative translations $a_i$. We first show that $\TTMAX$ can be expressed with consecutive relative translations $a_i' = 1$ for all $i$. Observe that for any $\varphi_{j_i}$ we have \[\max\left\{\varphi_{j_i},(\varphi_{j_i+1}-n), (\varphi_{j_i +2}-2n), \ldots, (\varphi_{j_i + \ell}-\ell n)\right\} = \varphi_{j_i}.\] In particular choosing $\ell = a_i - 1$ for each $i$, we have \[\TTMAX = \tmax{\varphi_{j_1}, \ldots, \varphi_{j_1+(a_1-1)},\varphi_{j_2},\ldots, \varphi_{j_p},\varphi_{j_p+1},\ldots, \varphi_{j_p+(a_p -1)}}\] where it remains that the relative translations between $\varphi_{j_i}$ and $\varphi_{j_{i+1}}$ is exactly $a_i$ for all $i$, but this $\TTMAX$ is now expressed with all consecutive relative translation $a_i' = 1$. In this case, we reindex the fundamental height functions as follows \begin{align*}
    \TTMAX &= \tmax{\varphi_{j_1}, \ldots,\varphi_{j_2},\ldots, \varphi_{j_p}, \ldots, \varphi_{j_p+(a_p -1)}} \\
    &= \tmax{\varphi_{j_1}, \varphi_{j_2'},\ldots, \varphi_{j_p'}}.
\end{align*} Now if $p' > k$, say $p' - k = r$, then for all $1 \le i \le r$ we note $(\varphi_{j_{k+i}} - (k+i-1)n)(f) < \TTMAX(f)$ for all $f \in F(G)$, so no such fundamental height $\varphi_{j_{k+i}}$ contributes to $\TTMAX$. Thus, we now assume $p' \le k$, but from our initial argument showing that $\TTMAX$ can be expressed with consecutive relative translations $a_i' =1$ for all $i$, we can extend our $\TTMAX$ over $p'$ fundamental height functions to a $\TTMAX$ over exactly $k$ fundamental height functions and our proof is complete. 
\end{proof}

\begin{lemma}
    \label{lem:tmin_k_funds}
For $G$ of type $(k,n)$, any $\TTMIN$ can be expressed with exactly $k$ fundamental height functions with consecutive relative translations $a_i = 1$ for all $i$. 
\end{lemma}

\begin{proof}
Proof of this statement is analogous to the proof of \Cref{lem:tmax_k_funds} with the change that we can add or remove the least indexed fundamental height functions appearing in $\TTMIN$ to achieve the same result. 
\end{proof}

We fix the notation $\Tmax = \tmax{\varphi_{j_1},\ldots, \varphi_{j_k}}$ and $\Tmin = \tmin{\varphi_{j_1},\ldots, \varphi_{j_k}}$ and study the particular case where consecutive relative translations $a_i =1$ and each $\Tmax$ and $\Tmin$ is taken over \textit{exactly} $k$ fundamental heights. We now proceed to main results of this subsection.

\begin{definition}
\label{def:fund_height_region}
Given $\Tmax$, the \textbf{fundamental height region},  $F_{j_i} \subseteq F(G)$, is the collection of all faces for which $\varphi_{j_i}$ uniquely contributes to $\Tmax$, i.e. $(\varphi_{j_i}-(i-1)n)(f) > (\varphi_{j_\ell} - (\ell-1) n)(f)$ for all $f \in F_{j_i}$ and $\ell \ne i$.

Analogously, given $\Tmin$, the \textbf{fundamental height region},  $F_{j_i} \subseteq F(G)$, is the collection of all faces for which $\varphi_{j_i}$ uniquely contributes to $\Tmin$, i.e. $(\varphi_{j_i} - (i-1)n)(f) < (\varphi_{j_\ell} - (\ell-1) n)(f)$ for all $f \in F_{j_i}$ and $\ell \ne i$.

For a fundamental height region $F_{j_i}$, we denote its corresponding (unoriented) boundary path by $P_{j_i} \subseteq E(G)$.
\end{definition}

The motivation for considering such $F_{j_i}$ and $P_{j_i}$ is that for $\Tmax$ and $\Tmin$ we can distinguish the regions of $G$ for which some fundamental height function in $\Tmax$ or $\Tmin$  uniquely contributes to $\Tmax$ or $\Tmin$, respectively. Then the case study between \Cref{def:land_bridge_2} and \Cref{lem:tmax_k_funds} is a special case of the following propositions which yield that $\Tmax$ and $\Tmin$ are tight height functions.

\begin{proposition}
\label{prop:tmax_decomp_prop} 
Let $j_1 < \cdots < j_k$ and $\Tmax = \tmax{\varphi_{j_1},\ldots, \varphi_{j_k}}$. Then $\Tmax$ is a tight height function. Further, $F_{j_i}$ occurs to the right of all $\mathcal{L}(\varphi_{j_i},\varphi_{j_m})$ for all $m > i$ and to the left of all $\mathcal{L}(\varphi_{j_\ell},\varphi_{j_i})$ for all $\ell < i$. Moreover, in orienting $P_{j_i}$ clockwise about $F_{j_i}$, the matching corresponding to $\Tmax$ contains all black-to-white edges on $P_{j_i}$ for all $i$.
\end{proposition}

\begin{proof}
We begin in showing that the fundamental height regions on $G$ are cut out by the land bridges as claimed. Fix any $\varphi_{j_i}$ in $\Tmax$ and fix the ordering on indices $1 \le \ell < i <  m \le k$. For each $m$, we consider $\mathcal{L}(\varphi_{j_i},\varphi_{j_m})$ with respect to the relative translation between $\varphi_{j_i}$ and $\varphi_{j_m}$ in $\Tmax$. Then $\varphi_{j_i}$ may only uniquely contribute to $\Tmax$ on faces $f \in F(G)$ which lie to the right of $\mathcal{L}(\varphi_{j_i},\varphi_{j_m})$. Similarly, for each $\ell$, we consider $\mathcal{L}(\varphi_{j_\ell},\varphi_{j_i})$  with respect to the relative translation between $\varphi_{j_\ell}$ and $\varphi_{j_i}$ in $\Tmax$. Then $\varphi_{j_i}$ may only uniquely contribute to $\Tmax$ on faces $f \in F(G)$ which lie to the left of $\mathcal{L}(\varphi_{j_\ell},\varphi_{j_i})$. Ranging over all such $\ell$ and $m$ and intersecting these faces gives that $\varphi_{j_i}$ contributes uniquely to $\Tmax$ on the region of $G$ cut out by land bridges with respect $\varphi_{j_i}$ and all other fundamental heights in $\Tmax$. Thus, the $F_{j_i}$ are cut out by land bridges as claimed.

We now show that $\Tmax$ is a tight height function. Notice that since each $F_{j_i}$ is cut out by land bridges with respect to $\varphi_{j_i}$ and all other fundamental heights in $\Tmax$, then $\varphi_{j_i}$ contributes non-uniquely to the height $\Tmax(g)$ for all $g \in F(G)$ which are separated from $F_{j_i}$ by $P_{j_i}$. Let $\{g_i\}$ be the set of all such faces separated from $F_{j_i}$ by $P_{j_i}$. Then since $\varphi_{j_i}$ contributes to $\Tmax$ on $F_{j_i} \cup \{g_i\}$, then $\Tmax$ is a tight height function on this region. As this holds for all $i \in [k]$, then for each $i$ the collection $\{g_i\}$ glues together $F_{j_i}$ with other $F_{j_\ell}$ whenever $\{g_i\} \cap \{g_\ell\} \ne \emptyset$. Thus, $\Tmax$ is built from piecing together fundamental height functions and so $\Tmax$ is also a tight height function. 

Lastly, we can study the matching corresponding to $\Tmax$ since $\Tmax$ is a height function. Note that from the study of $F_{j_i}$ by land bridges with respect to $\varphi_{j_i}$ and all other fundamental heights in $\Tmax$, then also matched edges along $P_{j_i}$ are determined by matched edges coming from right- and left-borders of these land bridges. In particular, for each $m > i$, $F_{j_i}$ lies to the right of each $\mathcal{L}(\varphi_{j_i},\varphi_{j_m})$ and the matching along $P_{j_i}$ contains a subset of matched edges along the right-borders of all such land bridges. These matched edges are all black-to-white edges with respect to the orientation of these land bridges. For each $\ell < i$, $F_{j_i}$ lies to the left of each $\mathcal{L}(\varphi_{j_\ell},\varphi_{j_i})$ and the matching along $P_{j_i}$ contains a subset of matched edges along the left-borders of all such land bridges. These matched edges are all white-to-black edges  with respect to the orientation of these land bridges. Orienting $P_{j_i}$ clockwise about $F_{j_i}$ is consistent with the orientations of $\mathcal{L}(\varphi_{j_i},\varphi_{j_m})$ for all $m$ and opposite the orientations of $\mathcal{L}(\varphi_{j_\ell},\varphi_{j_i})$ for all $\ell$. Thus, the matching corresponding to $\Tmax$ contains all black-to-white edges of $P_{j_i}$ in this orientation. 
\end{proof}

With respect to $\Tmin$'s, we have the following statement, which holds from a symmetric argument.

\begin{proposition}
\label{prop:tmin_decomp_prop}
Let $j_1 < \cdots < j_k$ and $\Tmin = \tmin{\varphi_{j_1},\ldots, \varphi_{j_k}}$. Then $\Tmin$ is a tight height function. Further, $F_{j_i}$ occurs to the left of all $\mathcal{L}(\varphi_{j_i},\varphi_{j_m})$ for all $m > i$ and to the right of all $\mathcal{L}(\varphi_{j_\ell},\varphi_{j_i})$ for all $\ell < i$. Moreover, in orienting $P_{j_i}$ clockwise about $F_{j_i}$, the matching corresponding to $\Tmin$ contains all white-to-black edges on $P_{j_i}$ for all $i$.
\end{proposition}

\begin{figure}[h]
    \centering
    \includegraphics[width=0.45\linewidth]{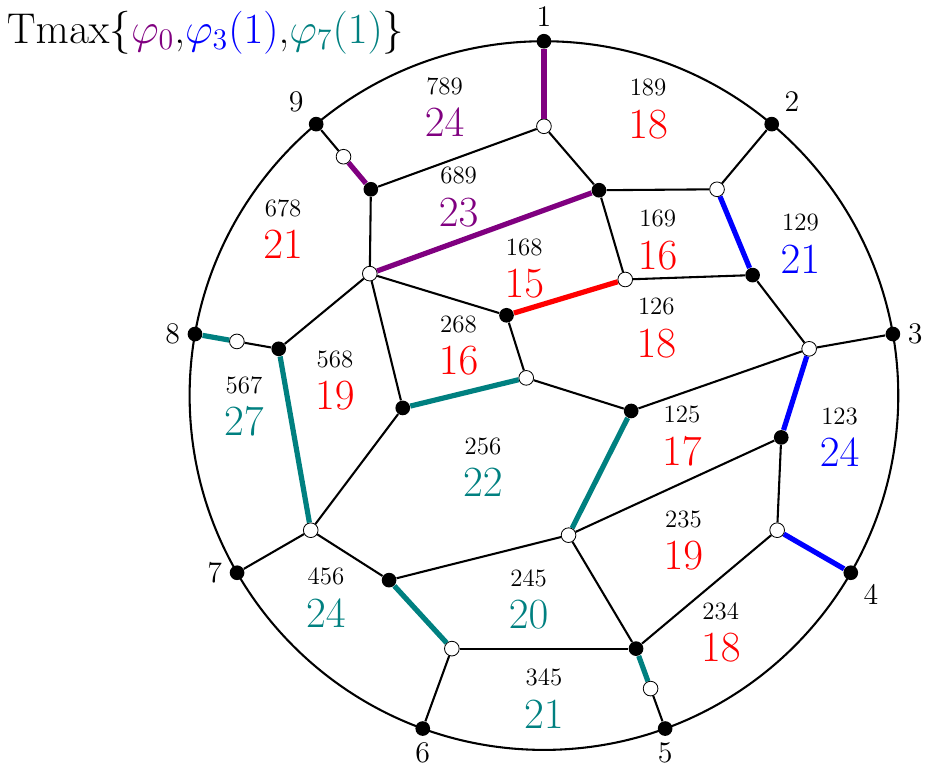}
    \caption{The matching $\TMAX{\varphi_0,\varphi_3(1),\varphi_7(1)} = \tmax{\varphi_0,\varphi_3,\varphi_7}$ on a source-labeled plabic graph corresponding to the top-cell of $\mathrm{Gr}(3,9)$. The fundamental height regions $F_0$, $F_3$, and $F_7$ are designated by {\color{violet} violet}, {\color{blue} blue}, and {\color{teal} teal}, in the top center, center right, and bottom left, respectively. In particular, the {\color{red} red} edge in this matching does not belong to any $F_j$.}
    \label{fig:fundheightbug}
\end{figure}

\begin{remark}
    \label{rem:decomp_internal_edges}
When applying \Cref{prop:tmax_decomp_prop} and \Cref{prop:tmin_decomp_prop}, it is possible to obtain matched edges that do not belong to any fundamental height region.  See \Cref{fig:fundheightbug} as an example.
\end{remark}

We conclude this section with the main result that $\Tmax$ and $\Tmin$ correspond to extremal matchings.

\begin{theorem}
\label{thm:tmax_tmin_are_extremal}
The height functions $\Tmax$ and $\Tmin$ correspond to extremal matchings, namely, $\Tmax$ corresponds to $M_-$ and $\Tmin$ corresponds to $M_+$. 
\end{theorem}

\begin{proof}
Begin with $\Tmax$. From \Cref{prop:tmax_decomp_prop} and \Cref{prop:fund_is_tight_unique}, the matching corresponding to $\Tmax$ on each $F_{j_i}$ is unique and cannot be swiveled. Moreover, if any $g \in F(G)$ is such that no edges of $g$ belong to some $P_{j_i}$ and $g \not \in F_{j_i}$ for all $i$, that is, $g$ is not directly adjacent to or included in any $F_{j_i}$, then $\Tmax(g)$ is contributed to non-uniquely by several $\varphi_{j_i}$ and so the matching corresponding to $\Tmax$ may not be swiveled at $g$. 

So suppose the matching corresponding to $\Tmax$ can be swiveled at some $f \in F(G)$, that is, this matching contains all alternating edges of $f$ with respect to the clockwise orientation of edges about $f$. At least one edge, say $e$, of $f$ belongs to some $P_{j_i}$. From \Cref{prop:tmax_decomp_prop}, the matching of $\Tmax$ contains all black-to-white edges of $P_{j_i}$ when orienting $P_{j_i}$ clockwise about $F_{j_i}$. In orienting edges of $f$ clockwise about $f$, then the edges of $f$ which belong to $P_{j_i}$ are opposite the orientation of $P_{j_i}$. Now either $e$ is included or excluded in the matching corresponding to $\Tmax$. If $e$ is included in this matching, then $e$ is white-to-black with respect to the orientation about $f$ and the matching of $\Tmax$ contains all white-to-black edges of $f$. If $e$ is excluded from this matching, then $e$ is black-to-white with respect to the orientation about $f$ and the matching of $\Tmax$ again contains all white-to-black edges of $f$. In particular, independent of whether $e$ is an included or excluded edge of this matching, swiveling this matching at $f$ is an upswivel. Therefore, the matching corresponding to $\Tmax$ is the minimal matching, $M_-$. 

Now consider $\Tmin$. From \Cref{prop:fund_is_tight_unique} and \Cref{prop:tmin_decomp_prop}, the matching corresponding to $\Tmin$ cannot be swiveled at any face $f' \in F_{j_i}$ for all $i$ or at any face $g \in F(G)$ such that $g$ is not directly adjacent to any $F_{j_i}$ for all $i$. Suppose the matching corresponding to $\Tmin$ can be swiveled at some $f \in F(G)$, that is, this matching contains all alternating edges of $f$ with respect to the clockwise orientation of edges about $f$. Again, at least one edge, say $e$, of $f$ belongs to some $P_{j_i}$ and the clockwise orientation of edges of $P_{j_i}$ about $F_{j_i}$ is opposite the orientation of edges clockwise about $f$. Thus, independent of whether $e$ is included or excluded from the matching corresponding to $\Tmin$, this matching contains all black-to-white edges clockwise about $f$ and swiveling at $f$ is a downswivel. Therefore, the matching corresponding to $\Tmin$ is the maximal matching, $M_+$. 
\end{proof}

\begin{figure}[!h]
    \centering
    \includegraphics[width=0.45\linewidth]{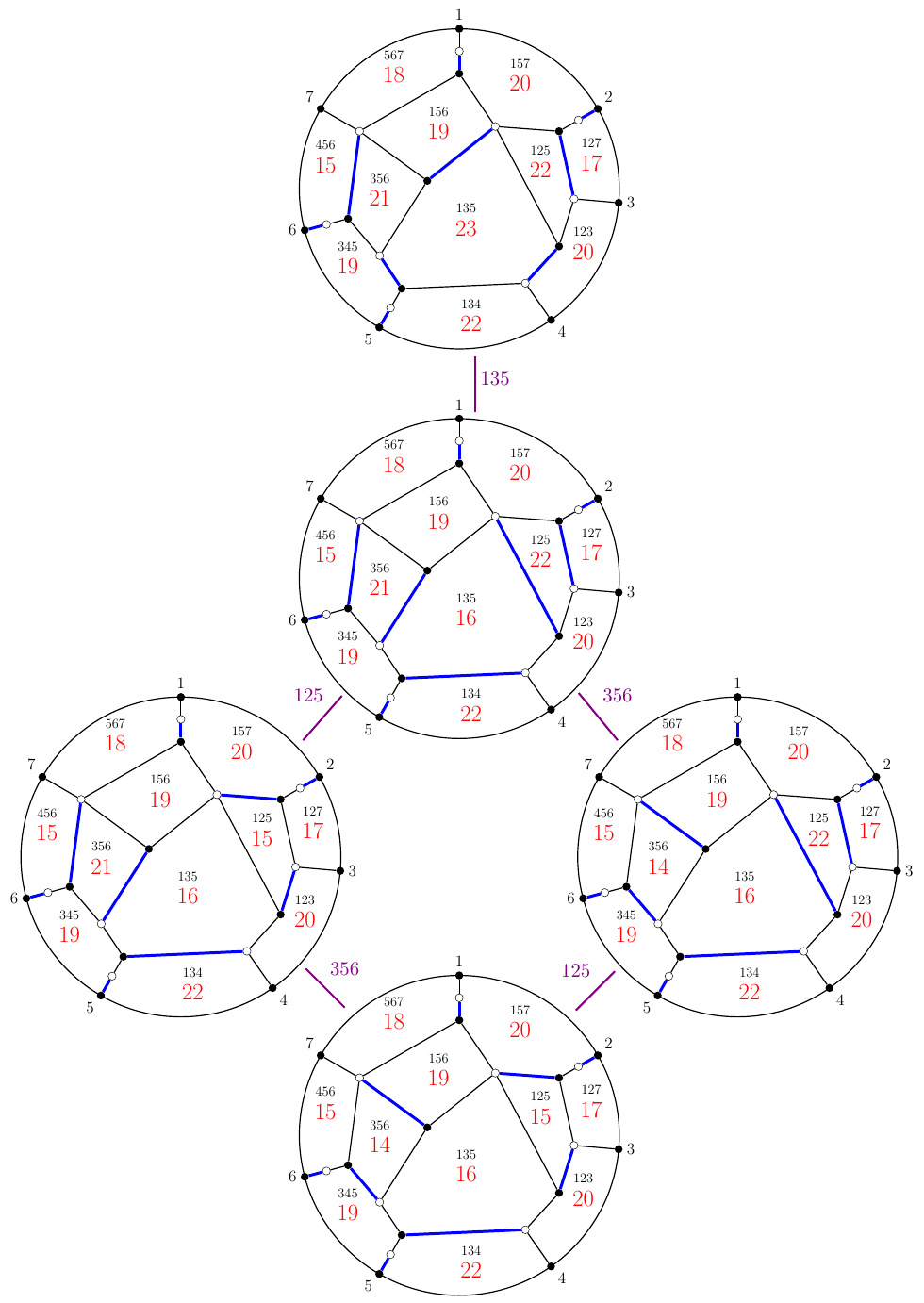} \hspace{.3em} \includegraphics[width=0.45\linewidth]{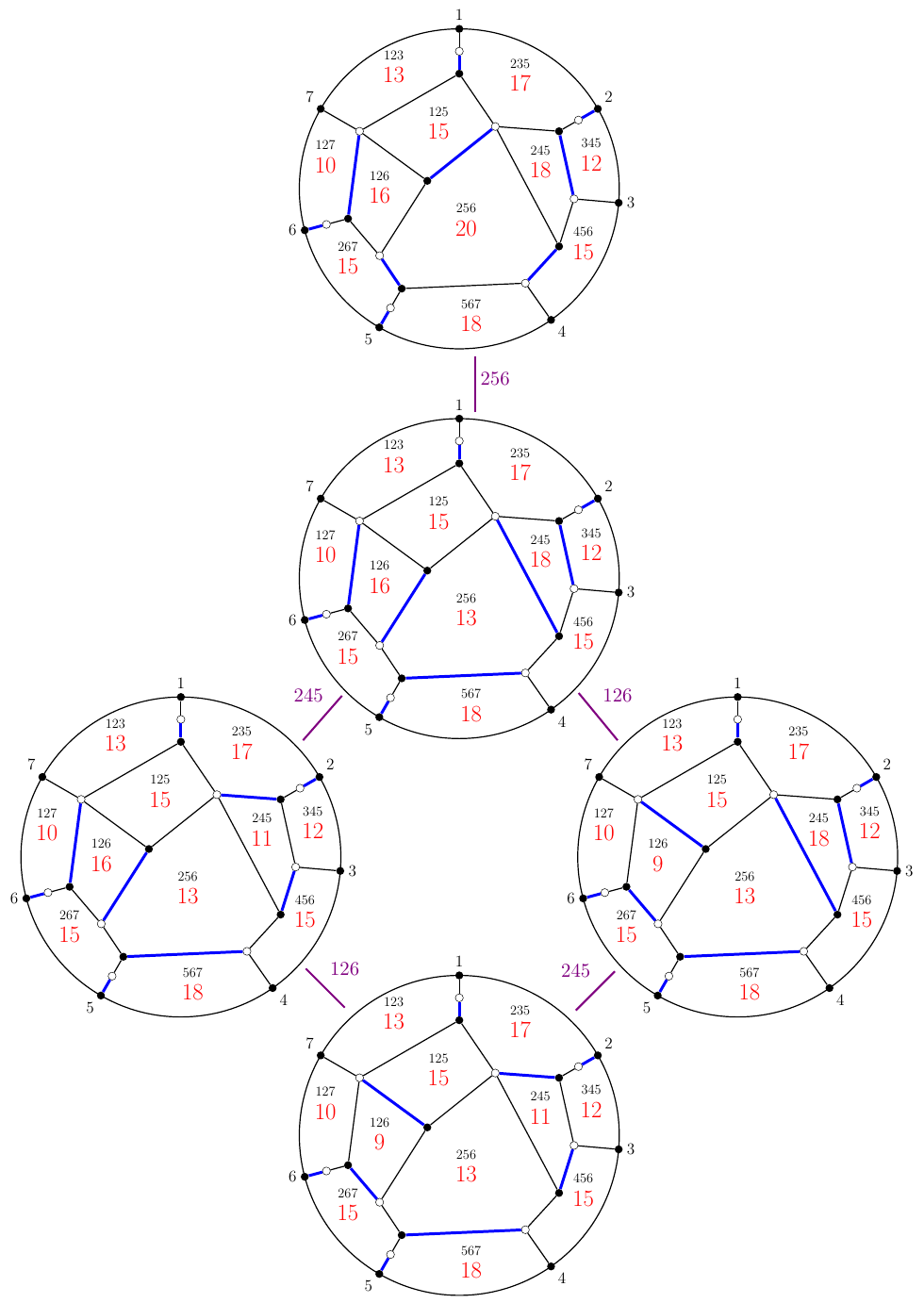}
    \caption{(Left): $M_-$ from $\TMAX{\varphi_1,\varphi_4(1)} = \tmax{\varphi_1,\varphi_4,\varphi_5}$ on source-labeled $G_B$ and the remaining poset; (Right):$M_+$ from $\TMIN{\varphi_2,\varphi_6(2)} = \tmin{\varphi_2,\varphi_5,\varphi_6}$ on target-labeled $G_B$ and the remaining poset. The labels of swiveled faces are included as labels on edges indicating cover relations.}
    \label{fig:256_poset_by_heights}
\end{figure}

\begin{example}
\label{ex:256_posets_by_heights}
We obtain the lattice $\mathcal{M}_{256}$ on $G_B$ of \Cref{fig:256_poset} by first constructing either $M_-$ as $\TMAX{\varphi_1,\varphi_4(1)} = \tmax{\varphi_1,\varphi_4,\varphi_5}$ in source-labeling or $M_+$ as $\TMIN{\varphi_2,\varphi_6(2)} = \tmin{\varphi_2,\varphi_5,\varphi_6}$ in target-labeling, and then applying swivels as in \Cref{prop:swivel_action_2} to obtain the remaining height functions and matchings of this lattice. These identical lattices are displayed in \Cref{fig:256_poset_by_heights}
\end{example}

\subsection{Boundary conditions for $\TTMAX$ and $\TTMIN$}
\label{ssec:extreme_match}

In this section, we let $G$ parameterize a positroid cell $\Pi_{\mathcal{P}}$ equipped with Grassmann necklace $\gneck$, reverse Grassmann necklace $\rgneck$, and decorated (trip) permutation $\tilde{\pi}$. We require a choice of face-labeling convention in order to obtain results regarding boundary conditions for matchings corresponding to $\TTMAX$ and $\TTMIN$. In particular, we associate source-labeling with $\TTMAX$ and target-labeling with $\TTMIN$. Our main results, \Cref{thm:main} and \Cref{thm:main_restate}, yield that $\TTMAX$ gives the minimal matching on $G$ with boundary condition read from $\gneck$ and that $\TTMIN$ gives the maximal matching on $G$ with boundary condition read from $\rgneck$. We conclude this section by showing that both pairings of $\TTMAX$ with source-labeling and $\TTMIN$ with target-labeling give all possible boundary conditions of extremal matchings on $G$, i.e. all $I \in \mathcal{P}$ arise as some $\TTMAX$ with respect to source-labeling and as some $\TTMIN$ with respect to target-labeling.

\begin{proposition}
\label{prop:source_fund_bdy}
In source-labeling, $\varphi_j$ has boundary condition $\gneckset{j+1}$. 
\end{proposition}

\begin{proof}
For all $i \in [n]$, let $e_i$ denote the unique edge incident to the boundary vertex $i$. Note for any white lollipop at $\ell$, then $e_\ell$ is included in any matching on $G$ and for any black lollipop at $\ell$, then $e_\ell$ is excluded from any matchings on $G$. Let $L = \{\ell \, | \, \ell \text{ is a lollipop of $G$}\}$ be the collection of all lollipops. Ignoring lollipops, we assume without loss of generality that $\tilde{\pi}^{-1}(i) \ne i$ for all $i$. We note that $e_i$ is included in the matching corresponding to $\varphi_j$ if and only if $\varphi_j(\rgneckset{i-1}) > \varphi_j(\rgneckset{i})$, that is, if and only if $\varphi_j$ is decreasing as we move clockwise about adjacent boundary faces of $G$. Note that $\varphi_j$ induces the $(j+1)$-ordering on $[n]$. In particular, $\rgneckset{j}$ is $(j+1)$-maximal and a boundary face label of source-labeled $G$, hence $\varphi_j(\rgneckset{j}) > \varphi_j(\rgneckset{j+1})$ and $e_{j+1}$ is included in our matching. With respect to $\tilde{\pi}^{-1}$, we observe $j+1 <_{j+1} \tilde{\pi}^{-1}(j+1)$ and so $j+1$ is an anti-exceedance of $\tilde{\pi}$, that is, $j+1 \in \gneckset{j+1}$. Traversing boundary faces of $G$ clockwise from $\rgneckset{j}$ corresponds to increasing in index $j$. From definition, we have that $\rgneckset{j+i}$ is $(j+i+1)$-maximal for all $i$, so in particular, $\rgneckset{j+i} >_{j+i+1} \rgneckset{j+i+1}$ for all $i$. Now taking $1 \le i \le n-2$, we claim that $\rgneckset{j+i} >_{j+i+1} \rgneckset{j+i+1}$ implies $\rgneckset{j+i}>_{j+1} \rgneckset{j+i+1}$ if and only if $j+i+1 <_{j+1} \tilde{\pi}^{-1}(j+i+1)$. In the forward direction, note that this also implies we observe $\rgneckset{j+i} >_{j+i+1 - r} \rgneckset{j+i+1}$ for all $0 \le r \le i$. One may check that as $r$ increases, then the $(j+i+1 -r)$-smallest elements of $\gneckset{j+i}$ must also be included in $\rgneckset{j+i+1}$. Rephrasing this in terms of our trip permutation, this means $j+i+1 <_{j+i+1-r} \tilde{\pi}^{-1}(j+i+1)$ for each $r \le i$. In particular, we have $j+i+1 <_{j+1} \tilde{\pi}^{-1}(j+i+1)$ as claimed. In the reverse direction, the face labels $\rgneckset{j+i}$ and $\rgneckset{j+i+1}$ share $k-1$ entries in common except $\tilde{\pi}^{-1}(j + i+1) \in \rgneckset{j+i}$ and $j+i+1 \in \rgneckset{j+i+1}$. Then from definition we have $\rgneckset{j+i}>_{j+i+1} \rgneckset{j+i+1}$ and given the $k-1$ common face label entries as well as $j+i+1 <_{j+1} \tilde{\pi}^{-1}(j+i+1)$, then also $\rgneckset{j+i} >_{j+1} \rgneckset{j+i+1}$ as desired. In context of our proposition, the left side of this equivalence yields that $e_{j+i+1}$ is included in the matching corresponding to $\varphi_j$ and the right side yields that $j+i+1$ is an anti-exceedance of $\tilde{\pi}$. Therefore, the boundary condition of the matching of $\varphi_j$ in source-labeling is $\gneckset{j+1}$
\end{proof}

By the analogous argument and applying the construction of $\rgneckset{j}$ as the exceedance set of $\tilde{\pi}$ under $(j+1)$-ordering, we also obtain the analogous boundary condition of matchings corresponding to $\varphi_j$ on target-labeled $G$. 

\begin{proposition}
\label{prop:target_fund_bdy}
In target-labeling, $\varphi_j$ has boundary condition $\rgneckset{j}$. 
\end{proposition}

In order to apply the boundary data of matchings we can obtain from fundamental height regions we develop the following notation.

\begin{definition}
\label{def:source_ess}
Fix a minimal matching $M_-$ which results from a $\TTMAX$ on source-labeled $G$. An ordered collection $S = \{\varphi_{j_1}, \ldots, \varphi_{j_p}\}$ of fundamental height functions and consecutive relative translations of type $\vec{a}$ is \textbf{source-essential} if $p$ is the least integer for which $\TMAX{\varphi_{j_1},\varphi_{j_2}(a_1),\ldots, \varphi_{j_p}(a_{p-1})}$ corresponds to $M_-$.
\end{definition}

\begin{definition}
\label{def:target_ess}
Fix a maximal matching $M_+$ which results from a $\TTMIN$ on target-labeled $G$. An ordered collection $T = \{\varphi_{j_1}, \ldots, \varphi_{j_p}\}$ of fundamental height functions and consecutive relative translations of type $\vec{a}$ is \textbf{target-essential} if $p$ is the least integer for which $\TMIN{\varphi_{j_1},\varphi_{j_2}(a_1),\ldots, \varphi_{j_p}(a_{p-1})}$ corresponds to $M_+$.
\end{definition}

From \Cref{def:tmax_tmin} and inverting the procedure given in the proof of \Cref{lem:tmax_k_funds}, any $\Tmax$ can be expressed as a compressed translated maximum over a source-essential set of fundamental height functions. By a slight abuse of notation, for any source-essential set of fundamental heights with consecutive relative translations of type $\vec{a}$, we have

\[\TMAX{\varphi_{j_1}, \varphi_{j_2}(a_1),\ldots, \varphi_{j_p}(a_{p-1})} = \tmax{\underbrace{\varphi_{j_1}, \ldots, \varphi_{j_1}}_{\text{$a_1$ times}}, \underbrace{\varphi_{j_2}, \ldots, \varphi_{j_2}}_{\text{$a_2$ times}}, \ldots, \underbrace{\varphi_{j_p}, \ldots,\varphi_{j_p}}_{\text{$a_p$ times}}}\] where the $\Tmax$ on the right has consecutive relative translations $a_i' = 1$ for all $i$. Similarly, inverting the analogous procedure outlined in the proof of \Cref{lem:tmin_k_funds}, any $\Tmin$ can be expressed as a compressed translated minimum over a target-essential set of fundamental height functions. Then also by a slight abuse of notation, for any target-essential set of fundamental heights with consecutive relative translations of type $\vec{a}$, we have
\[\TMIN{\varphi_{j_1}, \varphi_{j_2}(a_1),\ldots, \varphi_{j_p}(a_{p-1})} = \tmin{\underbrace{\varphi_{j_1}, \ldots, \varphi_{j_1}}_{\text{$a_p$ times}}, \underbrace{\varphi_{j_2}, \ldots, \varphi_{j_2}}_{\text{$a_1$ times}}, \ldots, \underbrace{\varphi_{j_p}, \ldots,\varphi_{j_p}}_{\text{$a_{p-1}$ times}}}\] where the $\Tmin$ on the right has consecutive relative translations $a_i' = 1$ for all $i$. These source- and target-essential collections for expressing $\Tmax$ as some $\TTMAX$ and $\Tmin$ as some $\TTMIN$ give boundary conditions for their corresponding extremal matchings as demonstrated in \Cref{thm:main} and \ref{thm:main_restate}. 

We recall that each term $\gneckset{i}$, resp. $\rgneckset{i-1}$, in the Grassmann necklace $\gneck$, resp. reverse Grassmann necklace $\rgneck$, comes with the $i$-ordering on $[n]$. Following the convention mentioned in \Cref{rem:GNwrittenorder}, for $\gneckset{i} =  \{i_1, i_2, \ldots, i_k\}$ with $i_1 <_i i_2 <_i \cdots <_i i_k$, we notate the truncated Grassmann necklace term as \[\gneckset{i}^\ell :=\{i_1, i_2, \ldots, i_\ell\}, \] i.e. the $\ell$ smallest elements, where $\ell \le k$. Similarly, for $\rgneckset{i} = \{i_1, i_2, \ldots, i_k\}$ with $i_1 >_{i+1} i_2 >_{i+1} \cdots >_{i+1} i_k$, we notate the truncated reverse Grassmann necklace term as \[\rgneckset{i}^\ell: \{i_1, i_2,\ldots, i_\ell\},\] i.e. the $\ell$ largest elements, where $\ell \le k$. We now present our main theorems. 

\begin{theorem}
\label{thm:main}
For source-labeled $G$ and the source-essential set $S = \{\varphi_{j_1},\varphi_{j_2},\ldots, \varphi_{j_p}\}$ with consecutive relative translations of type $\vec{a}$, the matching corresponding to $\TMAX{\varphi_{j_1},\varphi_{j_2}(a_1), \ldots, \varphi_{j_p}(a_{p-1})}$ is $M_-$ and has the boundary condition $\gneckset{j_1 + 1}^{a_1} \cup \gneckset{j_2 + 1}^{a_2} \cup \cdots \cup \gneckset{j_p+1}^{a_p}$. Here we recall the value of $a_p$ from \Cref{def:rel_trans}.  
\end{theorem}

\begin{proof}
Let $\TTMAX = \TMAX{\varphi_{j_1},\varphi_{j_2}(a_1), \ldots, \varphi_{j_p}(a_{p-1})}$ and let $M_-$ denote the corresponding matching. By \Cref{prop:tmax_decomp_prop}, we first consider the boundary conditions contributed on each fundamental height region $F_{j_i}$. In orienting $P_{j_i}$ clockwise about $F_{j_i}$, $M_-$ contains all black-to-white edges of $P_{j_i}$. In particular, the boundary vertex where $P_{j_i}$ originates contributes a boundary condition for $M_-$ on $F_{j_i}$ and we claim that this boundary condition is the $a_i$th term of $\gneckset{j_i +1}$. Following \Cref{prop:tmax_decomp_prop}, we have $F_{j_{i+1}}$ lies clockwise from $F_{j_i}$ and is separated from $F_{j_i}$ by $\mathcal{L}(\varphi_{j_i},\varphi_{j_{i+1}})$ for all $i \mod p$. With respect to the consecutive relative translation $a_i$ of $\varphi_{j_i}$ and $\varphi_{j_{i+1}}$ in $\Tmax$ for $i \ne p$ (or $\SumBlank{1 \le i < p}a_i$ if $i = p$), then any boundary face $g \in \mathcal{L}(\varphi_{j_i},\varphi_{j_{i+1}})$ satisfies that the face label of $g$ contains exactly $a_i$ elements from $[j_{i+1}]\setminus[j_i]$ for $i \ne p$ (or exactly $\SumBlank{1 \le i < p}a_i$ elements from $[j_p]\setminus[j_1]$ if $i = p$). Moreover, any such boundary face $g \in  \mathcal{L}(\varphi_{j_i},\varphi_{j_{i+1}})$ is not in any other $\mathcal{L}(\varphi_{j_i},\varphi_{j_m})$ for $m > i+1$ since the face label of $g$ contains exactly $a_i$ elements from $[j_{i+1}]\setminus[j_i]$. In particular, if $g \in \mathcal{L}(\varphi_{j_i},\varphi_{j_m})$ for $m > i+1$, then the face label of $g$ must contain exactly $\SumBlank{i+1 \le \ell < m} a_\ell$ entries from $[j_m]\setminus[j_{i+1}]$, so $g \in \mathcal{L}(\varphi_{j_{i+1}},\varphi_{j_m})$, too, which contradicts the source-essentiality of $S$. Applying a similar analysis yields that any lollipop of $G$ occurs strictly within some $F_{j_\ell}$, and so any boundary face $g \in  \mathcal{L}(\varphi_{j_i},\varphi_{j_{i+1}})$ is the unique boundary face separating $F_{j_i}$ and $F_{j_{i+1}}$ for all $i \mod p$. 

Now suppose $i \ne p$. From \Cref{prop:source_fund_bdy} and its proof, we have that $F_{j_i}$ contributes a boundary condition to $M_-$ whenever traversing the boundary faces of $G$ in $F_{j_i}$ clockwise about $G$, we have $\varphi_{j_i}(f) > \varphi_{j_i}(f')$, i.e. $f >_{j_i+1} f'$, where $f'$ is clockwise from $f$. Now recall that the boundary face $g \in  \mathcal{L}(\varphi_{j_i},\varphi_{j_{i+1}})$ has a face label which contains $a_i$ entries from $[j_{i+1}]\setminus[j_i]$ and each entry of $[j_{i+1}]\setminus[j_i]$ are at least $j_i+1$ in the $(j_i+1)$-ordering. Since $\varphi_{j_i}$ induces this $(j_i+1)$-ordering, this means within $F_{j_i}$ to $g$, there have been exactly $a_i-1$ instances for which $f >{j_i +1} f'$ with $f'$ clockwise from $f$ and for $f'' \in F_{j_i}$ separated from $g$ by $P_{j_i}$ is the $a_i$th instance for which $f'' >_{j_i+1} g$. Each such instance contributes a boundary condition, and since $\varphi_{j_i}$ on source-labeled $G$ has the boundary condition $\gneckset{j_i+1}$, then $F_{j_i}$ contributes the boundary condition $\gneckset{j_i+1}^{a_i}$ for each $i \ne p$. Now if $i = p$, by a similar argument, albeit using the complement of $[j_p]\setminus[j_1]$ in $[n]$, we observe that the remaining $a_p$ boundary conditions of $M_-$ which occur on $F_{j_p}$ are exactly given by $\gneckset{j_p+1}^{a_p}$ and we have reached our desired boundary condition of $M_-$.
\end{proof}

\begin{theorem}
\label{thm:main_restate}
For target-labeled $G$ and the target-essential set $T = \{\varphi_{j_1},\varphi_{j_2},\ldots, \varphi_{j_p}\}$ with consecutive relative translations of type $\vec{a}$, the matching corresponding to $\TMIN{\varphi_{j_1},\varphi_{j_2}(a_1), \ldots, \varphi_{j_p}(a_{p-1})}$ is $M_+$ and has the boundary condition $\rgneckset{j_1}^{a_p} \cup \rgneckset{j_2}^{a_1} \cup \cdots \cup \rgneckset{j_p}^{a_{p-1}}$. Here we recall the value of $a_p$ from \Cref{def:rel_trans}. 
\end{theorem}

We now show that any minimal matching on source-labeled $G$ is obtained by taking some $\Tmax$ of fundamental height functions and also that any maximal matching on target-labeled $G$ is obtained by taking some $\Tmin$ of fundamental height functions. In particular, \Cref{thm:boundary_meas_map} yields that $\mathcal{P}$ is realized as the collection of all $k$-subset boundary conditions of matchings on $G$. Further, \Cref{thm:gneck_positroid_biject} gives that $\mathcal{P}$ is uniquely determined by both $\gneck$ and $\rgneck$. From pairing these results, we give the following lemma. 

\begin{lemma}
    \label{lem:all_matchings}
The minimal matching with any boundary condition $J \in \mathcal{P}$ (i.e. for which $\Delta_J(A) > 0$ for all $A \in \Pi_\mathcal{P}$) on source-labeled $G$ is obtained as some $\TTMAX$ of fundamental height functions. Analogously, the maximal matching with boundary condition $J \in \mathcal{P}$ (i.e. for which $\Delta_J(A) > 0$ for all $A \in \Pi_\mathcal{P}$) on target-labeled $G$ is obtained as some $\TTMIN$ of fundamental height functions.
\end{lemma}

\begin{proof}
    Fix $J \in \mathcal{P}$ and write $J = \{j_1, j_2, \ldots, j_k\}$. Then for all $1 \le \ell \le k$ we have that $j_\ell \le_{j_\ell} i$ for all $i \in [n]$ so in particular $\gneckset{j_\ell}^1 = j_\ell$ for all $\ell$. Then the height function $\tmax{\varphi_{j_1-1},\varphi_{j_2 -1}, \ldots, \varphi_{j_k -1}}$ yields the minimal matching on source-labeled $G$ with boundary condition $J$, albeit this collection of fundamental heights not being source-essential. We can obtain the corresponding source-essential collection of fundamental heights by inverting the procedure in the proof of \Cref{lem:tmax_k_funds} and thereby the corresponding $\TTMAX$. Analogously, for all $1 \le \ell \le k$ we have that $j_\ell \ge_{j_\ell +1} i$ for all $i \in [n]$ so in particular $\rgneckset{j_\ell}^1 = j_\ell$ for all $\ell$. Then the height function $\tmin{\varphi_{j_1}, \varphi_{j_2},\ldots, \varphi_{j_k}}$ yields the maximal matching on target-labeled $G$ with boundary condition $J$, albeit this collection of fundamental heights not being target-essential. Similarly, we can obtain the corresponding target-essential collection of fundamental heights by inverting the procedure outlined in the proof of \Cref{lem:tmin_k_funds} and thereby the corresponding $\TTMIN$. 
\end{proof}

\begin{table}[ht]
    \centering
    {\setlength{\extrarowheight}{3pt}
    \begin{tabular}{|c|c|c|||c|c|c|}
    \hline
 $\tmax{\varphi_a,\varphi_b,\varphi_c}$   & \multicolumn{2}{c|||}{Boundary Condition} &  $\tmax{\varphi_a,\varphi_b,\varphi_c}$ & \multicolumn{2}{c|}{Boundary Condition} \\ 
\cline{2-3}  \cline{5-6}
$\{a,b,c\} \sim \{S(\vec{a})\}$  & On $\cev{F}(G)$ & On $\vec{F}(G)$ & $\{a,b,c\} \sim \{S(\vec{a})\}$ & On $\cev{F}(G)$ & On $\vec{F}(G)$ \\
       \hline
    $\{0,1,2\} \sim \{0\}$ &  123 & 567 &  $\{1,2,6\} \sim \{1,6(2)\}$ & 237 & 457 \\
     \hline
     $\{0,1,3\} \sim \{0,3(2)\}$ &  124 & 156 &  $\{1,3,4\} \sim \{1,3(1)\}$ & 245 & 125 \\
     \hline
     $\{0,1,4\} \sim \{0,4(2)\}$ &  125 & 156 &  $\{1,3,5\} \sim \{1,3(1),5(1)\}$ & 246 & 135 \\
     \hline
     $\{0,1,5\} \sim \{0,5(2)\}$ &  126 & 356 &  $\{1,3,6\} \sim \{1,3(1),6(1)\}$ & 247 & 145 \\
     \hline
     $\{0,1,6\} \sim \{6\}$ &  127 & 456 &  $\{1,4,5\} \sim \{1,4(1)\}$ & 256 & 135 \\
     \hline
     $\{0,2,3\}\sim \{0,2(1)\}$ &  134 & 157 &  $\{1,4,6\} \sim \{1,4(1),6(1)\}$ & 257 & 145 \\
     \hline
     $\{0,2,4\} \sim \{0,2(1),4(1)\}$ &  135 & 157 &  $\{1,5,6\} \sim \{5\}$ & 267 & 345 \\
     \hline
     $\{0,2,5\}\sim \{0,2(1),5(1)\}$ &  136 & 357 &  $\{2,3,4\} \sim \{2\}$ & 345 & 127 \\
     \hline
     $\{0,2,6\} \sim \{2,6(1)\}$ &  137 & 457 &  $\{2,3,5\} \sim \{2,5(2)\}$ & 346 & 137 \\
     \hline
     $\{0,3,4\} \sim \{0,3(1)\}$ &  145 & 125 &  $\{2,3,6\} \sim \{2,6(2)\}$ & 347 & 147 \\
     \hline
     $\{0,3,5\} \sim \{0,3(1),5(1)\}$ &  146 & 135 &  $\{2,4,5\} \sim \{2,4(1)\}$ & 356 & 135 \\
     \hline
     $\{0,3,6\} \sim \{3,6(1)\}$ &  147 & 145 &  $\{2,4,6\} \sim \{2,4(1),6(1)\}$ & 357 & 147 \\
     \hline
     $\{0,4,5\} \sim \{0,4(1)\}$ &  156 & 135 & $\{2,5,6\} \sim \{2,5(1)\}$ & 367 & 347 \\
     \hline
     $\{0,4,6\} \sim \{4,6(1)\}$ &  157 & 145 &  $\{3,4,5\} \sim \{3\}$ & 456 & 123 \\
     \hline
     $\{0,5,6\}\sim \{5\}$ &  267 & 345 &  $\{3,4,6\} \sim \{3,6(2)\}$ & 457 & 124 \\
     \hline
     $\{1,2,3\} \sim \{1\}$ &  235 & 157 &  $\{3,5,6\} \sim \{3,5(1)\}$ & 467 & 134 \\
     \hline
     $\{1,2,4\} \sim \{1\}$ &  235 & 157 &  $\{4,5,6\} \sim \{4\}$ & 567 & 134 \\
     \hline
     $\{1,2,5\} \sim \{1,5(2)\}$ &  236 & 357 &\multicolumn{3}{c|}{ } \\
     \hline
    \end{tabular}}
    \caption{The boundary conditions of matchings for all combinations $\{a,b,c\}\sim \{S(\vec{a})\}$ corresponding to $\tmax{\varphi_a,\varphi_b,\varphi_c} = \TMAX{S(\vec{a})}$ on both source- and target-labeled $G = G_B$, where $\{S(\vec{a})\}$ is the corresponding source-essential set with relative translation $\vec{a}$.}
    \label{tab:gb_missing_bdys}
\end{table}

\begin{table}[ht]
    \centering
    {\setlength{\extrarowheight}{3pt}
    \begin{tabular}{|c|c|c|||c|c|c|}
    \hline
 $\tmin{\varphi_a,\varphi_b,\varphi_c}$   & \multicolumn{2}{c|||}{Boundary Condition} & $\tmin{\varphi_a,\varphi_b,\varphi_c}$ & \multicolumn{2}{c|}{Boundary Condition} \\ 
 \cline{2-3} \cline{5-6}
 $\{a,b,c\} \sim \{T(\vec{a})\}$  & On $\vec{F}(G)$ & On $\cev{F}(G)$ &  $\{a,b,c\} \sim \{T(\vec{a})\}$  & On $\vec{F}(G)$ & On $\cev{F}(G)$ \\
       \hline
    $\{0,1,2\} \sim \{2\}$ &  127 & 345 &  $\{1,2,6\} \sim \{2,6(1)\}$ & 126 & 245 \\
     \hline
     $\{0,1,3\} \sim \{1,3(1)\}$ &  137 & 356 &  $\{1,3,4\} \sim \{4\}$ & 134 & 567 \\
     \hline
     $\{0,1,4\} \sim \{1,4(1)\}$ &  147 & 357 &  $\{1,3,5\} \sim \{1,3(1),5(1)\}$ & 135 & 256 \\
     \hline
     $\{0,1,5\} \sim \{1\}$ &  157 & 235 &  $\{1,3,6\} \sim \{1,3(1),6(1)\}$ & 136 & 256 \\
     \hline
     $\{0,1,6\} \sim \{1\}$ &  157 & 235 &  $\{1,4,5\} \sim \{1, 5(2)\}$ & 145 & 257 \\
     \hline
     $\{0,2,3\} \sim \{0,3(2)\}$ &  237 & 356 &  $\{1,4,6\} \sim \{1,4(1),6(1)\}$ & 146 & 257 \\
     \hline
     $\{0,2,4\} \sim \{0,2(1),4(1)\}$ &  247 & 357 &  $\{1,5,6\} \sim \{1,6(2)\}$ & 156 & 125 \\
     \hline
     $\{0,2,5\} \sim \{0,2(1),5(1)\}$ &  257 & 235 &  $\{2,3,4\} \sim \{4\}$ & 134 & 567 \\
     \hline
     $\{0,2,6\} \sim \{0,2(1)\}$ &  267 & 235 &  $\{2,3,5\} \sim \{3,5(1)\}$ & 235 & 256 \\
     \hline
     $\{0,3,4\} \sim \{0,4(2)\}$ &  347 & 367 &  $\{2,3,6\} \sim \{3,6(1)\}$ & 236 & 256 \\
     \hline
     $\{0,3,5\} \sim \{0,3(1),5(1)\}$ &  357 & 236 &  $\{2,4,5\} \sim \{2,5(2)\}$ & 245 & 257 \\
     \hline
     $\{0,3,6\} \sim \{0,3(1)\}$ &  367 & 236 &  $\{2,4,6\} \sim \{2,4(1),6(1)\}$ & 246 & 257 \\
     \hline
     $\{0,4,5\} \sim \{0,5(2)\}$ &  457 & 237 &  $\{2,5,6\} \sim \{2,6(2)\}$ & 256 & 125 \\
     \hline
     $\{0,4,6\} \sim \{0,4(1)\}$ &  467 & 237 &  $\{3,4,5\} \sim \{5\}$ & 345 & 267 \\
     \hline
     $\{0,5,6\} \sim \{0\}$ &  567 & 123 &  $\{3,4,6\} \sim \{4,6(1)\}$ & 346 & 256 \\
     \hline
     $\{1,2,3\} \sim \{3\}$ &  123 & 456 &  $\{3,5,6\} \sim \{3,6(2)\}$ & 356 & 126 \\
     \hline
     $\{1,2,4\} \sim \{2,4(1)\}$ &  124 & 457 &  $\{4,5,6\} \sim \{6\}$ & 456 & 127 \\
     \hline
     $\{1,2,5\} \sim \{2,5(1)\}$ &  125 & 245 &\multicolumn{3}{c|}{ } \\
     \hline
    \end{tabular}}
    \caption{The boundary conditions of matchings for all combinations $\{a,b,c\} \sim \{T(\vec{a})\}$ corresponding to $\tmin{\varphi_a,\varphi_b,\varphi_c} = \TMIN{T(\vec{a})}$ on both target- and source-labeled $G = G_B$, where $\{T(\vec{a})\}$ is the corresponding target-essential set with relative translation $\vec{a}$.}
    \label{tab:gb_good_bdys}
\end{table}

\begin{remark}
    \label{rem:mismatched_bdys}
In \Cref{lem:all_matchings}, we note that the pairings of source-labeled $G$ with $\TTMAX$ and target-labeled $G$ with $\TTMIN$ yield all possible boundary conditions $J \in \mathcal{P}$ of extremal matchings. From \Cref{thm:tmax_tmin_are_extremal}, one may consider the opposite pairings, i.e. $\TTMAX$ with target-labeled $G$ and $\TTMIN$ with source-labeled $G$, and still obtain extremal matchings. Furthermore, the boundary conditions for matchings coming from $\TTMAX$ on target-labeled $G$ (although not proven here) will arise as $j_i$-minimal entries from $\rgneckset{j_i}$ for ``target-essential" $\varphi_{j_i}$ appearing in $\TTMAX$, and dually, the boundary conditions for matchings coming from $\TTMIN$ on source-labeled $G$ will arise as $(j_i+1)$-maximal entries from $\gneckset{j_i+1}$ for ``source-essential" $\varphi_{j_i}$ appearing in $\TTMIN$.
\end{remark}

 We give in \Cref{tab:gb_missing_bdys} and \ref{tab:gb_good_bdys} the boundary conditions for $\TTMAX$ under source- and target-labelings of $G_B$ and the boundary conditions for $\TTMIN$ under the target- and source-labelings of $G_B$, respectively. One may observe that in either of the target-labeled column ($\vec{F}(G_B)$) of \Cref{tab:gb_missing_bdys} or the source-labeled column ($\cev{F}(G_B)$) of \Cref{tab:gb_good_bdys}, that over all 35 choices of $a,b,c$ inputs for $\Tmax$ and $\Tmin$, respectively, one yields 27 distinct boundary conditions (of the potential $|\mathcal{P}(B)| = 33$ boundary conditions). By constrast, either of the source-labeled column of \Cref{tab:gb_missing_bdys} or the target-labeled column of \Cref{tab:gb_good_bdys}, one yields all possible 33 boundary conditions of extremal matchings on $G_B$ with $\Tmax$ or $\Tmin$, respectively, as was asserted in \Cref{lem:all_matchings}.  

\begin{example}
\label{ex:running_boundary_conditions}
As in Figures \ref{fig:tmax14_example}, \ref{fig:tmin26_example} and\ref{fig:256_poset_by_heights}, see also \Cref{ex:positroid_example_matrix} and
\Cref{ex:256_posets_by_heights}, the height functions $\TMAX{\varphi_1,\varphi_4(1)} = \tmax{\varphi_1,\varphi_4,\varphi_5}$ in source-labeling and $\TMIN{\varphi_2,\varphi_6(2)} =  \tmin{\varphi_2,\varphi_5,\varphi_6}$ in target-labeling both yield almost perfect matchings with boundary condition $\{2,5,6\}$.

In particular, in the appropriate $a$-ordering (see \Cref{rem:GNwrittenorder}), the Grassmann necklace (resp. reverse Grassmann necklace) corresponding to $\mathcal{P}(B)$ is given by 
$\{123, 235, 345, 456, 567, 672, 712\}$ (resp. $\{175, 217, 321, 431, 543, 654, 765\}$).
Thus, applying \Cref{thm:main}, 
$\TMAX{\varphi_1,\varphi_4(1)}$ corresponds to ${\bf 2}35$ and ${\bf 56}7$ while applying \Cref{thm:main_restate}, we get that $\TMIN{\varphi_2,\varphi_6(2)}$ corresponds to ${\bf 2}17$ and ${\bf 65}4$.

Compare with the fifth (resp. thirteenth) line of the second column of \Cref{tab:gb_missing_bdys} (resp. \Cref{tab:gb_good_bdys}). 
\end{example}

\appendix 
\section{Applications of Extremal Matchings}
\label{appsec:Applications}

While the results of this paper regarding extremal matchings can be approached entirely through the combinatorics of lattices of dimer covers, our motivation for constructing these extremal matchings comes from the theory of cluster algebras.
With this motivation in mind, we  discuss the application of extremal matchings in computing \textit{twist maps} and $F$\textit{-polynomials} arising from the cluster algebra structures coming from positroid cells. We give a brief overview of the relevant cluster theory in \Cref{ssec:cluster_background} and give the main applicaitons and motiviation with regards to constructing extremal matchings in \Cref{ssec:twists}.

\subsection{Cluster algebra structures within the Grassmannian}
\label{ssec:cluster_background}

Cluster algebras were introduced in \cite{FZ1} as special subalgebras of the field of rational functions in $n$ variables. In particular, a cluster algebra, $\mathcal{A}$, is determined by choice of \textbf{initial seed}, i.e. a triple $(Q,\mathbb{X},\mathbb{Y})$ where $Q$ is a \textbf{quiver} on $n +m$ vertices, $\mathbb{X} = \{x_1,\ldots,x_n,x_{n+1}, \ldots, x_{n+m}\}$ are \textbf{cluster variables}, and $\mathbb{Y} = \{y_1,\ldots,y_n\}$ are \textbf{coefficients}. From the initial seed, one can perform \textbf{mutation} (in direction $i$), denoted $\mu_i$, to reach a new seed, which is an involution that exchanges a cluster variable $x_i$ with a new cluster variable $x_i'$, expresses each coefficient $y_j$ as some ratio of coefficients $y_j'$, and locally modifies arrows of $Q$ near the vertex $i$. Of note is that any cluster variable $x$ after any sequence of mutations can be expressed as a Laurent polynomial in variables and coefficients of the initial seed {\cite[Laurent Phenomenon]{FZ1}} and it was later shown that these Laurent expansions of cluster variables are totally positive \cite{LS}, \cite{GHKK}. 

The homogeneous coordinate ring of the affine cone over the Grassmannian, denoted $\coring$, has a cluster structure due to \cite{JS}, and similarly, the homogeneous coordinate ring of the affine cone over an open positroid variety within the Grassmannian, denoted $\poring$, has a cluster structure due to \cite{GL}. Consequently, given either of these coordinate rings, one can define an initial seed for a cluster algebra whose cluster consists entirely of Pl\"ucker coordinates and whose quiver is deducible from Pl\"ucker relations. More precisely, this uses the construction of a reduced plabic graph $G$, as was discussed in Section \ref{sec:plabic_graphs}, where the relevant Pl\"ucker coordinates are indexed by the face labels (using either the source- or target-labeling) and the quiver is defined as the (oriented) planar dual of the (bipartite) plabic graph. The corresponding quivers (with respect to the source- and target-labeling of $G_B$) are shown in \Cref{fig:dual_quivers}.

\begin{figure}[h!]
    \centering
    \includegraphics[width=0.3\linewidth]{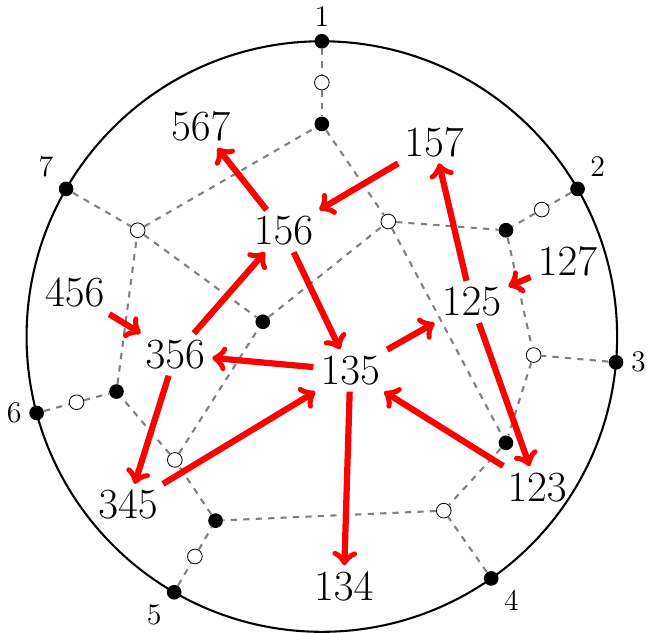} \hspace{2em} 
    \includegraphics[width=0.3\linewidth]{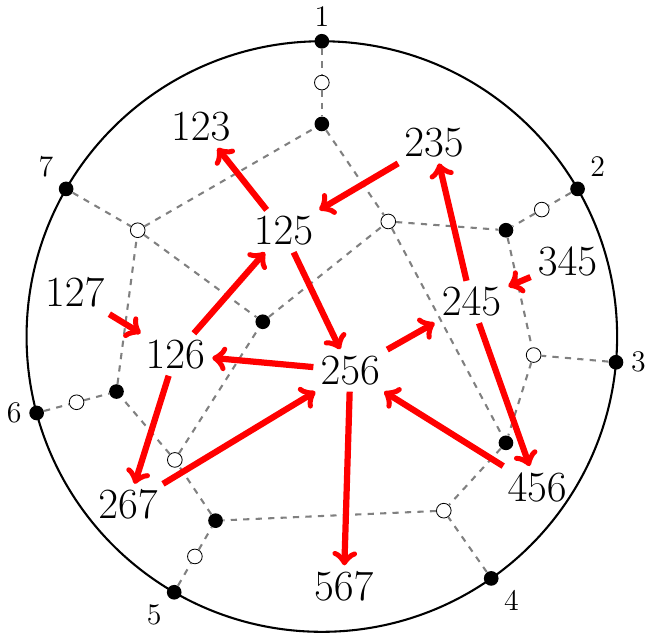}
    \caption{The corresponding quivers (in \textcolor{red}{red}) with respect to source- (left) and target-labeled (right) $G_B$.}
    \label{fig:dual_quivers}
\end{figure}

\subsection{Extremal matchings in twist maps and $F$-polynomials}
\label{ssec:twists}

The \textbf{twist map} (or twist automorphism), $\tau$, was originally introduced in \cite{BFZ} in the study of total positivity of unipotent cells in simple algebraic groups of simply-laced Lie type. It turns out $\tau$ is a cluster algebra automorphism. When applied to Pl\"ucker coordinates in the top-dimensional positroid cell, i.e. $\coring$, Marsh and Scott \cite{marsh2016twists} expressed $\tau$ in terms of a partition function on matchings where the weight of a matching $M$ comes from a face-weighting on $G$. Similarly, Muller and Speyer \cite{MSTwist} realized $\tau$ as partition function on matchings for all positroid cells with a slightly different face-weighting scheme. Many others have worked to define combinatorial interpretations of the twist automorphism in other (related) cluster algebra structures (see \cite{GLTwists}, \cite{BraidWeaves}). 

We direct our focus to the twist map of {\cite[Section 6]{MSTwist}}. Considering $\tau$ as some partition function over matchings on $G$, one obtains for the input cluster variable $\Delta_I$ the output cluster variable $\tau(\Delta_I)$ which is understood as the Laurent expansion for $\tau(\Delta_I)$ in terms of cluster variables from the initial seed corresponding to $G$. More precisely, to compute $\tau(\Delta_I)$, we obtain this Laurent expansion by summing over the set of all matchings with boundary condition $I$. An example of this computation, and simplification of $\tau(\Delta_I)$ into a single monomial cluster variable, is shown in \Cref{ex:twist_target}. Our result of constructing extremal matchings of any boundary condition yields an efficient means for constructing all matchings with a fixed boundary condition.

\begin{example}
    \label{ex:twist_target}
We consider the set of matchings $\mathcal{M}_{256}$ on target-labeled $G_B$ where $\Delta_{256}$ appears as a cluster variable. Following the enumeration of matchings in \Cref{ex:boundary_meas_map}, or also matchings from bottom-to-top in \Cref{fig:FpolyPoset}, the face weightings of the $M_i$ are:
\[\mathrm{wt}_f(M_1) = \frac{\Delta_{127}\Delta_{256}\Delta_{345}}{\Delta_{126}\Delta_{235}\Delta_{245}\Delta_{567}} \qquad \mathrm{wt}_f(M_2) = \frac{\Delta_{125}\Delta_{267}\Delta_{345}}{\Delta_{126}\Delta_{235}\Delta_{245}\Delta_{567}} \qquad \mathrm{wt}_f(M_3) = \frac{\Delta_{127}\Delta_{456}}{\Delta_{126}\Delta_{245}\Delta_{567}}\]
\[\mathrm{wt}_f(M_4) = \frac{\Delta_{125}\Delta_{267}\Delta_{456}}{\Delta_{126}\Delta_{245}\Delta_{256}\Delta_{567}} \qquad \mathrm{wt}_f(M_5) = \frac{1}{\Delta_{256}}.\]
Then the (right) twist of $\Delta_{256}$ is
\begin{align*}
 \tau(\Delta_{256}) &= \SumBlank{M_i \in \mathcal{D}_{256}}\mathrm{wt}_f(M_i) \\ 
 &=  \frac{\Delta_{127}\Delta_{256}\Delta_{345}}{\Delta_{126}\Delta_{235}\Delta_{245}\Delta_{567}} + \frac{\Delta_{125}\Delta_{267}\Delta_{345}}{\Delta_{126}\Delta_{235}\Delta_{245}\Delta_{567}} + \frac{\Delta_{127}\Delta_{456}}{\Delta_{126}\Delta_{245}\Delta_{567}} \\
  & \hspace{6em}+ \frac{\Delta_{125}\Delta_{267}\Delta_{456}}{\Delta_{126}\Delta_{245}\Delta_{256}\Delta_{567}} + \frac{1}{\Delta_{256}} \\
  &= \frac{1}{\Delta_{235}\Delta_{567}}\bigg(\frac{\Delta_{127}\Delta_{256}^2\Delta_{345}+\Delta_{125}\Delta_{256}\Delta_{267}\Delta_{345} + \Delta_{127}\Delta_{235}\Delta_{256}\Delta_{456}}{\Delta_{126}\Delta_{245}\Delta_{256}} \\
  & \hspace{6em}+ ~\hspace{6em} \frac{\Delta_{125}\Delta_{235}\Delta_{267}\Delta_{456}+\Delta_{126}\Delta_{235}\Delta_{245}\Delta_{567}}{\Delta_{126}\Delta_{245}\Delta_{256}} \bigg) \\
  &= \frac{\Delta_{357}}{\Delta_{235}\Delta_{567}},
 \end{align*}

 where $\Delta_{357}$ is another mutable cluster variable in this cluster algebra.
\end{example}

Another interesting application of extremal matchings to the cluster algebra theory is that of computing $\mathbf{F}$\textbf{-polynomials} for twisted cluster variables. $F$-polynomials were defined in \cite{FZ4} as the result of mutating clusters with principle coefficients and specializing all cluster variables to 1. In particular, $F$-polynomials track mutation patterns within cluster algebras and find important application in the \textbf{Separation of Additions Formula} of \cite{FZ4}. 

We give the process for computing $F$-polynomials in the following, with the result due to {\cite[Theorem 6.8]{MMSBV}}. In particular, to compute the $F$-polynomial corresponding to $\tau(\Delta_I)$, we consider the lattice of matchings $\mathcal{M}_I$. For each $M \in \mathcal{M}_I$, we assign a monomial $y_M = \ProdBlank{I}y^{m(I)}_I$, where the product is over face labels $I$ of $G$ which have been swiveled up from $M_-$ to reach $M$ in the lattice, and $m(I)$ is number of times the face labeled with $I$ has been swiveled up from $M_-$ to reach $M$. Then the corresponding $F$-polynomial is \[F_{\tau(\Delta_I)} = \SumBlank{M \in \mathcal{M}_I} y_M = \SumBlank{M \in \mathcal{M}_I}\ProdBlank{I}y^{m(I)}_I.\] 
In particular, $y_{M_-} = 1$, i.e. every $F$-polynomial has constant term 1, which is a well-known property in \cite{FZ4}. This additionally yields that the uses of \textit{minimal matching} and \textit{maximal matching} have algebraic significance as the minimal and maximal degree monomials, respectively, in the $F$-polynomial. We give an example computation of the $F$-polynomial $F_{\tau(\Delta_{256})}$ in \Cref{ex:256_F_poly}. 

\begin{figure}[h!]
    \centering
    \includegraphics[width=0.68\linewidth]{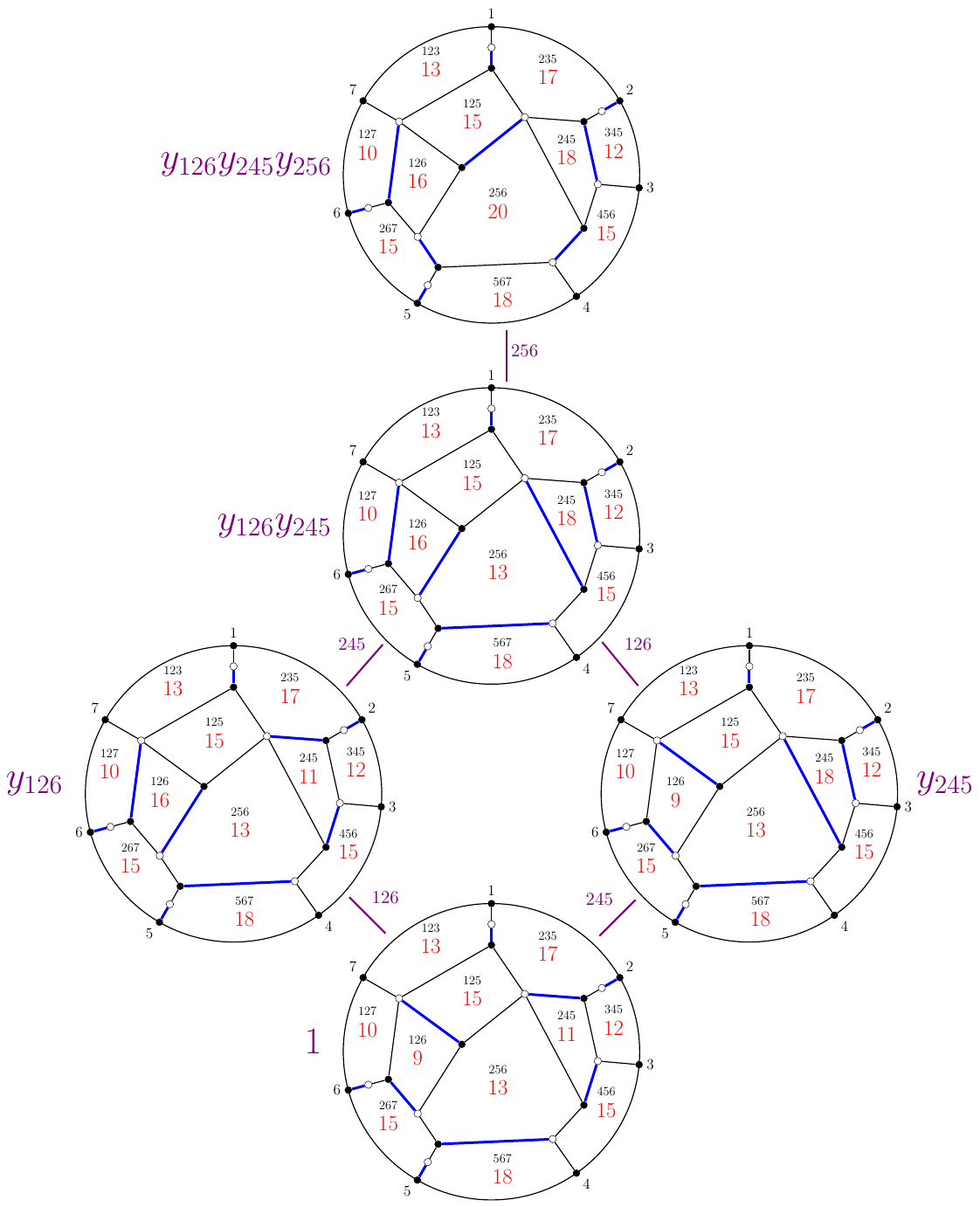}
    \caption{The lattice of $\mathcal{M}_{256}$ on target-labeled $G_B$ with the $y_M$ monomials next to each matching $M$ in this lattice.}
    \label{fig:FpolyPoset}
\end{figure}

\begin{example}
    \label{ex:256_F_poly}
We revisit the lattice of matchings $\mathcal{M}_{256}$ on target-labeled $G_B$ from \Cref{ex:256_posets_by_heights} shown in \Cref{fig:FpolyPoset} where the monomials $y_M$ next to each matching $M \in \mathcal{M}_I$. From \Cref{ex:twist_target}, we have $\tau(\Delta_{256}) = \frac{\Delta_{357}}{\Delta_{235}\Delta_{567}}$, where in the context of computing the $F$-polynomial, we can disregard the frozen variables appearing in the denominator, i.e. we compute $F_{\Delta_{357}}$. Then we observe \[F_{\tau(\Delta_{256})} = F_{\Delta_{357}} = 1 + y_{126} + y_{245} + y_{126}y_{245} + y_{126}y_{245}y_{256}.\]
In context, this shows that if we want to mutate our initial seed in such a way that $\Delta_{357}$ appears as a cluster variable in our mutated seed, then we could perform e.g. the mutation sequence $\mu_{256} \circ \mu_{245}\circ \mu_{126}$ read from right-to-left. 
\end{example}

\bibliographystyle{alpha}
\bibliography{bibliography}

\end{document}